\documentclass{amsart}

\usepackage[numbers]{natbib}
\usepackage{amssymb, amsmath}
\usepackage{graphicx, enumerate}
\usepackage[usenames]{color}
\newtheorem{ass}{Assumption}
\usepackage{xcolor}
\usepackage{cancel}
\usepackage{comment}
\usepackage{ulem}
\usepackage{hyperref}
\usepackage{subcaption}
\usepackage{colortbl}
\usepackage{placeins}

\newcommand\bR{\mathbb{R}}
\newcommand\cD{{\mathcal D}}
\renewcommand\Pr{\mathbb{P}}
\newcommand\bE{\mathbb{E}}
\newcommand{\Rd}{{\rm d}}
\newcommand{\wtd}{\widetilde}
\newcommand{\wht}{\widehat}
\newcommand{\II}[1]	{[\![#1]\!]}
\newcommand{\idc}[1]  { \mathbf{1}_{\left \lbrace #1 \right \rbrace}  }
\newcommand{\eT}{ ^{\mathsf{T}} }
\newcommand\eps{{\varepsilon}}

\newcommand{\hDev}{h}
\newcommand{\ts}{{\theta^*}}
\newcommand{\thN}{\hat{\theta}^N_t}

\newtheorem{theorem}{Theorem}[section]
\newtheorem{proposition}[theorem]{Proposition}
\newtheorem{corollary}[theorem]{Corollary}

\newtheorem{lemma}[theorem]{Lemma}
\newtheorem{remark}[theorem]{Remark}

\begin{document}
\title[LAN for neurons]{LAN property for the parameter of the jump rate in mean field interacting systems of neurons} 

\author[Aline \ Duarte]{Aline Duarte}
\address{A. Duarte: Departamento de Estatística, Instituto de Matem\'atica e Estat\'istica, Universidade de S\~ao Paulo, Rua do Mat\~ao 1010, 05508-090, Brazil}
\email{alineduarte@usp.br}

\author[Dasha \ Loukianova]{Dasha Loukianova} 
\address{D. Loukianova: Université Paris-Saclay, Université d’Evry Val d’Essonne, CNRS, LaMME, \ UMR 8071
	\ 91037 Evry, France} 
\email{dasha.loukianova@univ-evry.fr}

\author[Aur\'{e}lien \ Velleret]{Aur\'{e}lien Velleret} 
\address{A. Velleret: Université Paris-Saclay, Université d’Evry Val d’Essonne, CNRS, LaMME, \ UMR 8071
	\ 91037 Evry, France} 
\email{aurelien.velleret@nsup.org}

\begin{abstract}
	In the context of a large system of $N$ neurons interacting through spike events in a mean-field regime as $N\to \infty$, we characterize the estimation of a multidimensional parameter in the spiking rate, when the neural states are observed over a fixed time horizon. We first prove the local asymptotic normality (LAN) property and leverage classical theory to establish the asymptotic efficiency of the maximum likelihood estimator. While the theory of Ibragimov and Has'minskii yields strong results, up to global asymptotic minimax bound, its applicability appears currently limited to models without state resets at spike times. Following then H\"opfner’s classical approach, we nevertheless derive, in a general setting including neuron reset, the consistency, asymptotic normality and local asymptotic minimax optimality of the estimator.
	

\textbf{AMS Classification 2020:} 62F12, 60G55, 62C20 (primary); 
60K35, 62F99 (secondary)
\end{abstract}

\maketitle

\section{Introduction}

In modern neuro-mathematics, a substantial body of work focuses on large-scale limits of systems of interacting neurons, particularly in brain regions where neurons exhibit similar behavior and can be modeled through mean-field type interactions (see \cite{chevallier}, \cite{duartechevallier}, \cite{CTV20}, \cite{CTV21}, \cite{DGLP}, \cite{SusanneEva}, \cite{FL16}, \cite{SDG17}, \cite{RT16} and the references therein). 
In this framework, the large population limit  
yields a mesoscopic description of an individual neuron
via a limiting theoretical dynamics, 
referred to as the mean-field limit, 
which is typically more tractable 
than the corresponding large finite system.

In the present work, we investigate the insights that the mean-field limit can provide about the actual finite neuronal system from a statistical standpoint. More specifically, we focus on understanding how a given neuron responds to incoming stimuli through its spiking rate function, which characterizes its ability to react by generating sequences of spikes. 

As is standard in the literature, we model the spiking activity using a variant of nonlinear Hawkes processes (see \cite{Brillinger}, \cite{Cessac}, \cite{do}, \cite{ae}, as well as \cite{GLP} and references therein). More precisely, the model consists of $N$ interacting components (neurons) that evolve according to a deterministic flow between successive spikes of the system  and emit spikes at rates depending on their membrane potentials. This dynamics is described by the following system of stochastic differential equations driven by the system $(\pi^j),\; j\in\mathbb{N}^*$,  of i.i.d. Poisson Random Measures 
on $\bR_+\times\bR_+\times\bR$ with intensity $\Rd s\, \Rd z\, \nu(\Rd u)$, where $\nu$ is a probability measure on $\bR$ subject to moment conditions: 
\begin{multline*}
	\Rd X_t^{i, N,\theta } = b ( X_t^{i, N,\theta} ) \Rd t + \frac{1}{ {N} } \sum_{ j \neq i }^N \int_{\bR_+} \int_{\bR} u \idc{ z \le f_{\theta} ( X_{s-}^{j, N,\theta } ) }  \pi^j (\Rd t, \Rd z, \Rd u )
	\\ + \int_{\bR_+} \int_{\bR}  \varphi(X_{s-}^{i, N,\theta})\, \idc{ z \le f_{\theta} ( X_{s-}^{i, N,\theta } ) }  \pi^i (\Rd t, \Rd z, \Rd u )\,,
	\; \quad X_0^{i, N,\theta }{=  X_0^{i}}\; .
\end{multline*}

In this system $X_i^N(t)$ represents the membrane potential of neuron $i$ in a network of $N$ neurons;
a typical description of the flow $b$ is $b(x) = -\alpha x$ where 
$\alpha>0$ describes a leak of the potential;
the spikes of neuron $i$ are generated by the Poisson measure $\pi^{i}$ with conditional intensity $ f(X_i^N(t-))$. The function $\varphi$ describes the instantaneous loss of the potential of neuron $i$  when it emits a spike, and a typical example is $\varphi(x)=-x,$ which models the situation where
when a neuron $i$ emits a spike, its potential is reset (set to zero).
The spikes of neuron $j$ produce jumps of size $U/N$ in any neuron $i\neq j$, where the random height $U$ is generated following the distribution $\nu$, independently of all the past. 

Our goal is to estimate the unknown spiking rate function 
for a fixed membrane potential level, 
based on continuous-time observations 
of the neurons’ membrane potentials 
over a fixed time interval $[0, T]$, 
while the number of observed neurons increases. 
Thus, the asymptotic regime we consider corresponds to $N \to \infty$ 
with the observation time window kept fixed.

This work lies within the domain of statistical inference for particle systems and related McKean–Vlasov limits, a very active and rapidly developing research area in recent years (see \cite{HP}, \cite{AmNS}, \cite{AmP}, \cite{AmBPP}, \cite{AmHPP}, \cite{dMH23}, \cite{VCF}, among others). In most of these papers, Brownian or fractional Brownian noise is considered. Consequently, systems driven by jump processes represent an important and relatively less explored direction within this line of research.

In the recent paper \cite{DKLL}, we introduced a kernel estimator of Nadaraya-Watson type and discussed its asymptotic properties with the help of the deterministic dynamical system describing the mean field limit in the absence of the state resets at spike times, i.e. when $\varphi\equiv 0.$ We computed the minimax rate of convergence in an $L^2$-error loss over a range of H\"older classes and obtained the classical rate of convergence $N^{\frac{-2\beta}{2\beta+1}}$, where $\beta$ is the regularity of the unknown spiking rate function. 

In this paper, 
we assume that the spiking rate function  depends 
on some multidimensional parameter
and  show the Local Asymptotic Normality 
(LAN) property 
as the size of the system $N\to\infty$ 
with the rate $\sqrt N$ 
and the limiting Fisher information 
depending on the limit of the empirical measure of the system. 
We then study the asymptotic properties of the Maximum Likelihood Estimator (MLE). Under the LAN property, we leverage classical theory to establish the asymptotic efficiency of the maximum likelihood estimator. While the approaches of \cite{dMH23} and 
of Ibragimov and Has'minskii \cite {IH13} yield strong results, up to global asymptotic minimax bound, its applicability appears currently limited to models without state resets at spike times, i.e. $\varphi\equiv 0.$
Following H\"opfner's classical approach \cite{Ho14},  we nevertheless derive, in a general setting including neuron reset, the consistency, asymptotic normality and local asymptotic minimax optimality of the estimator.

The remainder of this paper is organized as follows. In Section~\ref{sec_results}, we introduce the mathematical framework for the finite system of interacting neurons (Section~\ref{sec_finite}) and its mean-field counterpart (Section~\ref{sec_limit}). Our main theoretical results, namely the LAN property and the asymptotic properties of the  MLE, are stated in Section~\ref{sec_main_results}. Section~\ref{sec_numerical_illustration} provides a  numerical study to clarify the scope of these theoretical results for finite systems, including simulations that illustrate the behavior of LAN quantities and the alignment of the MLE with the theoretical predictions derived from LAN. The proofs of our results are developed in Sections~\ref{sec_aux}--\ref{sec_MLE}: auxiliary lemmas are gathered in Section~\ref{sec_aux}, the proof of the LAN property is presented in Section~\ref{sec:proof_lan}, and the analysis of the MLE (consistency, asymptotic normality, and optimality) is carried out in Section~\ref{sec_MLE}. Finally, Appendix~\ref{app:simulation_details}  
provides technical implementation details
and further extends the numerical exploration with additional scenarios (varying system size and an alternative parameter set). 


\section{Model, assumptions and main results}
\label{sec_results}
\subsection{Finite system}
\label{sec_finite}

Let   $(\pi^j),\; j\in\mathbb{N}^*$ be a family 
of i.i.d. Poisson Random Measures 
on $\bR_+\times\bR_+\times\bR$ 
with the  intensity $\Rd s\, \Rd z\, \nu(\Rd u)$, 
where $\nu$ a probability measure on $\bR.$ 
We consider the following  model of interacting spiking  neurons:
\begin{multline}\label{eq:finite}
	\Rd X_t^{i, N,\theta } = b ( X_t^{i, N,\theta} ) \Rd t + \frac{1}{ {N} } \sum_{ j \neq i }^N \int_{\bR_+} \int_{\bR} u \idc{ z \le f_{\theta} ( X_{s-}^{j, N,\theta } ) }  \pi^j (\Rd t, \Rd z, \Rd u )
	\\ + \int_{\bR_+} \int_{\bR}  \varphi(X_{s-}^{i, N,\theta})\, \idc{ z \le f_{\theta} ( X_{s-}^{i, N,\theta } ) }  \pi^i (\Rd t, \Rd z, \Rd u )\,,
	\; \quad X_0^{i, N,\theta }{= X_0^{i}}.
\end{multline}
The jumping rate $f$ depends on a parameter $\theta\in \Theta,$ with $\Theta$ a compact subset of $\bR^d.$ We denote by $ \mu^{N, \theta}$ the empirical measure of the system \eqref{eq:finite},  and by $ \mu_s^{N, \theta}$ its projection onto the $s$-th coordinate.
\begin{equation}\label{eq:muN}
	\mu_s^{N, \theta}(\Rd x)
	= \frac1N\sum_{j = 1}^N \delta_{X^{j, N,\theta}_{s}}(\Rd x)\,.
\end{equation}

We impose the following assumptions on the system~\eqref{eq:finite}
\begin{ass}\label{ass:theta}
	The parameter set $\Theta$ is a compact subset of $\bR^d$ with non-empty interior $\mathring\Theta.$
\end{ass}

\begin{ass}\label{ass:finite}
	\begin{enumerate}[(i)]
		\item 
		The function $b: \bR \to \bR $ is  Lipschitz-continuous, 
		\item  The function $\varphi: \bR \to \bR$ is Lipschitz-continuous and bounded,
		
		\item For all $\theta \in \Theta$ the function $:x\in \bR\mapsto f_{\theta}(x)\in\bR_+$ is Lipschitz-continuous and bounded,
		
		\item
		The measure $\nu$ satisfies: $ \int_\bR |u| \nu (du ) < \infty  $  and  $ \nu ( \{ 0 \} ) = 0,$
		\item  The initial condition is i.i.d. distributed according to a square integrable distribution $\mu_0$.
	\end{enumerate}
\end{ass}
Under the above assumption, for each fixed $N,$  equation~\eqref{eq:finite} admits a unique strong solution (see \cite{andreis_mckeanvlasov_2018}). 
\begin{ass}\label{ass:moments}
		The measure $\nu$ satisfies: $ \int_\bR |u|^2 \nu (du ) < \infty  $  and  $ \nu ( \{ 0 \} ) = 0.$
\end{ass}
\subsection{Limit system}
\label{sec_limit}
As the system size $N$ tends to infinity, under Assumptions \eqref{ass:finite} and \eqref{ass:moments},
system \eqref{eq:finite} exhibits the propagation of chaos property \cite{andreis_mckeanvlasov_2018}: the system \eqref{eq:finite}, extended by setting  $\Rd X_t^{i, N,\theta }\equiv 0,\;\ i >N,$ converges in law in the  infinite product  Skorokhod space $\cD^{\otimes\infty}(\bR_+,\bR)$ to the infinite collection of  i.i.d. processes solving 
the following McKean-Vlasov SDE:
\begin{equation}\label{eq_limit}
	\Rd \bar X_t^{ i, \theta } = b ( \bar X_t^{i,\theta} ) \Rd t + 
	m\,\bar \mu^{\theta}_t[f_{\theta}] \Rd t
	+ \int_{\bR_+} \int_{\bR} \varphi(\bar X_{s-}^{i, \theta})\, \idc{ z \le f_{\theta} ( \bar X_{s-}^{i, \theta } ) }  \pi^i (\Rd t, \Rd z, \Rd u )\,,
	\; \quad \bar X_0^{i, \theta }{=  X_0^{i}}\;.
\end{equation}

Here we denoted $m=\int_{\bR}u\;\nu(du)$ and $\bar \mu^{\theta}$ the law on the Skorokhod space  of $(\bar X^{\theta })_{t\geq 0},$
whereas $\bar \mu^{\theta}_s$ denotes its projection on the $s$-coordinate. Moreover, under Assumptions \eqref{ass:finite} the system 
\eqref{eq_limit} admits a unique strong solution (see \cite{andreis_mckeanvlasov_2018}).

\begin{remark}\label{re:reset} In the case $\phi(x)=-x,$ assuming that the initial condition admits exponential moments, the existence and the uniqueness of solutions to systems \eqref{eq:finite} and \eqref{eq_limit}, as well as the propagation of chaos property, were established in \cite{Er22}. All our results remain valid in this setting.
\end{remark}

\subsection{Main results}
\label{sec_main_results}
We fix  $t>0$  and continuously  observe on $ [0,t]$  the trajectories of all the coordinates   $X^{i, N,\theta^*}_s;\; i\le N$, of the particle system \eqref {eq:finite} generated under 
the  parameter value  $\ts\in\mathring{\Theta}.$
We denote by  $\Omega_N$ the $N$ product Skorokhod space $\cD^{N}([0,t],\bR)$ 
equipped with the natural filtration $F_N:=(\mathcal F_s^N)_{s\in[0,t]}$ 
induced by the canonical mappings
$$X_s^{(N)}(\omega)=(X_s^{1, N}(\omega),\ldots,X_s^{N, N}(\omega))=\omega_s.$$
We write $X^{(N)}:=(X^{(i, N)}_{s})_{i\le N};\;  s\in [0,t],$ for a generic element of $\Omega^N$ and 
we denote $\P_N^{\theta}$ the law on $(\Omega_N,F_N)$ of the strong solution of \eqref{eq:finite}
and  $\bE_{\Pr_N^{\ts}}$ the expectation under this law.  
We study the sequence of statistical experiments 
$$({\mathcal E}^N)_{N\geq 1}=\left(\Omega_N,F_N,(\Pr_N^{\theta},\theta\in\Theta)\right)_{N\geq 1},$$
corresponding to the observation on $[0,t]$ 
of the solution of \eqref{eq:finite}.
Our first main result establishes that, for any  $\ts\in\mathring{\Theta},$  the Local Asymptotic Normality (LAN)  property holds for the sequence of experiments $({\mathcal E}^N)$  as $N\to\infty.$  To formulate this first result, we need to impose additional assumptions.
\begin{ass}\label{ass:dif}
	\begin{enumerate}[(i)]
		\item The function  $ \bR \ni x \mapsto f_{\theta}(x)\in\bR_+$ 
		satisfies uniform upper- and lower-bounds,  and its Lispchitz constant is bounded:
		
		$\sup_{\theta \in \Theta}\sup_{x\in \bR} f_{\theta} (x)< \infty$, $\inf_{\theta \in \Theta}\inf_{x\in \bR} f_{\theta} (x)>0$,
		$\sup_{\theta \in \Theta} \| f_\theta \|_{Lip}<\infty$.
		
		\item For all $x\in\mathbb{R}$, the function $\theta \mapsto f_{\theta}(x)$ is differentiable, and the gradient, 
		denoted $\dot f_{\theta} (x)\in \bR^d$,  satisfies a uniform bound:
		$\sup_{\theta \in \Theta}\sup_{x\in \bR} \|\dot f_{\theta} (x)\|< \infty$.

		%
	\end{enumerate}
\end{ass}

Under Assumption~\ref{ass:dif}, 
we denote
\begin{equation}\label{eq:const}
	\ell_0 = \big(\inf_{\theta, x} f_{\theta}(x)\big)^{-1},\;
	B_0 = \sup_{\theta, x} |f_{\theta}(x)|,\;
	B_1 = \sup_{\theta, x} \Big\|\frac{\dot f_{\theta}}{f_{\theta}}(x)\Big\|\,,
	B_L = \sup_{\theta} \| f_\theta \|_{Lip},\;
\end{equation}
which are all finite and well-defined quantities.

\begin{ass}\label{ass:deriv} There exists  $C>0$ such that for all $\theta, \theta'\in\Theta,$
	\begin{equation*}
		\left| \frac{f_{\theta}}{f_{\theta '}} (y) -1   - (\theta - \theta')\eT\cdot \frac{\dot{f}_{\theta'}}{f_{\theta'}} (y) \right|  \le C\, \| \theta - \theta' \|^2.
	\end{equation*}
\end{ass}
For any $\theta\in\Theta,$ $\ts\in\Theta,$ the likelihood ratio process ${L_t^{ \theta/\ts}}:=\dfrac{d\Pr_N^{\theta}}{d\Pr_N^{\ts}}$    for the generic trajectory $X^{(N)}\in\Omega_N$ is given by 
\begin{equation}\label{eq_def_L}
	{L_t^{ \theta/\ts}}\big ( X^{(N)}\big)  =\prod_{i= 1}^N \prod_{ n \geq 1 : T^i_n \le t } \frac{ f_{\theta  } }{ f_{\theta^*} } \big( X^{(i, N)}_{T^i_n- } \big) \cdot  \exp\Big[- \sum_{j= 1}^N \int_0^t ( f_{\theta  } - f_{\theta^*} ) \big (X^{(j, N)}_s\big) \Rd s \Big] ,
\end{equation}

where  $(T^i_n)_{n\ge 1}$
for $i\in \II{1, N}$
is the  ordered sequence of  jump times of $X^{(i, N)}.$ 
We omit in the following the dependency on the  trajectory $X^{(N)}$. 

For all $\ts\in\Theta,$ denote $I_t^{\ts}$ the {limiting Fisher information matrix} of the sequence of the experiments ${\mathcal E}^N$. It is defined by
\begin{equation}\label{eq_def_I}
	I_t^{\ts} =\int_0^t \bar \mu^{\ts}_s \Big[\frac{(\dot f_{\ts})^{\otimes 2}}{f_{\theta^*}}\Big] \Rd s 
	= \int_0^t \bar \mu^{\ts}_s \Big[\frac{\dot f_{\ts} \cdot [\dot f_{\ts}]\eT}{f_{\theta^*}}\Big] \Rd s\,,
\end{equation}
where $\cdot^{\otimes 2}$ denotes the outer product (tensor product) of the gradient vector.

Denote by ${\mathcal{M}}^d_+$ the set of symmetric and strictly positive definite matrices in $\bR^{d\times d}.$
\begin{ass}\label{as:fischer}
	For all $\theta^*\in\Theta,$ $I_t^{\theta^*}\in  {\mathcal{M}}^d_+.$ 
\end{ass}

\begin{remark}\label{rem:unif-nondegen}
	Assumptions \ref{ass:theta}, \ref {ass:dif}  and \ref {as:fischer} imply that the limiting Fischer information is also uniformly non-degenerated:
	$$\inf_{\ts\in\Theta}\det I_t^{\ts}>0.$$
	
\end{remark}
\begin{remark}An example of the rate function that is pertinent in the neuroscience and obviously satisfies Assumptions \ref{ass:dif}, \ref{ass:deriv} and \ref{as:fischer} is
	$$f_\theta(x)=\theta_1 \arctan x+\theta_2,\; x\geq 0,\; \theta=(\theta_1,\theta_2),\;   $$
	$\theta_1\in\Theta_1\subset [0, 1]$ and $\theta_2>\pi/2, \; \theta_2\in\Theta_2; $ $\Theta_1,\Theta_2$-compacts in $\bR_+.$
	
	Another example, where $\theta_0$ corresponds to the maximum intensity, 
	$\theta_1$ to the threshold value of potential,
	where $f_\theta$ has its inflexion point,
	and $\theta_2$ to the roughness of the transition of $f_\theta$ at $\theta_1$
	\begin{equation*}
		f_\theta(x) = \frac{\theta_0}{1 + \exp(-\theta_2 \cdot (x - \theta_1))},
		\quad \theta = (\theta_0, \theta_1, \theta_2),
	\end{equation*}
	is extensively studied in Section \ref{sec_numerical_illustration}.
\end{remark}
\
\begin{theorem}\label{th:LAN}
	Under Assumptions \eqref{ass:theta} to \eqref{as:fischer} the LAN property for the sequence of experiments $({\mathcal E}^N)_{N\geq 1}$ holds for any $\ts \in \Theta$ at rate $1/\sqrt{N}$ and with the Fisher information matrix $I_t^{\theta^*}$. Namely,  there exist for any $t>0$ a random sequence  $(\Delta_t^{N,\theta^*})_{N}$, of vectors with $\Delta_t^{N,\theta^*} \in \mathbb{R}^d$, and a random sequence 
	$(I_t^{N, \theta^*})_{N}$, of matrices with $I_t^{N, \theta^*} \in {\mathcal{M}}^d_+$,  such that for all $\hDev \in \mathbb{R}^d$:
	\begin{equation}\label{eq:LAN}
		\log \frac{L_t^{ \theta^*+\hDev/\sqrt{N}}}{L_t^{\ts}} = \hDev^{\mathsf{T}} \Delta_t^{N,\theta^*} - \frac{1}{2} \hDev^{\mathsf{T}} I_t^{N, \theta^*} \hDev +  o_{\Pr^{\ts}_N}(1),
	\end{equation}
	and the following convergences hold 
	when $N\to\infty$:
	\begin{equation*}
		\Delta_t^{N,\theta^*} \longrightarrow \mathcal N(0, I_t^{\theta^*}),\;\; \mbox{in law under }\;  \Pr^{\theta^*}_N;
	\end{equation*}
	and
	\begin{equation*}
		I_t^{N, \theta^*} \longrightarrow  I_t^{\theta^*},\;\; \mbox{in}\;\;   \Pr^{\theta^*}_N-\mbox{probability};
	\end{equation*}
	where $\mathcal N(0, \Sigma)$ denotes the $d$-dimensional Gaussian distribution with mean zero and covariance matrix $\Sigma$.  
\end{theorem}
LAN property implies that locally around a given point $\ts,$ the experiment ${\mathcal E}^N$
parametrized by $\theta=\ts+N^{-1/2}h$ behaves asymptotically like the simplest possible experiment, namely a Gaussian shift experiment.
This property  has powerful consequences 
for the  optimality  of  estimators 
through the celebrated Hajek convolution theorem \cite{Ha72}, which states that under LAN 
the limiting  law of  any regular estimator 
is the convolution of the normal law with variance $(I^{\ts})_t^{-1}$ 
and an independent error distribution.
In particular, the following local asymptotic minimax lower bound holds for arbitrary loss function $w: \bR^p \rightarrow [0,\infty)$ 
which are continuous, bounded and subconvex: 
\begin{corollary}\label{corol:hajek} For any sequence of estimators $(\tilde\theta_N)$  
	which estimation errors in $\ts$ are tight at rate $\sqrt N$ 
	and for any continuous, bounded and subconvex loss function 
	$w: \bR^p \rightarrow [0,\infty)$, it holds that
	\begin{equation*}
		\liminf_{C\to\infty}	\liminf_{N \to \infty}\sup_{\sqrt{N}|\theta'-\ts| \leq C}\bE_{\Pr^{\theta'}_N}\Big[w\Big(N^{1/2}(I_t^{\theta^*})^{1/2}\big(\tilde\theta^N_t - \theta'\big)\Big)\Big] 
		\geq  \frac{1}{(\sqrt{2\pi})^d}\int_{\bR^d} w(x)\,e^{-\frac{1}{2}\|x\|^2}\Rd x\,.
	\end{equation*}
	
\end{corollary}
This result is proved in \cite{Ho14}, $(7.12).$
We exploit now  the LAN property
to deduce the  results of asymptotic minimax optimality for the Maximum Likelihood Estimator (MLE), showing in particular that MLE achieves the lower bound of Corollary \ref{corol:hajek}, see Theorem \ref{thm_minimax} bellow.
Assume  that the trajectory $X^{(N)}$ is generated under an unknown parameter $\ts.$  The maximum likelihood estimator $\hat\theta^N_{t}$ of $\ts$  is given by
\[ \hat\theta^N_{t} = \mbox{argmax}_{\theta \in \Theta} L_t^{ \theta/\theta^*}(X^{(N)}). \]
To formulate our  results about the MLE we need to impose additional assumptions.

Recall the definition of the Fisher information $I_t^{\theta^*}$ in \eqref{eq_def_I}. The assumption that $I_t^\theta\in{\mathcal M}^d_+$ holds for any $\theta$ entails the local identifiability of the parameter.
We need moreover to impose an identifiability assumption implying that the mapping $\theta\mapsto\Pr_N^{\theta}$ is one-to-one,  which  we  express in terms of 
\begin{equation*}
	\mathcal I_t^{\theta/\theta'} = \int_0^t \bar\mu^{\theta'}_s\Big[\Big|\frac{f_\theta}{f_{\theta'}} - 1\Big|\Big] \Rd s\,.
\end{equation*}
\begin{ass}\label{ass_non_deg}
	$\mathcal I_t^{\theta/\theta'} >0$ for any $\theta \neq \theta'\in \Theta$.
\end{ass}

\begin{remark}\label{rem_non_deg}
	Assumption~\ref{ass_non_deg} can be interpreted in terms of the set $\mathfrak D^{\theta/\theta'}
	:= \{y\in \bR, f_{\theta}(y) \neq f_{\theta'}(y)\}\,.$
	We exploit in the proofs the fact that, for any $\theta\neq \theta'$, 
	$\big\{s\in [0, t];\, \bar\mu^{\theta'}_s\big[\mathfrak D^{\theta/\theta'}] > 0\big\}$ has positive Lebesgue measure,
	which is equivalent to $\mathcal I_t^{\theta/\theta'} >0$.
\end{remark}

Our  next result states the consistency of MLE. 
\begin{proposition}\label{prop_cons}
	Under Assumptions~\ref{ass:theta}, \ref{ass:finite},  \ref{ass:dif} and \ref{ass_non_deg},   
	$$\hat\theta^N_{t}\longrightarrow \ts\; \; \mbox{in}\; \; \Pr^{\ts}_N-\mbox{probability}.$$ 
\end{proposition}
This proposition is proved in Section \ref{sec:consistency}. 

We next show that MLE is asymptotically normal, and under LAN we show its  efficiency and local minimax optimality.
In order to formulate these results we need 
to require an additional  regularity condition on $\theta\mapsto f_\theta$.

\begin{ass}\label{ass_ddf}
	For all $x\in\mathbb{R}$, the function $\theta \mapsto f_{\theta}(x)$ is twice continuously differentiable, and the Hessian matrix, 
	denoted by $\ddot f_{\theta} (x)\in \mathcal{M}^d$,  satisfies a uniform bound:
	$\sup_{\theta \in \Theta}\sup_{x\in \bR} \|\ddot f_{\theta} (x)\|< \infty$.
\end{ass}

\begin{theorem}\label{thm_minimax}
	Under Assumptions~\ref{ass:theta} to
	\ref{ass_ddf},
	\begin{enumerate}[(i)]
		\item \textbf{Asymptotic normality:} For any $\ts \in \mathring{\Theta}$, under $\Pr^{\ts}_N,$ the following convergence in law holds
		\begin{equation*}
			\sqrt{N}\big(\hat\theta^N_t - \ts\big) \stackrel{{\mathcal L}}{\longrightarrow} \mathcal N\big(0, (I_t^{\theta^*})^{-1}\big)\,.
		\end{equation*}
		\item  \textbf{Local asymptotic minimax:} The following local asymptotic minimax optimality holds, for any $C>0$ and any loss functions $w$ that are continuous, bounded and subconvex
		\begin{equation*}
			\lim_{N \to \infty}\sup_{\sqrt{N}|\theta'-\ts| \leq C}\bE_{\Pr^{\theta'}_N}\Big[w\Big(N^{1/2}(I_t^{\theta^*})^{1/2}\big(\hat\theta^N_t - \theta'\big)\Big)\Big] 
			= \frac{1}{(\sqrt{2\pi})^d}\int_{\bR^d} w(x)\,e^{-\frac{1}{2}\|x\|^2}\Rd x\,.
		\end{equation*}
	\end{enumerate}
\end{theorem}

The last Theorem \ref{thm_optimality} is inspired by the paper \cite{dMH23} and their approach based on the results of \cite[Section III.1]{IH13}. In this Theorem  we consider a simplified model \eqref{eq:finite}, in the absence of large jumps, i.e. with $\varphi\equiv 0$
\begin{equation*}
	\Rd X_t^{i, N,\theta } = b ( X_t^{i, N,\theta} ) \Rd t + \frac{1}{ {N} } \sum_{ j \neq i }^N \int_{\bR_+} \int_{\bR} u \idc{ z \le f_{\theta} ( X_{s-}^{j, N,\theta } ) }  \pi^j (\Rd t, \Rd z, \Rd u );
	\; \quad X_0^{i, N,\theta }{=  X_0^{i}}\;
\end{equation*}
and its large population limit
\begin{equation*}
	\Rd \bar X_t^{ i, \theta } = b ( \bar X_t^{i,\theta} ) \Rd t 
	+ m\,\bar \mu^{\theta}_t[f_{\theta}] \Rd t;
	\; \quad \bar X_0^{i, \theta }{=  X_0^{i}}\;.
\end{equation*}

This approach is based on the following polynomial deviation inequality
in the Wasserstein 1 distance between the empirical measure $\mu^{N,\ts}$ 
of the finite system \eqref{eq:finite} and the law $\bar \mu^{\ts}$  of the limit system, see Lemma \ref{lem_moment_ctrl}, which hold for any $r>2$, for some $C>0$ that only depends on $r$  and on  the upper-bound on polynomial moments of $\bar \mu^{\ts}_{s}$ (up to $2r$)
\begin{equation*}
	\sup_{s\in [0, t]}
	\bE_{\Pr_N^{\ts}}[\mathcal W_1(\mu^{N, \ts}_s, \bar \mu^ {\ts}_s)^r]
	\le \frac{C}{N^{r/2}}.
\end{equation*}
For this result to hold, we need the following assumption.
\begin{ass}\label{ass_any_moment}
	\begin{enumerate}[(i)]
		\item For all $r\geq 1,$  $\bE[|X_0^1|^r]<\infty$.
		\item  The measure $\nu$ satisfies: $ \int_\bR |u|^r \nu (du ) < \infty  $ for all $r \geq 1 .$ 
	\end{enumerate}
\end{ass}
The idea behind  is to exploit a diagonal coupling with the same Poisson measure for the $j$-th particle
and control together $|X^{i, N, \ts}_s - \bar X^{i, N, \ts}_s|^r$ and $\mathcal W_1(\mu^{N, \ts}_s, \bar \mu_s^{N, \ts})^r$
by means of the Gr\"onwall Lemma. More precisely, 
the Wasserstein norm appears as the natural quantity to bound the
predictible part of the interaction term,
together with a correction term which we handle with concentration inequalities taken from \cite{FG15}, 
while the martingale terms are controlled
by the Burckholder-Davis-Gundy inequality.	
This last control fails to produce satisfying bounds
if $\varphi \neq 0$,
due to the dependency of the jump rate 
on the state of the particle.
We conjecture that such a strong result actually does not hold 
in this case where  $\varphi \neq 0$.

\begin{theorem}[Convergence and optimality of the MLE without large jumps]\label{thm_optimality}\hfill
	
	Under Assumptions~\ref{ass:theta} to 
	\ref{ass_non_deg} and \ref{ass_any_moment}, 
	the following asymptotic properties hold:
	\begin{enumerate}[(i)]
		
		\item \textbf{Local asymptotic minimax optimality:} For every polynomial loss function $w$, and any $\delta > 0$ sufficiently small,
		\begin{equation*}
			\limsup_{N \to \infty}\sup_{|\theta-\ts| \leq \delta}\bE_{\Pr^{\theta}}\Big[w\Big(N^{1/2}(I_t^{\theta^*})^{1/2}\big(\hat\theta^N_t - \theta\big)\Big)\Big] 
			\to \frac{1}{(\sqrt{2\pi})^d}\int_{\bR^d} w(x)\,e^{-\frac{1}{2}\|x\|^2}\Rd x
		\end{equation*}
		as $\delta \to 0$.
		
		\item \textbf{Global asymptotic minimax optimality:} For any non-empty open set $\Theta_0 \subset \mathring{\Theta}$ and every polynomial loss function $w$,
		\begin{equation*}
			\mathcal R_w^N(\hat\theta^N_t; \Theta_0) = \inf_{\tilde\theta_N} \mathcal R_w^N(\tilde\theta_N; \Theta_0)\,(1+o(1)),\quad\text{as } N \to \infty\,,
		\end{equation*}
		where the risk is defined
		as
		\begin{equation*}
			\mathcal R_w^N(\tilde\theta_N; \Theta_0) = \sup_{\theta \in \Theta_0}\bE_{\Pr^{\theta}}\Big[w\Big(N^{1/2}(I_t^{\theta})^{1/2}\big(\tilde\theta_N - \theta\big)\Big)\Big]\,.
		\end{equation*}
	\end{enumerate}
\end{theorem}
\begin{remark}
	The minimax bound in Theorem~\ref{thm_minimax}
	holds for a narrower constraint on $\theta'$
	and a stronger bound on $w$ as compared to the one in
	Theorem~\ref{thm_optimality}.(ii),
	in line with the conditions in \cite{Ho14} which we draw upon in the present subsection. 
\end{remark}

\begin{remark}
	Under the LAN property, we know from \cite[Section II.12.1]{IH13},
	originally due to H\'ajek 
	\cite{Ha72}, 
	that the r.h.s. of $(i)\; $
	is actually the minimax lower-bound corresponding to the loss function $w$ among any family of estimators.
\end{remark}

\section{Numerical illustration of LAN and MLE behavior}
\label{sec_numerical_illustration}

\subsection{Simulation setup}
\label{sec_simulation_setup}

Following Remark~\ref{re:reset}, Equation~\ref{eq:finite}
has been considered with $b(x) = -\alpha x$
corresponding to an exponential leakage at rate $\alpha$,
$\phi(x) = -x$ corresponding to the reset of the neuron state at spiking time,
$\nu$ the uniform distribution over some set $[\nu_1 -\nu_2, \nu_1 +\nu_2]$ and the following sigmoid function
for the spiking rate:

\begin{minipage}{0.5\textwidth}
	\begin{equation*}
		f_\theta(x) = \frac{\theta_0}{1 + \exp(-\theta_2 \cdot (x - \theta_1))}
		;\quad \theta = (\theta_0, \theta_1, \theta_2),
	\end{equation*}
	\vspace{.3cm}
	
	where $\theta_0$ corresponds to the maximum intensity, 
	$\theta_1$ to the threshold value of potential,
	where $f_\theta$ has its inflexion point,
	and $\theta_2$ to the roughness of the transition of $f_\theta$ at $\theta_1$.
\end{minipage}
\hspace{0.03\textwidth}
\begin{minipage}{0.45\textwidth}
	
	\centering
	\includegraphics[width=0.8\textwidth]{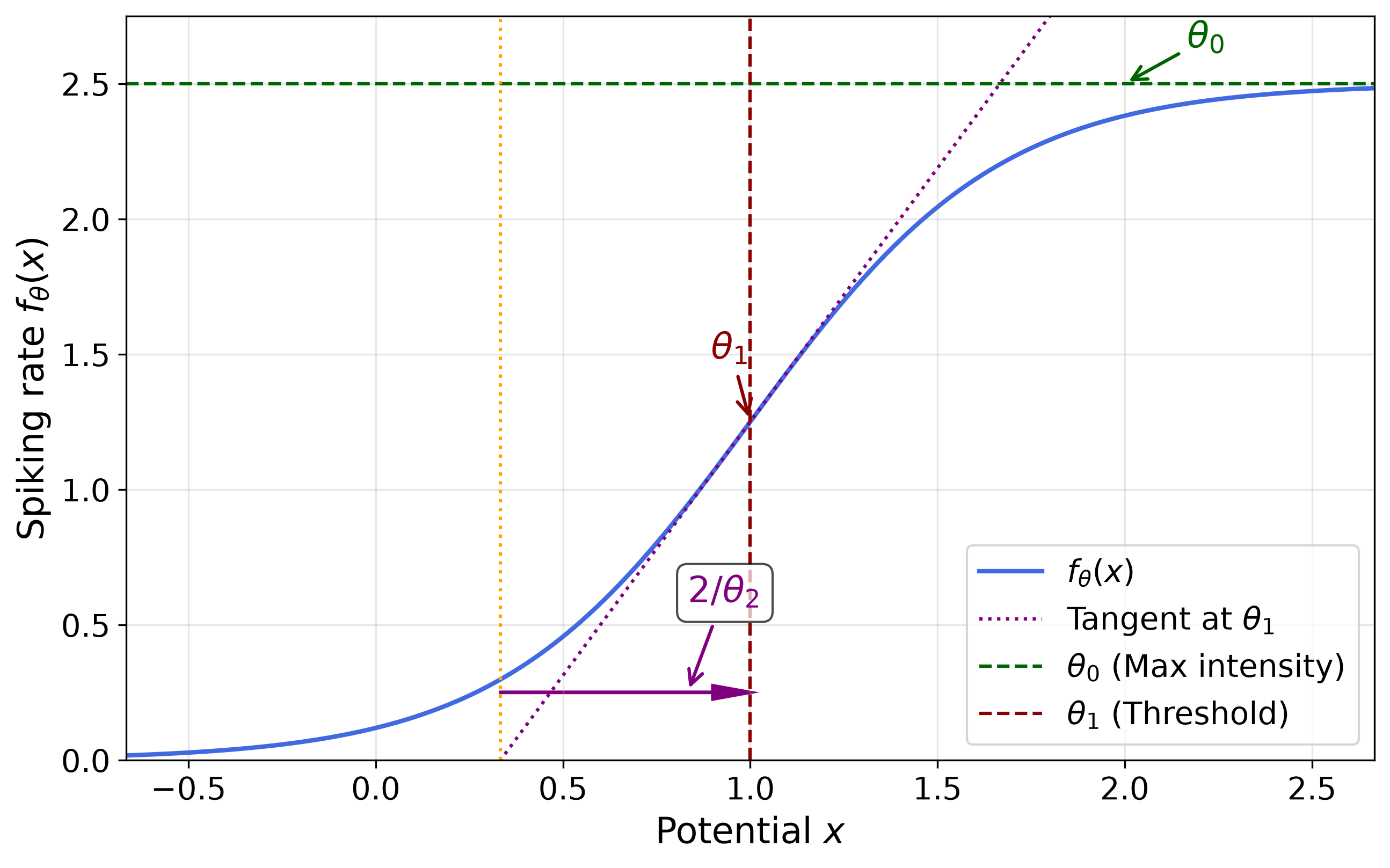}
	
	\captionof{figure}[width=\textwidth]{Spiking rate}
\end{minipage}

The parameter values chosen as reference are given in Table~\ref{tab_model_parameters_v16} for the model parameters and Table~\ref{tab_simulation_settings_v16} 
for the context of observation.
A persisting neuronal activity on the interval $[0, T]$ was observed, as planned, for these parameters.
Though one could discern a form of stabilization towards stationnarity,
the internal randomness of the system for this neuron number 
makes each realization unique in terms of the global activity, 
as detailed in Section~\ref{sec_visual_evidence}.
The choice of parameters was also adapted so that the three parameters $\theta_0, \theta_1$ and $\theta_2$ could be inferred with a similar efficiency,
which we qualify thanks to the numerical evaluation of the Fisher information matrix $\hat I_t^{N, \ts}$. This is discussed in next Section~\ref{sec_mle_alignment}, while we evaluate the prediction derived from the LAN approximation.

\begin{table}[htbp]
	\begin{tabular}{|c|c|c||c|c|c||c|c|c|}
		\hline
		$\varphi(x)$& $b(x)$& $\nu(\Rd u)$ &
		$\alpha$ & $\nu_1$ & $\nu_2$ & $\theta_0$ & $\theta_1$ & $\theta_2$  \\
		\hline
		$-x$ & $-\alpha x$& $\mathcal{U}(\nu_1\!-\!\nu_2, \nu_1\!+\!\nu_2)$ &
		$1.5$ & $2.4$ & $ 0.8=\nu_1 / 3$  & $2.5$ & $1.0$ & $3.0$  \\
		\hline
	\end{tabular}
	\centering
	\caption{Reference model parameters used in the simulations.}
	\label{tab_model_parameters_v16}
\end{table}

\begin{table}[htbp]
	\begin{tabular}{|c|c|c|c|}
		\hline
		$N$&	$T$ & $\Delta t$ &  $X_0^i$, $i \in \{1, \dots, N\}$ \\
		\hline
		$50$&	$20.0$ s & $0.002$ s & $\mathcal{U}(2, 4)$  \\
		\hline
	\end{tabular}
	\centering
	\caption{Reference context of observation.}
	\label{tab_simulation_settings_v16}
\end{table}

The dynamics was implemented 
with a   Python implementation  customized
for the facts that the exponentional leakage is easy to implement
and that jumps occur scarcely for each neuron at each small time-step, but reasonably as a whole.
The algorithm combines the Lie-Trotter approach in that it alternates between the jumps and the drift,
and the Sellke construction \cite{Se83} for the identification of the neurons that jump at each step.
Besides the interest of keeping the random number generation 
to its bare minimum,
the construction moreover produces a natural ordering according to which we can make the neurons jump.
Each neuron is treated in order, 
first with the common random interaction effect, 
then with the reset of its potential.
More details are provided in Appendix~\ref{app:simulation_details}.

\subsection{Visual evidence of persistent stochasticity}
\label{sec_visual_evidence}

To demonstrate the strength of the proposed approximation,
we evaluate it in a context of observation for which stochasticity is visibly not entirely averaged out,
though the system approaches a stationary regime.
More regular scenarios are also displayed in Appendix Section~\ref{app_larger_nbr}, for $N= 100$ and $N=500$ neurons.

\subsubsection{Individual neuron trajectories}
\label{sec_trajectories}

The randomness of the system can be observed first
by looking at neuron trajectories, which we show in Figure~\ref{fig_neuron_traj} for 5 neuron trajectories with two replicates of the same system (one per subpanel).
Even in the limiting regime, randomness persists at the neuronal level due to the spike events, that trigger neuron potential reset. These events are easily identified under the proposed scenario and occur several times for each neuron.
Although the limiting drift presumably accounts for most of the dynamics between these events, shared fluctuations introduce significant deviations from the deterministic flow across all neurons.

\begin{figure}
	\begin{subfigure}[b]{0.48\textwidth}
		\centering
		\includegraphics[width=\textwidth]{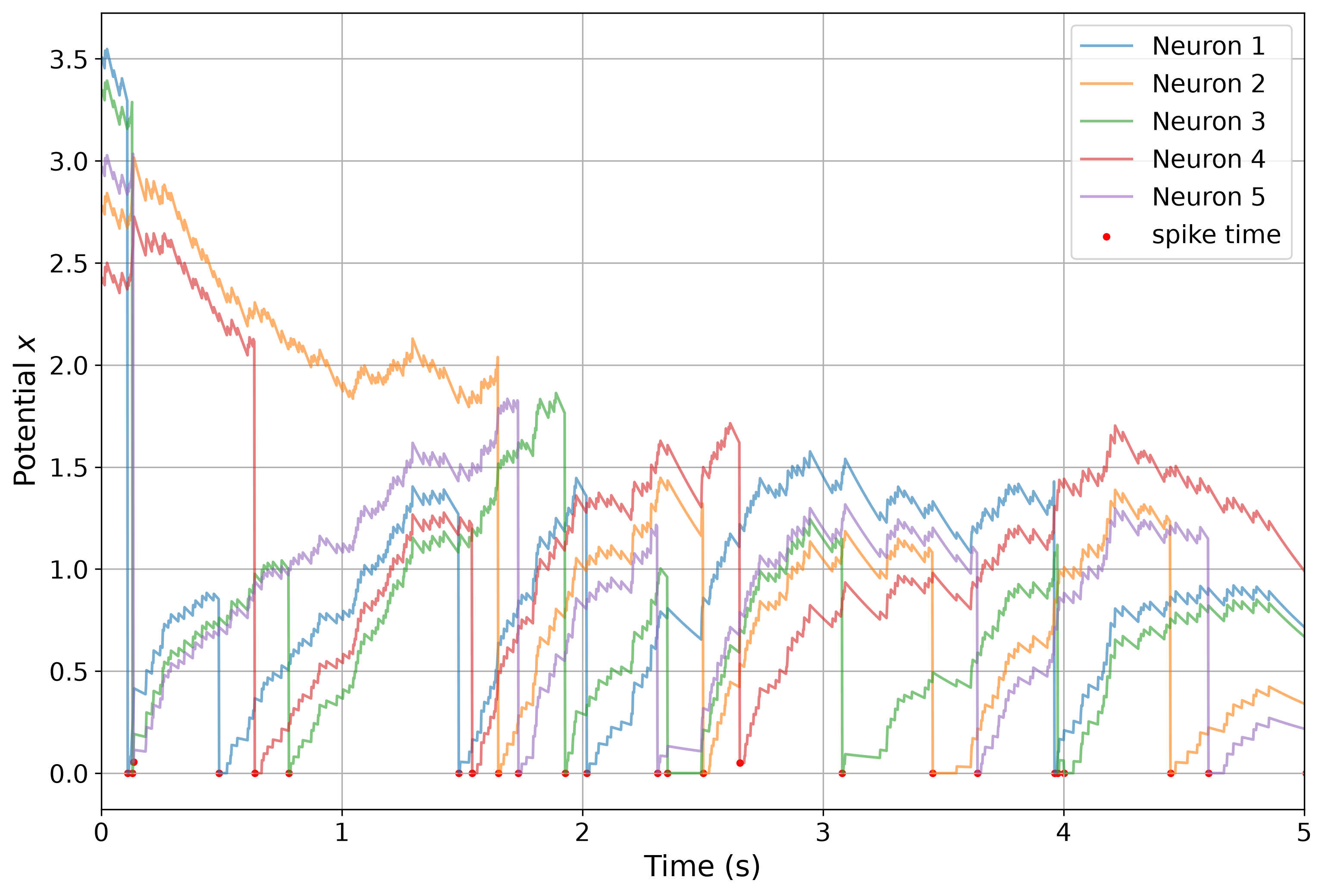}
	\end{subfigure}
	\hfill
	\vrule
	\hfill
	\begin{subfigure}[b]{0.48\textwidth}
		\centering
		\includegraphics[width=\textwidth]{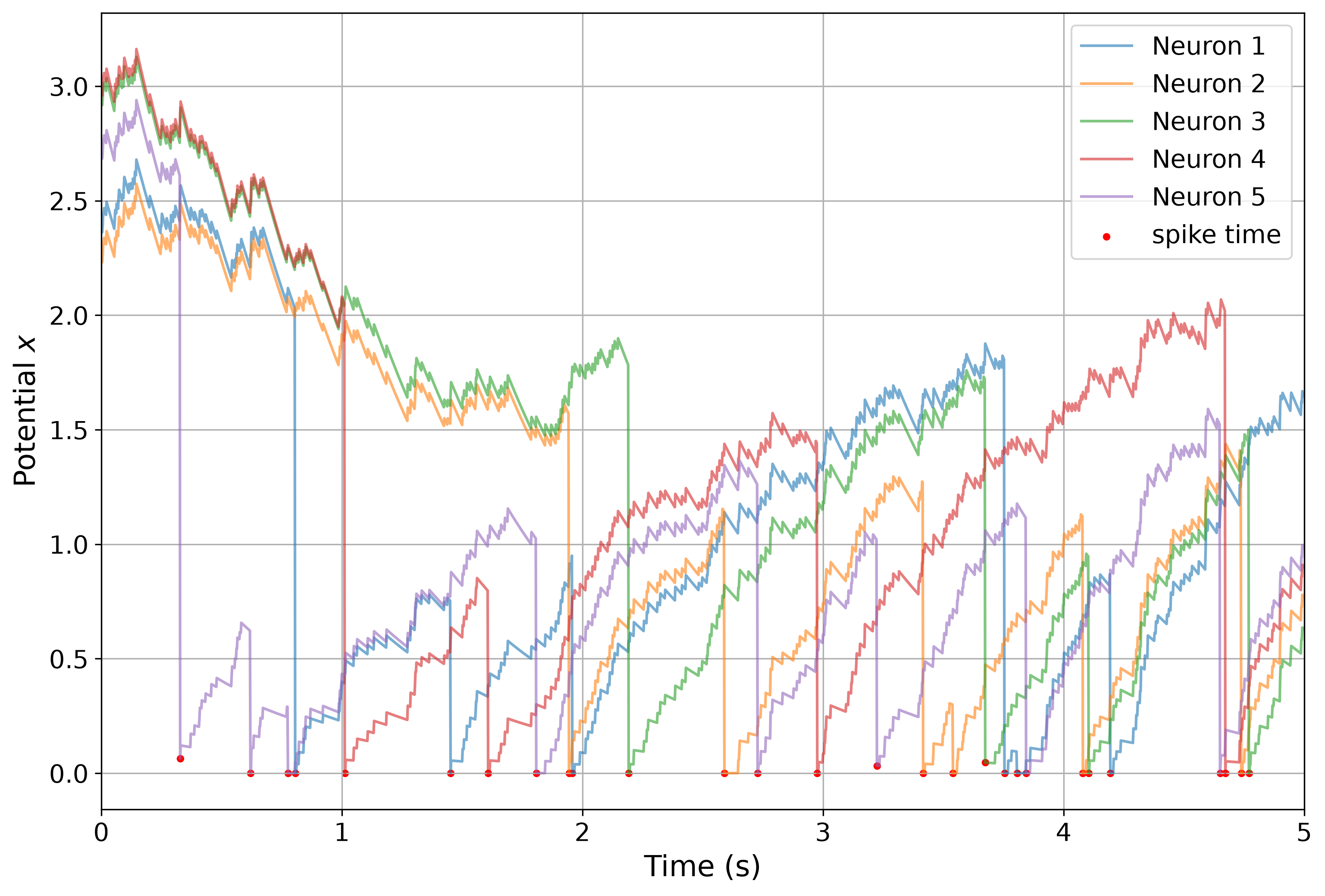}
	\end{subfigure}

	\caption{Illustration of 5 neuron state trajectories over the time-window $[0, 5]$, for two replicates.}
	\label{fig_neuron_traj}
\end{figure}

\subsubsection{Spike raster plot and global intensity}
\label{sec_spike_raster}

In Figure~\ref{fig_raster}, we additionally present an analysis of the global spiking intensity.
The spiking times for each neuron is displayed in subpanel $(A)$. We can observe globally that the spike events do not occur uniformly and that some time-periods appear to be reduced (between 5s and 7s) or on the contrary enriched (between 7.5s and 8s) in spike events.
We exploited a nonparametric gaussian kernel approach 
for estimating  the spiking intensity of the whole system.
The outcome is presented in blue in subpanel (B), 
and compared with the expected spike intensity (in red) that we can derive by computing the spiking rate $f_\theta$ 
over the empirical distribution of neuron potentials.
The blue curve appears to confirm the visual observation of fluctuations, though it is less confirmed than with the red curve. The red curve shows that the depletion is expected from the empirical distribution of neuron states.
While the  activity might have been larger than this expectation, the kernel bandwidth,  adjusted for spike event randomness, 
may be the main reason for the smaller decrease of the blue curve.

\begin{figure}
	\centering
	\begin{subfigure}[b]{0.48\textwidth}
		\centering
		\includegraphics[width=\textwidth]{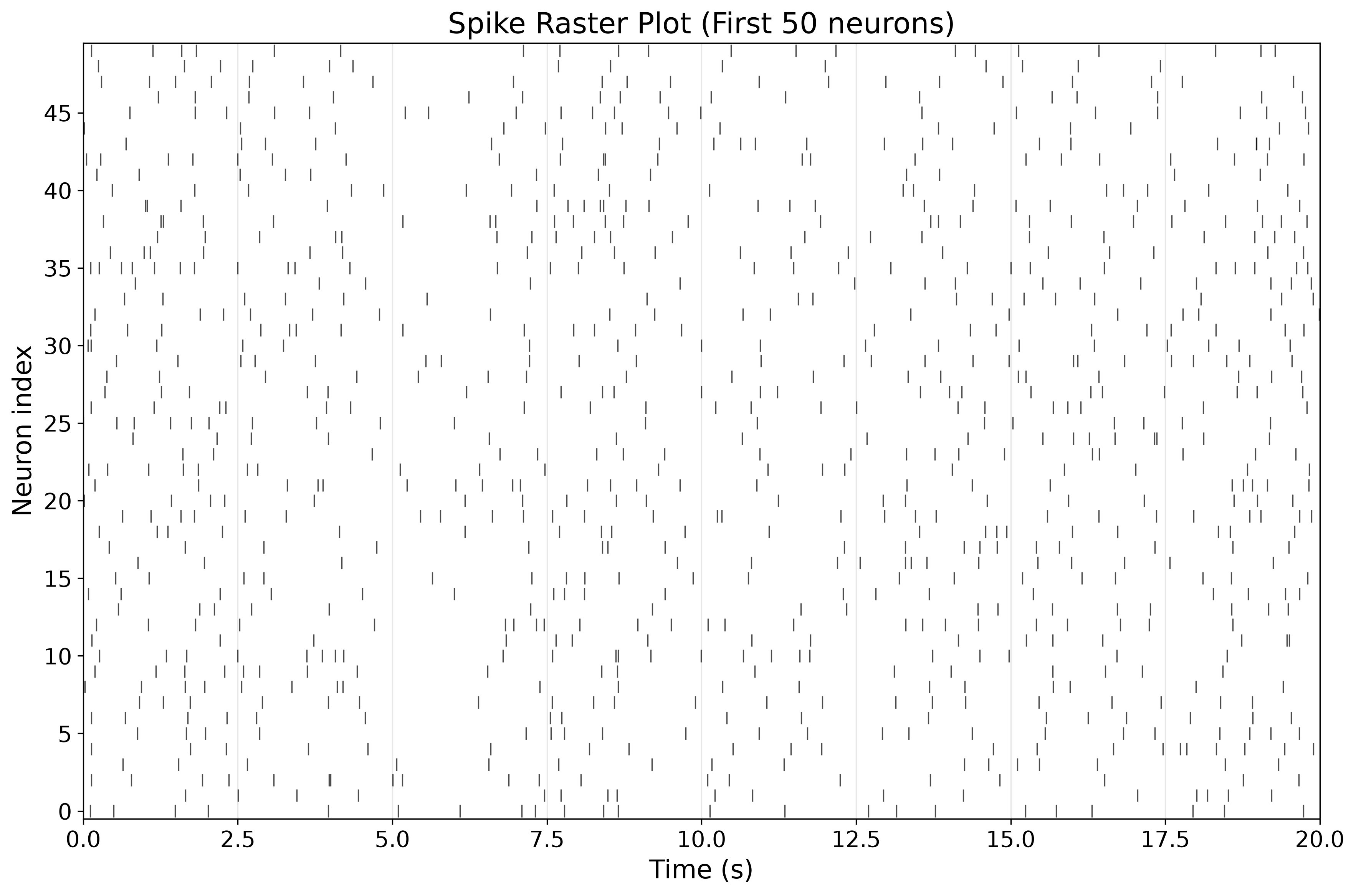}
		\caption{Raster plot of the spike trains for all 50 neurons. Time variations of the spike intensity can be identified.}
		
	\end{subfigure}
	\hfill
	\vrule
	\hfill
	\begin{subfigure}[b]{0.48\textwidth}
		\centering
		\includegraphics[width=\textwidth]{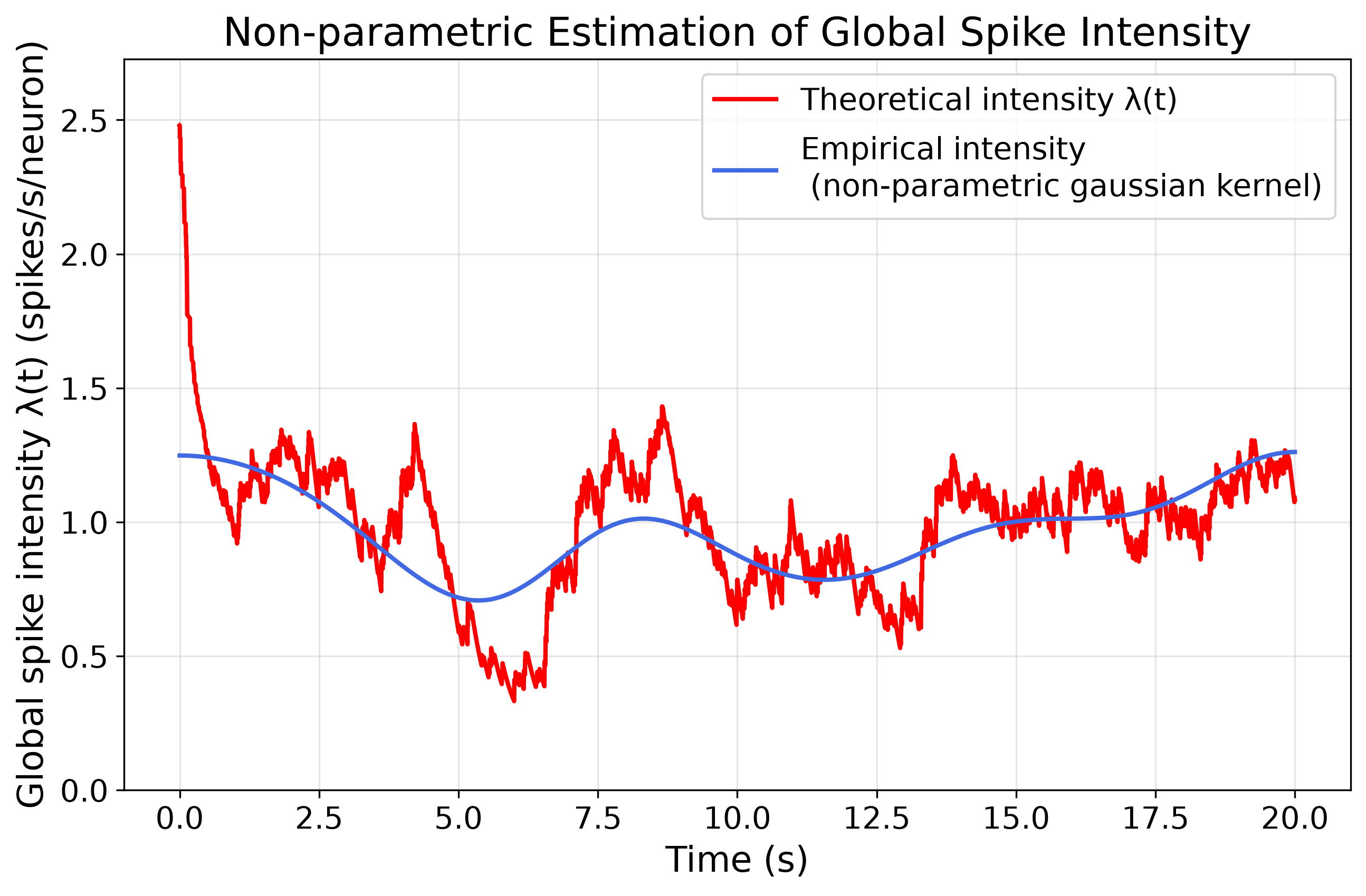}
		\caption{Estimation of the global spiking intensity, in blue with a nonparametric gaussian kernel intensity derived from the spike times,
			in red with the computation of the spiking rate over the neuronal states (for comparison).
		}
	\end{subfigure}
	\caption{Study of the spiking times.}
	\label{fig_raster}
\end{figure}

\subsection{Statistical validation of LAN quantities}
\label{sec_lan_validation}

To evaluate the LAN property as stated in Theorem~\ref{th:LAN}, we focus on the empirical behavior of the random sequence $\Delta_t^{N,\theta^*}$ and the empirical Fisher information matrix $I_t^{N,\theta^*}$, given respectively by Equations~\ref{eq_def_MN1} and \ref{eq_def_INts}, across the $100$ replicates. The LAN property predicts that $\Delta_t^{N,\theta^*}$ should converge in law to a Gaussian distribution $\mathcal{N}(0, I_t^{\theta^*})$, with the empirical Fisher information $I_t^{N,\theta^*}$ converging to its theoretical counterpart $I_t^{\theta^*}$,
defined in Equation~\ref{eq_def_I}.

Because it appears of practical interest to characterize
the variance of the MLE in terms of an estimator from the observed data, we will make a discrete computation of 
$\wht I_t^N = I_t^{N,\hat\theta_t^N}$
as the estimated Fisher information.
Since we want to select only one value $\wht I_t^N$ among the set of replicates, 
the chosen one  corresponds to a typical replicate.  The typical replicate is set as the one for which the relative error in Mahalanobis distance
is exactly at the median value, see Figure~\ref{fig_LAN_MLE}-(B) explanations.

The difference in Frobenius norm as compared to $I_t^{N,\theta^*}$ has been observed at around 5\%.
This in strong constrast with the relative difference between the discrete computation of $\Delta_t^{N,\hat\theta_t^N}$
and of $\Delta_t^{N,\theta^*}$,
the former being very close to 0
(which we can explain by relating $\Delta_t^{N,\hat\theta_t^N}$ to the derivative of the log-likelihood at $\hat\theta_t^N$).
Only the later is considered in the following.

The alignment between the empirical and theoretical covariance structures is first illustrated in Figure~\ref{fig_cov_comparison}-(A), which compares the empirical covariance matrix of $\Delta_t^{N,\theta^*}$ (computed over the $100$ replicates) with the theoretical Fisher information matrix $\wht I_t^{N}$. The visual similarity between the two matrices is confirmed by a relative Frobenius norm error of $11\%$. 
On the right panel, a similar analysis is conducted for the empirical covariance matrix of the MLE
as compared to $N^{-1} (\wht I_t^{N})^{-1}$.
The relative Frobenius norm error is at $20\%$, which is good yet not entirely satisfying for computing quasi-confidence intervals.

We conducted three multivariate goodness-of-fit tests for $\Delta_t^{N,\theta^*}$. The Henze-Zirkler test for multivariate normality \cite{HZ90} does not contradict the null hypothesis of normality ($p$-value of $0.45$). 
Then, Hotelling's $T^2$ \cite{Ho31} test does not reject either the null hypothesis of a null mean vector ($p$-value of $0.71$).
Finally, Mauchly's test \cite{Mau40} for sphericity was exploited 
after correcting the covariance matrix by considering $(\wht I_t^{N})^{-1/2} \Delta_t^{N,\theta^*}$.
The covariance structure is also far from being rejected ($p$-value of $0.20$).
So all three tests pass, providing strong empirical evidence that the distribution of $\Delta_t^{N,\theta^*}$ aligns with the theoretical predictions of the LAN property for $N=50$.

The same tests were then applied to the MLE.
Normality is confirmed (p-value of $0.52$ at the Henze-Zirkler test), as well as the mean value at $\ts$ 
(p-value of $0.52$ at the Hotelling's $T^2$ test),
and the projected covariance matrix
(p-value of $0.77$ at Mauchly's test 
applied 
to $N^{1/2}(\wht I^{N}_t)^{1/2}\, (\hat\theta_t^N - \ts)$).

\begin{figure}
	\centering
	\begin{subfigure}[b]{0.48\textwidth}
		\centering
		\includegraphics[width=\textwidth]{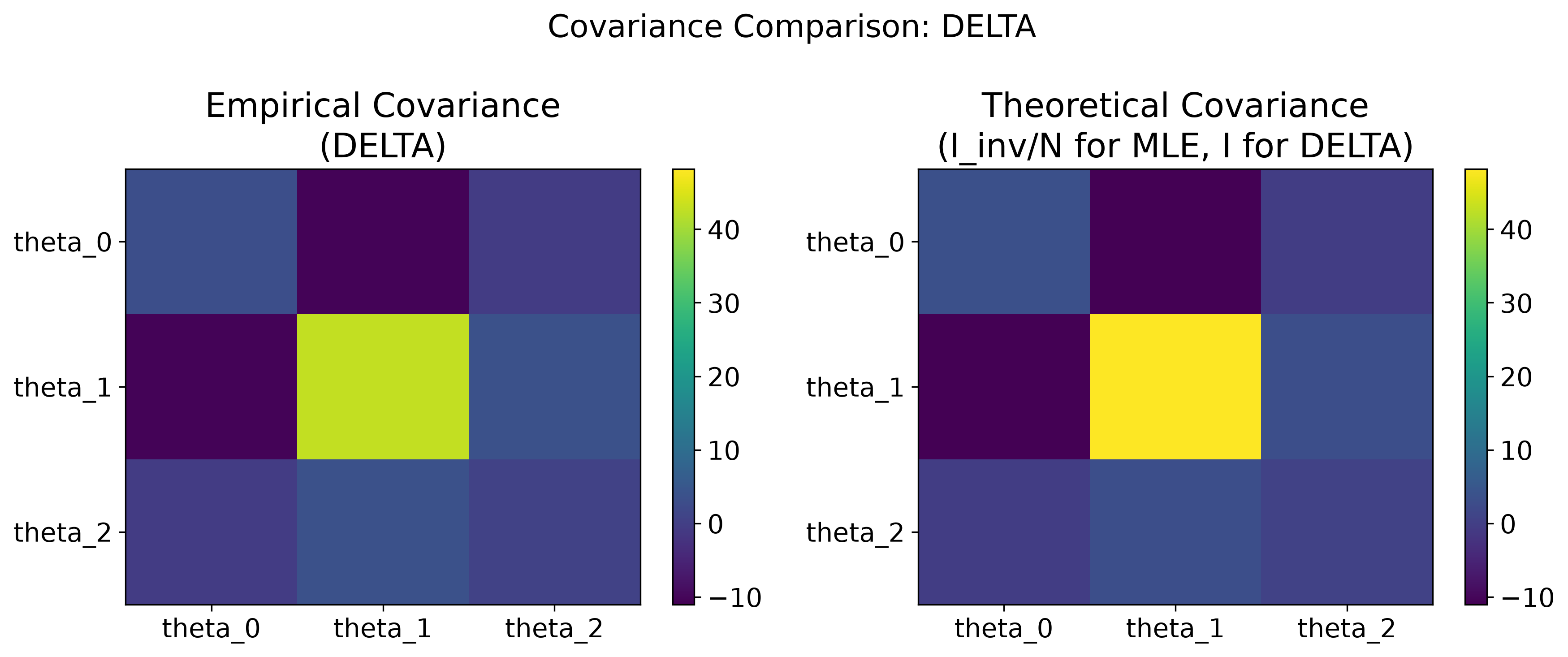}
		\caption{Empirical covariance matrix of $ \Delta^{N, \ts}_t$ (over 100 replicates) as compared to $\wht I^{N}_t$.
			Relative error in Frobenius norm: 0.11.}
	\end{subfigure}
	\hfill
	\vrule
	\hfill
	\begin{subfigure}[b]{0.48\textwidth}
		\centering
		\includegraphics[width=\textwidth]{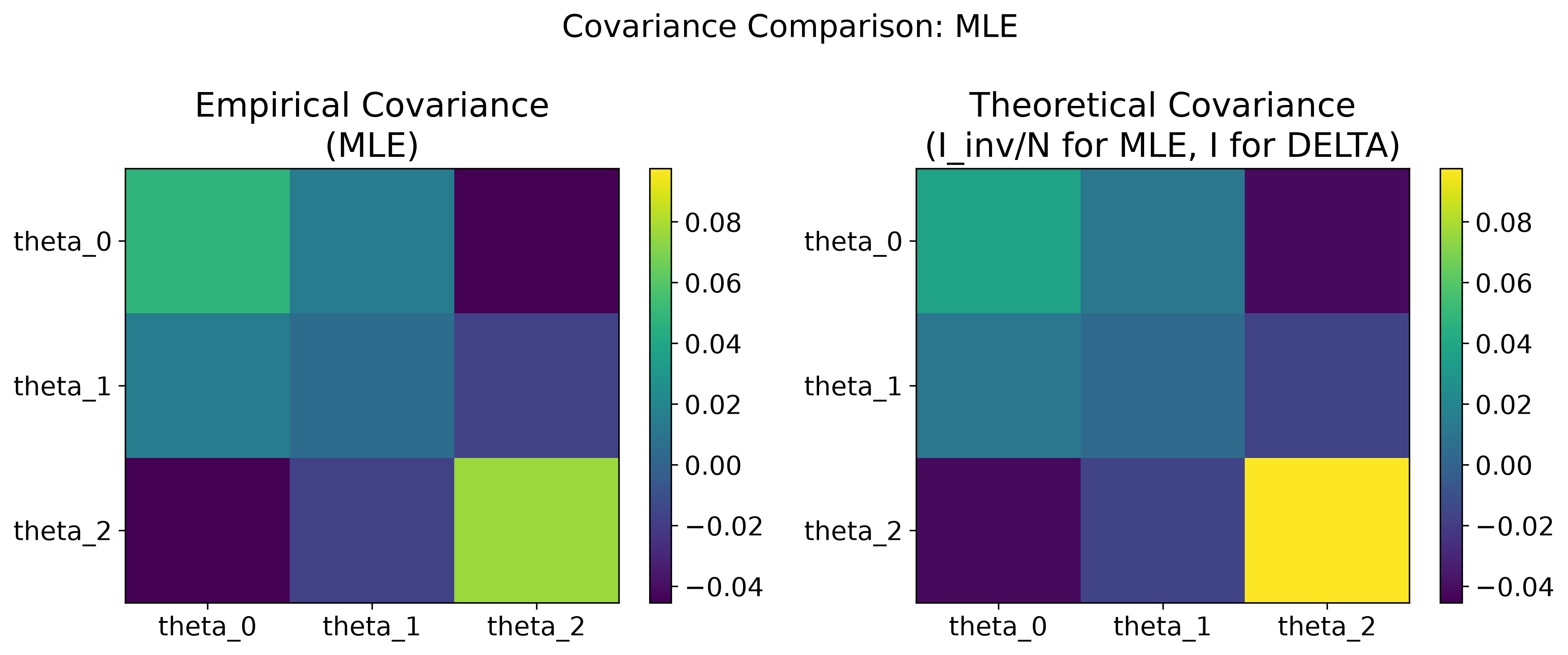}
		\caption{Empirical covariance matrix of the MLE as compared to $N^{-1} (\wht I^{N}_t)^{-1}$.
			Relative error in Frobenius norm: 0.20.
		}
	\end{subfigure}
	\caption{Comparison between the empirical covariance and the prediction with the estimated Fisher information matrix.}
	\label{fig_cov_comparison}
\end{figure}

\subsection{MLE performance and alignment with the LAN predictor}
\label{sec_mle_alignment}

In Figure~\ref{fig_LAN_MLE}-(A), 
we compare the histograms derived from the MLE $\hat \theta_t^N$ and the LAN predictor 
$\check \theta_t^N = \ts + 	(\wht I^{N, \ts}_t)^{-1} 	 \Delta^{N, \ts}_t/\sqrt{N}$.
It demonstrates that they have very similar distributions,
close to normality
and with very similar intervals of fluctuation (\texttt{Mean}$\pm 2$\texttt{Std}).

To quantify more precisely how closely the MLE  is well-described by the LAN predictor  in finite-$N$ regimes, we analyze the relative errors between the two quantities as compared to $\theta^*$ across a set of $100$ simulations.
The relative error is defined as 
$\|\hat \theta_t^N - \check \theta_t^N\|/\|\hat \theta_t^N - \theta_*\|\,.$
A small relative error is indicative 
that the higher-order terms in 
entail indeed second order corrections
for the MLE, and that the numerical estimation has been efficient.

Figure~\ref{fig_LAN_MLE}-(B) displays the cumulative distribution function (c.d.f.) of the relative error across the 100 simulations, in blue line for the Euclidian distance.
It is also natural to check 
the quality of estimation
when we take into account the expected covariance matrix,
as it was done e.g. in Corollary~\ref{corol:hajek}.
This corresponds to computing the relative error with respect to a relevant Mahalanobis distance $\|.\|_M$, 
where $\|\theta\|_M^2 = N^{1} \theta^\top \,\wht I_t^{N, \theta^*}\, \theta$\footnote{In computing the relative error,  the scaling factor $N$ plays no role and can be forgotten}.
As we see in Figure~\ref{fig_LAN_MLE}-(B),
the LAN predictor appears to fit closely 
the MLE  for  both distances, with very similar distributions. 
$50\%$ of the replicates have a relative error below 10\%,
and close to 90\% of them below 20\%, for both distances.

For some first sets of parameters,
arbitrarily chosen,
the relative error in Euclidian distance was much larger.
We suspected it could be partly due to the fact that one of the parameters (often $\theta_2$) is much worse estimated than the others, 
leading to larger deviations between $\hat \theta_t^N$   and $\check \theta_t^N$ in this direction as compared to $\theta_*$. 
This was our original motivation for computing 
the relative error 
in terms of the Mahalanobis distance $\|.\|_M$. 

\begin{figure}
	\centering
	\begin{subfigure}[b]{0.48\textwidth}
		\centering
		\includegraphics[width=\textwidth]{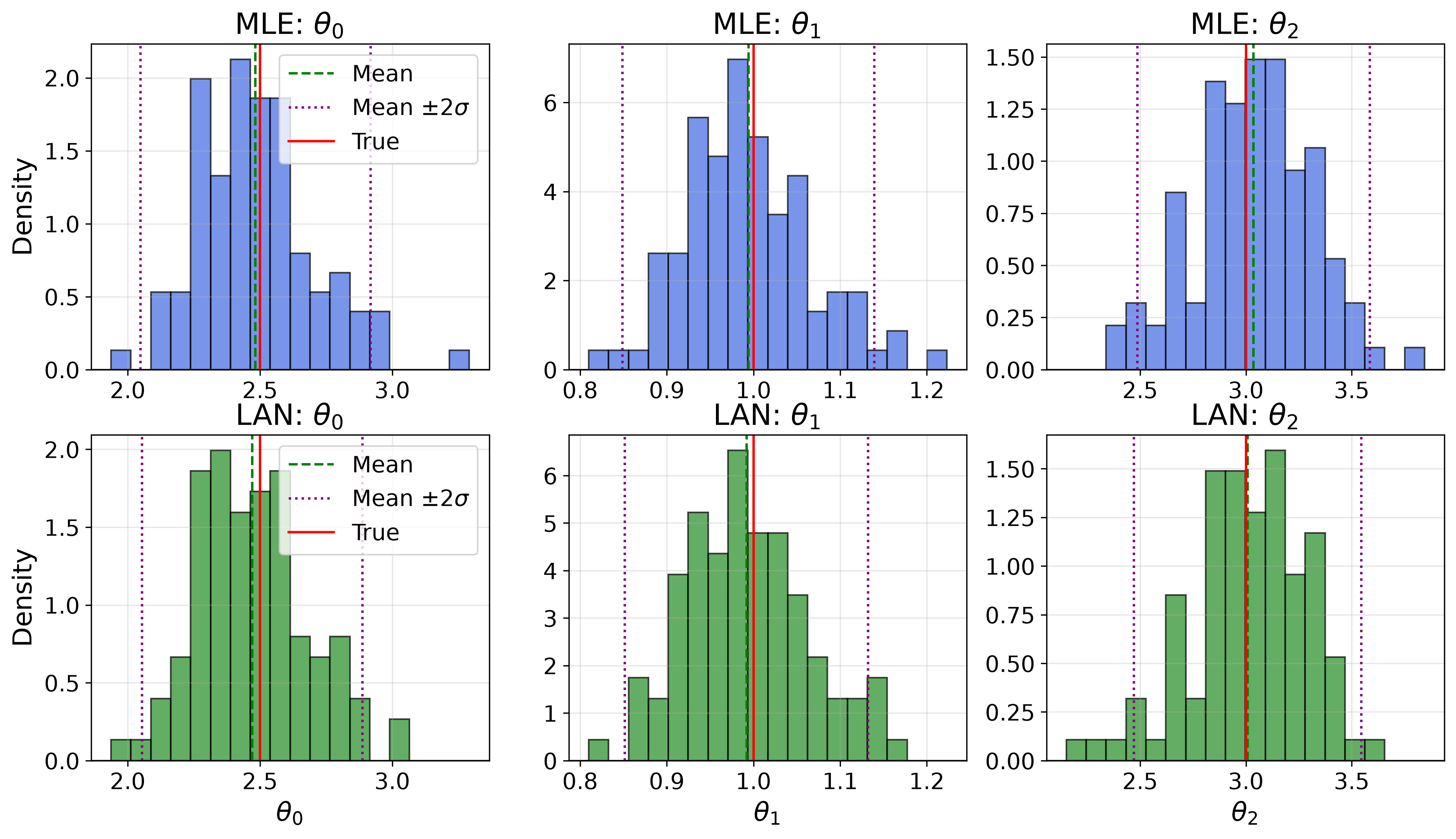}
		\caption{Histograms of the MLE (top row) and LAN predictor (bottom row) for the three parameters $\theta_0$, $\theta_1$, and $\theta_2$. The red lines indicate the true parameter values $\theta^*= [2.5, 1.0, 3.0]$, the dashed green line the empirical mean and the two orange dotted lines the interval of fluctuations (Mean $\pm 2\sigma$) . }
	\end{subfigure}
	\hfill
	\vrule
	\hfill
	\begin{subfigure}[b]{0.48\textwidth}
		\centering
		\includegraphics[width=\textwidth]{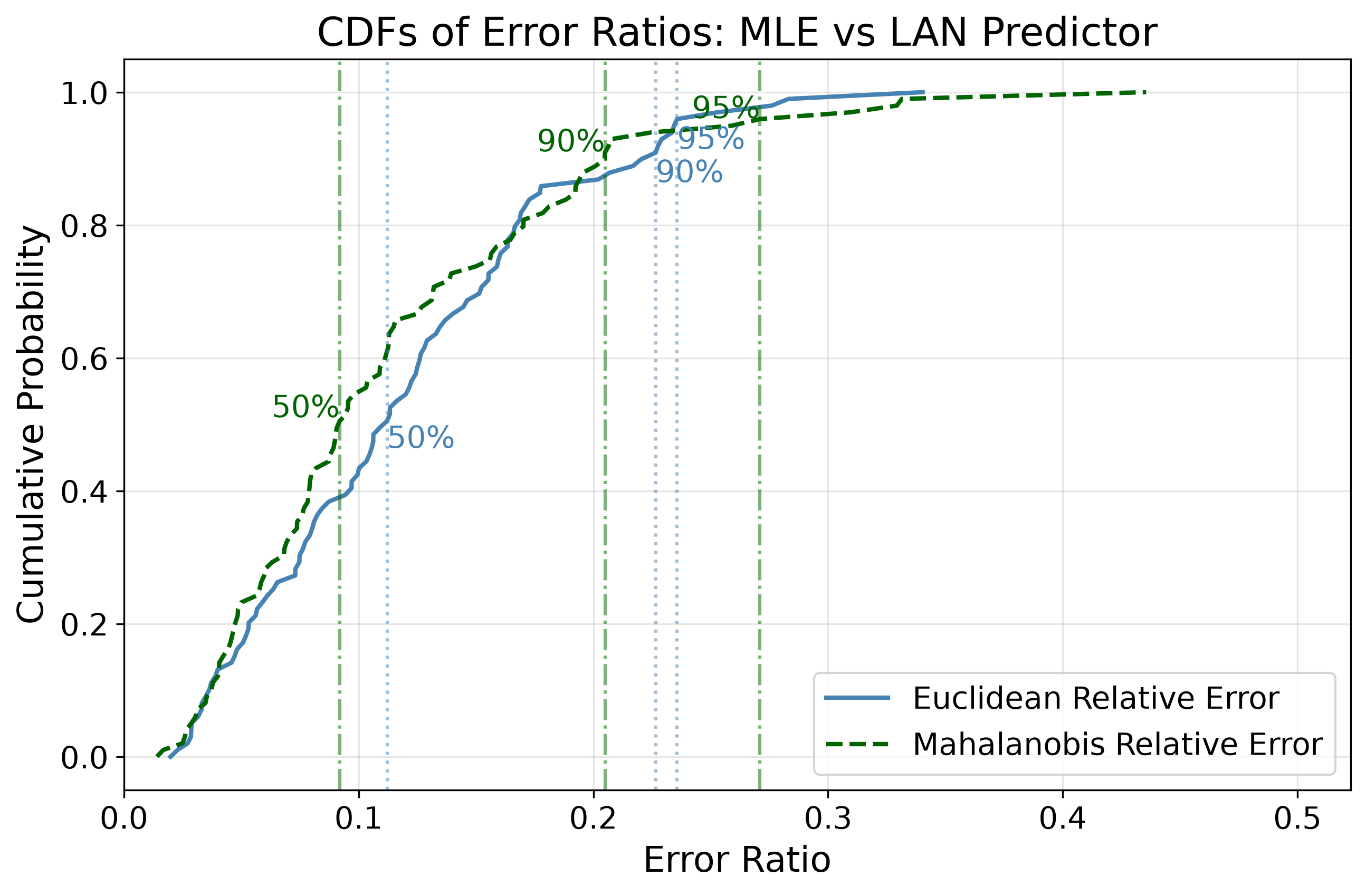}
		\caption{C.d.f. of the relative error $\|\hat \theta_t^N - \check \theta_t^N\|/\|\hat \theta_t^N - \theta_*\|$ for the Euclidian distance (blue line) and the Mahalanobis distance (green dashed line). Vertical lines display the quantiles of both distributions corresponding to 
			$50$\%, $90$\% and $95$\%.
		}
	\end{subfigure}
	\caption{Comparison between the estimated MLE $\hat \theta_t^N$ and the LAN predictor 
		$\check \theta_t^N = \ts + 	(\wht I^{N, \ts}_t)^{-1} 	 \Delta^{N, \ts}_t/\sqrt{N}$,
		as compared to $\ts$.}
	\label{fig_LAN_MLE}
\end{figure}

\subsection{Conclusion of the simulations}
\label{sec_cl_simu}

This numerical evaluation complements the theoretical results of Section~\ref{sec_results} and demonstrates that the LAN approximation remains accurate even for moderate population sizes where stochastic fluctuations are still visible.

\section{Auxiliary results}
\label{sec_aux}
This section gathers  results used to establish the  LAN property 
and to analyze the Maximum Likelihood Estimator.
\subsection{Continuity}
\begin{proposition}\label{prop:conv}
	Under Assumptions~\ref{ass:finite}-\ref{ass:moments},
	for any continuous and bounded function $g:\bR\to\bR$, 
	the following convergence holds in   $P^{\theta^*}_N$- probability
	$$\forall t>0, \quad \int_0^t\mu^{N,\ts}_s(g)ds\to\int_0^t\bar \mu_s^{\ts}(g)ds .$$
\end{proposition}
\begin{proof}
	We assume that the family of i.i.d. Poisson random measures $(\pi_i)_{ i\geq 1},$ is defined on some filtered  probability space $(\Omega, \mathcal F, (\mathcal F_t)_t,\Pr)$ and that for each $\theta\in\Theta$ the strong solution of \eqref{eq_limit} is constructed. The law on $\cD(\bR_+,\bR)$ equipped with the natural filtration of  this solution  is denoted by $\bar\mu^{\theta}.$
	We know that $\mu^{N, \ts} $ converges in law in $\mathcal P(\cD(\bR_+,\bR))$ to $\bar\mu^{\ts}$
	under Assumptions~\ref{ass:finite}-\ref{ass:moments},
	see \cite{andreis_mckeanvlasov_2018}.
	
	Let $g:\bR\to\bR$ be continuous and bounded, $t>0$ fixed and $F^t_g:\mathcal P(\cD(\bR_+,\bR))\to\bR$ given by
	$$\mathcal P(\cD(\bR_+,\bR))\ni m\mapsto\int_0^tm_s(g)ds.$$
	Using Billingsley, for fixed $s>0,$ the projection application $\mathcal P(\cD(\bR_+,\bR))\ni m\mapsto m_s\in \mathcal P(\bR)$ is continuous 
	in all point $\tilde m\in \mathcal P(\cD(\bR_+,\bR))$, such that
	$$\tilde m(\{\gamma\in\cD,\; \gamma_s\ne\gamma_{s-}\})=0.$$
	It is easy to see that the law of the limit McKean-Vlasov equation satisfy this property, since 
	\begin{equation*}
		\bar\mu^{\ts} (\{\gamma\in\cD,\; \gamma_s\ne\gamma_{s-}\})
		\leq  \Pr^{}(\pi(s\times[0,  B_0])>0)=0.
	\end{equation*}
	Using the continuous mapping theorem, if $m^n$ converges in law to $m$, then $m^n_s$ converges in law to $m_s$, which means that
	$\forall g\in C_b(\bR),$ $m^n_s(g)\to m_s(g).$ We conclude by dominated convergence theorem that
	$\int_0^t m_s^n(g)ds\to\int_0^tm_s(g)ds.$ As a result $F^t_g$ is continuous in $\bar\mu$ and the proposition follows. 
\end{proof}

\subsection{Handleable likelihood}

Let us give a more handleable expression of the likelihood. 
Denote $\bar \pi^j(\Rd t, \Rd z)=\int_{\bR}\pi^j(\Rd t, \Rd z, \Rd u)$ and as usually denote  $\wtd \pi^j$ the compensated measure: 
\begin{equation*}
	\wtd \pi^j(\Rd t, \Rd z, \Rd u ) 
	= \pi^j (\Rd t, \Rd z, \Rd u ) - \Rd s\, \Rd z\, \nu(\Rd u),\quad
	\wtd{ \bar \pi}^j(\Rd t, \Rd z) 
	= \int_\bR \wtd \pi^j(\Rd t, \Rd z, \Rd u)\,.
\end{equation*}
Define
\begin{equation}\label{eq_def_JN}
	J^{N,\ts} (\Rd s, \Rd x)
	= \sum_{j= 1}^N\sum_{n\ge 1} \delta_{\Big(T^j_n, X^{j, N,\ts}_{T^j_n -}\Big)}(\Rd s, \Rd x),
\end{equation}
in which all jump times and corresponding neuron potentials are recorded,
and the related compensated measure:
\begin{equation}\label{eq_def_JNcomp}
	\wtd J^{N,\ts}  = J^{N,\ts} - \varUpsilon^{N, \theta^*},\quad 
	\text{with }\;  \varUpsilon^{N, \ts} (\Rd s, \Rd x)
	= N\, f_{\ts}(x)\,
	\mu_s^{N, \ts}(\Rd x)\, \Rd s\,,
\end{equation}
and
\begin{equation*}
	\mu_s^{N, \ts}(\Rd x)
	= \frac1N\sum_{j = 1}^N \delta_{X^{j, N,\ts}_{s}}(\Rd x) ,
\end{equation*}
where $ \mu_s^{N, \ts}$ denotes the empirical measure  of the  system $(X^{i, N,\theta^*})_{ i\le N}$ at time $t. $

Therefore
\begin{equation}\label{eq_def_LN}
	\log  L_t^{ \theta/\theta^*}  
	= \int_{ [0, t ] \times \bR} \log \left( \frac{ f_{\theta} }{ f_{\theta^*} }\right)(x)\; J^{N,\ts} ( \Rd s, \Rd x) 
	- \int_{ [0, t ] \times \bR} \left( \frac{ f_{\theta} }{ f_{\theta^*} }(x) - 1\right)  \varUpsilon^{N, \ts} (\Rd s, \Rd x)\,,
\end{equation}
which can also be rewritten as 
\begin{equation}\label{eq:decllh}
	\log  L_t^{ \theta/\theta^*}  
	= M_t^{N,\theta/\theta^*}+A_t^{N,\theta/\theta^*}  \,,
\end{equation}
where 
\begin{align*}
	M_t^{N,\theta/\theta^*} &:= \iint_{ [0, t ] \times \bR}\left(
	\log  \frac{ f_{\theta} }{ f_{\theta^*} }\right)(x)
	\;  \wtd J^{N,\ts}  (\Rd s, \Rd x)
\end{align*}
and
\begin{align*}
	A_t^{N,\theta/\theta^*} &:= \int_0^t\!\int_{\bR}\left(  
	\log  \frac{ f_{\theta} }{ f_{\theta^*} }  
	-  \frac{ f_{\theta}}{f_{\theta^*}} + 1
	\right)  (x)\;
	\varUpsilon^{N, \theta^*}  (\Rd s, \Rd x)
	\\&= N\int_0^t \mu^{N, \ts}_s \Big[
	\Big( \log  \frac{ f_{\theta} }{ f_{\theta^*} }  -  \frac{ f_{\theta}}{f_{\theta^*}} + 1\Big)\cdot f_{\theta^*}
	\Big]\; \Rd s. 
\end{align*}


Now, put $\theta=\ts+\hDev/\sqrt N$ and define
\begin{equation}\label{eq:decM}
	M_t^{N,\theta^*/\theta^*}:=\hDev\,  \Delta_t^{N,\theta^*}
	+\wtd M_t^{N, \theta/\theta^*}
	,
\end{equation}
with
\begin{equation}\label{eq_def_MN1}
	\begin{split}
		\;  \Delta_t^{N,\theta^*}
		&:=\frac 1{\sqrt N}\iint_{[0,t]\times \bR}\frac{\dot{f}_{\theta^*} }{{f}_{\theta^*}} (x)\;
		\wtd J^{N,\ts}  ( \Rd s, \Rd x)
		\\& =\frac 1{\sqrt N}\sum_{j=1}^N 
		\iint_{[0,t]\times \bR^+}
		\frac{\dot{f}_{\theta^*} }{{f}_{\theta^*}} \big(X^{j, N,\ts}_{s- }\big) 
		\idc{ z \le f_{\theta^*} \big( X^{j, N,\ts}_{s- }\big)} 
		\wtd{ \bar \pi}^j (\Rd s, \Rd z )\,
	\end{split}
\end{equation}
and
\begin{align*}
	\wtd M_t^{N, \theta/\theta^*}
	&:= \iint_{[0,t]\times \bR}\left ( \log \frac{f_{\theta }}{f_{\theta^*}}    -(\theta - \theta^*)\eT\cdot \frac{\dot{f}_{\theta^*}}{f_{\theta^*}}\right ) (x)\; \tilde J^{N,\ts} ( \Rd s, \Rd x)
	\\&=\sum_{j=1}^N\int_{[0,t]\times \bR_+}\left ( \log
	\frac{f_{\theta }}{f_{\theta^*}}   
	-(\theta - \theta^*)\eT\cdot  \frac{\dot{f}_{\theta^*}}{f_{\theta^*}} \right )
	\big(X^{j, N,\ts}_{s- }\big) 
	\idc{ z \le f_{\theta^*} \big( X^{j, N,\ts}_{s- } \big) }
	\wtd{ \bar \pi}^j (\Rd s, \Rd z)\,.
\end{align*}
Letting $g:y\in (0,\infty)\mapsto \log(1+y) - y$,
we observe that $g(y)\sim - \frac{1}{2} y^2$ for $y$ in the vicinity of $0$.
This explains the interest in the following decomposition of $A_t^{N,\theta/\ts}:$
\begin{equation}\label{eq:decA}
	A_t^{N,\theta/\theta^*}=\frac{-1}{2}\,\hDev\eT
	I_t^{N, \ts}\hDev + \wtd A_t^{N, \theta/\theta^*},
\end{equation} 
where 
\begin{equation}\label{eq_def_INts}
	I_t^{N, \ts} = \frac{1}{N} \iint_{[0,t]\times \bR}
	\Big(\frac{\dot f_{\theta^*}}{f_{\theta^*}}\Big)^{\otimes 2} (x)\;
	\varUpsilon^{N, \theta^*} (\Rd s,\Rd x) 
	=\int_0^t 
	\mu^{N, \ts}_s \Big[\frac{(\dot{f}_{\theta^*})^{\otimes 2}}{f_{\theta^*}}
	\Big]\; \Rd s\,
\end{equation}
and
\begin{equation*}
	\begin{split}
		\wtd A_t^{N, \theta/\theta^*}
		&= \iint_{[0,t]\times \bR}\Big(
		g\circ\Big( \frac{ f_{\theta  }}{f_{\theta^*}} - 1\Big) + \frac{1}{2} 
		\Big[(\theta-\theta^*)\eT
		\frac{\dot f_{\theta^*}}{f_{\theta^*}}\Big]^2
		\Big)(x)
		\;\varUpsilon^{N, \theta^*} (\Rd s,\Rd x)
		\\&= N \int_0^t 
		\mu^{N, \ts}_s\Big[ \Big(
		g\circ\Big( \frac{ f_{\theta  }}{f_{\theta^*}} - 1\Big) + 
		\frac{1}{2} 
		\Big[(\theta-\theta^*)\eT
		\frac{\dot f_{\theta^*}}{f_{\theta^*}}\Big]^2
		\Big)\cdot f_{\theta^*}\Big] \Rd s
		\,.
	\end{split}
\end{equation*}

\section{Proof of the LAN property}\label{sec:proof_lan}

\begin{proposition}\label{prop_aux} Under Assumptions~\ref{ass:theta} to \ref{ass:deriv}, for all $t>0,$ the following convergences hold, as $N\to \infty$,
	\begin{enumerate}[(i)]
		\item $ I_t^{N, \ts} {\longrightarrow}\; I_t^{\theta^*}$ in $\Pr^{\theta^*}_N$-probability;
		\item $
		\Delta_t^{N,\theta^*}
		{\longrightarrow} \;  {\mathcal{N}}(0, I_t^{\theta^*})$ in distribution  under $\Pr^{\theta^*}_N$;
		\item $ \wtd A_t^{N, \theta/\theta^*} \longrightarrow 0$ in  $\Pr^{\theta^*}_N$\ -probability;
		\item $\wtd M_t^{N, \theta/\theta^*}\longrightarrow 0$ in $\Pr^{\theta^*}_N$-probability.
	\end{enumerate}
\end{proposition}
\begin{proof}
	$(i)$ Using the representation \begin{equation*}
		I^{N, \theta^*}_t = \int_0^t\mu^{N, \ts}_s \Big[\frac{(\dot{f}_{\theta^*})^{\otimes 2}}{f_{\theta^*}}
		\Big]\; \Rd s\,
	\end{equation*}
	with $(\dot{f}_{\theta^*})^{\otimes 2}/f_{\theta^*}$ continuous  
	and bounded  by $B_1^2 B_0$  
	according to Assumption \ref{ass:dif}, 
	item $(i)$ follows from Proposition \ref {prop:conv}.
	
	
	$(ii)$ 
	Observe that  this result can be derived from a following more general statement: 
	\begin{equation*}
		(\Delta_s^{N,\theta^*})_{s\geq 0}
		\rightarrow	
		\Big(\int_0^s
		\Big(\bar\mu^{\ts}_r \Big[\frac{(\dot{f}_{\theta^*})^{\otimes 2}}{f_{\theta^*}} 		\Big]\Big)^{1/2}
		\; \Rd B_r\Big)_{s\geq 0},
	\end{equation*}
	in distribution in $\cD(\bR_+,\bR),$ as $N\to \infty$.
	Here $(B_s)_{s\geq 0}$ denotes a standard $d$-dimensional Brownian motion. To prove this last assertion  we  use martingale convergence theorem VIII.3.22 of \cite{JS}, therefore we need to satisfy the following convergences in $\Pr^{\ts}_{N}$ probability
	\begin{equation}\label{eq:bigjumpscontrol}
		\int_{\bR_+}\int_{\bR}
		\frac{(\dot{f}_{\theta^*})^{\otimes 2}}{f_{\theta^*}}(x)\,
		\idc{\Big \| \frac{\dot{f}_{\theta^*} }{f_{\theta^*}}(x)\Big \|>\eps \sqrt N}
		\mu^{N, \ts}_s(\Rd x)\;\Rd s
		\rightarrow 0, \quad \forall \eps>0\,,
	\end{equation}
	and Condition  $[\gamma'5-D]$ in \cite[p.432]{JS}
	\begin{equation}\label{eq:quadvariat}
		< \Delta^{N,\theta^*}>_t\rightarrow I_t^{\ts},
	\end{equation}
	where the quadratic variation is identified as follows:
	\begin{equation*}
		<\Delta^{N,\theta^*}>_t
		= \frac 1{N}\iint_{[0,t]\times \bR}
		\Big(\frac{\dot{f}_{\theta^*} }{{f}_{\theta^*}}\Big)	^{\otimes 2} (x)\;
		\varUpsilon^{N,\ts}  ( \Rd s, \Rd x)
		=\int_0^t 
		\mu^{N, \ts}_s \Big[\frac{(\dot{f}_{\theta^*})^{\otimes 2}}{f_{\theta^*}}
		\Big]\; \Rd s
		=  I^{N, \ts}_t\,.
	\end{equation*}
	Therefore \eqref{eq:quadvariat} is already proved
	in  point $(i)$.
	With $\|\dot{f}_{\theta^*}/f_{\theta^*}\|$  bounded by $B_1$,
	\eqref{eq:bigjumpscontrol} is straightforward, since
	for $N$ sufficiently large the  indicator function of the event $\Big \| \frac{\dot{f}_{\theta^*} }{f_{\theta^*}}(x)\Big \|>\eps \sqrt N$ is zero.
	It concludes the proof of the point $(ii).$ 
	
	$(iii)$ We remark that $g(y) + (1/2)\, y^2 = O(y^3)$ as $y$ tends to $0$.
	For $\frac{\dot f_{\theta^*}}{f_{\theta^*}}$
	is bounded, Assumption~\ref{ass:deriv}
	entails that 
	$\Big( \frac{ f_{\theta  }}{f_{\theta^*}} - 1\Big)$
	is sufficiently small for $N$ large enough so that the following holds for some $C_1, C_3,
	C_g >0$:
	\begin{equation*}
		\begin{split}
			\Big\| \frac{ f_{\theta  }}{f_{\theta^*}} - 1\Big\|_\infty
			&\le C_1\, \|\theta-\ts\|,\qquad 
			\Big\|\Big( \frac{ f_{\theta  }}{f_{\theta^*}} - 1\Big)^2
			- \, 
			\Big( (\theta-\ts)\eT\frac{ \dot f_{\ts }}{f_{\ts}}\Big)^2\Big\|_\infty
			\le C_3\, \|\theta-\ts\|^3\,,
			\\&\Big\|g\circ\Big( \frac{ f_{\theta  }}{f_{\theta^*}} - 1\Big)
			+ \frac{1}2\, \Big( (\theta-\ts)\eT\frac{ \dot f_{\ts }}{f_{\ts}}\Big)^2
			\Big\|_\infty
			\le C_g\, \|\theta-\ts\|^3 = O\Big(\frac1{N\,\sqrt{N}}\Big)\,.
		\end{split}
	\end{equation*}
	Since $f_{\ts }$ is bounded and $\mu^{N, \ts}_s$ has total mass $1$, it concludes the proof of $(iii)$.
	
	$(iv)$ We prove it by justifying the convergence to 0 of the quadratic variation:
	\begin{equation*}
		\begin{split}
			< \wtd M^{N, \theta/\theta^*}>_t
			&= \iint_{[0,t]\times \bR}\Big( \log \big[\frac{f_{\theta }}{f_{\theta^*}} (x) \big]  -(\theta - \theta^*)\eT\cdot \frac{\dot{f}_{\theta^*}}{f_{\theta^*}}(x)\Big)^2\; \varUpsilon^{N,\ts}  (\Rd s, \Rd x)
			\\&=N\int_0^t 
			\mu^{N, \ts}_s \Big[
			\Big( \log \big[\frac{f_{\theta }}{f_{\theta^*}}  \big] -(\theta - \theta^*)\eT\cdot \frac{\dot{f}_{\theta^*}}{f_{\theta^*}}\Big)^2
			\cdot f_{\ts}
			\Big]\; \Rd s\,.
		\end{split}
	\end{equation*}
	
	In the same spirit as for point $(iii)$, there exists $C_2>0$ such that:
	\begin{equation*}
		\Big\|\log \big[\frac{f_{\theta }}{f_{\theta^*}}  \big] -(\theta - \theta^*)\eT\cdot \frac{\dot{f}_{\theta^*}}{f_{\theta^*}}\Big\|_\infty
		\le C_2\, \|\theta - \theta^*\|^2
		= O(1/N)\,.
	\end{equation*}
	Therefore $< \wtd M^{N, \theta/\theta^*}>_t = O(1/N)$ which entails that the martingale  $ \wtd M^{N, \theta/\theta^*}$ converges to 0 under $\Pr^{\ts}_{N}$, as $N\to \infty$.
\end{proof}

\begin{proof}[Proof of Theorem~\ref{th:LAN}]
	
	Using the decomposition \eqref{eq:decllh} together with decompositions \eqref{eq:decM}, \eqref{eq:decA},
	we can write
	$$\log  L_t^{ \theta/\theta^*}  
	=  \hDev  \Delta_s^{N,\theta^*}-\frac 12 \hDev\eT I_t^{N, \theta^*}\hDev + r_N(\theta,\ts)\,,
	$$
	where $$r_N(\theta,\ts)= \wtd M_t^{N, \theta/\ts}+\wtd A_t^{N, \theta/\ts}.$$
	The proof is complete with the results of Proposition~ \ref{prop_aux} 
	telling us that both $\wtd M_t^{N, \theta/\ts}$ 
	and $\wtd A_t^{N, \theta/\ts}$ converge to zero in $\Pr^{\theta^*}_N$-probability, as $N\to \infty$. 
	
\end{proof}

\section{Maximum Likelihood estimator}
\label{sec_MLE}
\subsection{Consistency }\label{sec:consistency}
In this section we show the consistency of the MLE. 
\subsubsection{Proof of the consistency of the MLE }\label{sec:proof_consistency}

\begin{proof}[of Proposition \ref{prop_cons}]
	We first show that
	\[
	\frac 1N \log L_t^{ \theta/\theta^*}\big(X^{\theta^*,N}_{[0,t]}\big) \rightarrow
	A^{\theta^*}_t (\theta)  \quad 
	\]
	in $\Pr^{\theta^*}_N-$ probability, as $N\to \infty$, where $A^{\theta^*}_t:\Theta\to \mathbb{R}$ is such that $A^{\theta^*}_t (\theta)<0 = A^{\theta^*}_t (\theta^*)$ for any $\theta\in\Theta$
	with $\theta\neq \ts$.
	
	Recall (see LAN subsection) that the log-likelihood ratio process  
	is given by  $$ \log  L_t^{ \theta/\theta^*}  
	= A_t^{N,\theta/\theta^*} + M_t^{N,\theta/\theta^*},$$
	where the processes $A_t^{N,\theta/\theta^*}$ 
	and $M_t^{N,\theta/\theta^*}$  are defined in \eqref{eq:decllh}. 
	On the one hand,  using Proposition \ref{prop:conv} 
	together with Assumption \ref{ass:dif} $(i)$, we get that
	\begin{equation}\label{eq_conv_MNT}
		\frac 1{N^2}  <M^{N,\theta/\theta^*}>_t
		= \frac 1N\int_0^t \mu^{N, \ts}_s \Big[
		\Big( \log  \frac{ f_{\theta} }{ f_{\theta^*} }\Big)^2\cdot f_{\theta^*}
		\Big]\; \Rd s
		\rightarrow 0\,,
	\end{equation}
	which entails that 
	$$\frac 1N {M_t^{N,\theta/\theta^*}}\longrightarrow 0\; \; \mbox{in}\; \; \Pr^{\ts}_N-\mbox{probability}.$$
	For the consistency, we will need though a stronger result to deal with the uniformity in $\theta$ for this convergence.
	
	On the other hand, we exploit again Proposition~\ref{prop:conv}
	since the function $\big(\log  \frac{ f_{\theta} }{ f_{\theta^*} }  -  \frac{ f_{\theta}}{f_{\theta^*}} + 1\big)\cdot f_{\theta^*}$
	is bounded continuous on $\bR$, leading to:
	\begin{multline}\label{eq_conv_ANT}
		\frac 1N A^{N,\theta/\theta^*}_t
		= \int_0^t \mu^{N, \ts}_s \Big[
		\Big( \log  \frac{ f_{\theta} }{ f_{\theta^*} }  -  \frac{ f_{\theta}}{f_{\theta^*}} + 1\Big)\cdot f_{\theta^*}
		\Big]\; \Rd s \\
		\rightarrow \int_0^t \bar\mu^{\ts}_s \Big[
		\Big( \log  \frac{ f_{\theta} }{ f_{\theta^*} }  -  \frac{ f_{\theta}}{f_{\theta^*}} + 1\Big)\cdot f_{\theta^*}
		\Big]\; \Rd s:= A^{\theta/\theta^*}_t,
	\end{multline}
	in $\Pr^{\ts}_N-$ probability, as $N\to \infty$.
	Notice that $\log(r)-r+1\leq 0, \forall r\in \mathbb R,$ with the equality holding only at $r=1$.
	Therefore, 
	$$A^{\theta/\theta^*}_t \le 0 = A^{\ts/\theta^*}_t .$$
	In addition, Assumption~\ref{ass_non_deg} entails that  $A^{\theta/\theta'}_t < 0$ for any $\theta \neq \theta'\in \Theta$, see also Remark~\ref{rem_non_deg}. 
	$\theta^*$ is thus the only maximum of $\theta\mapsto A^{\theta/\theta^*}_t$.
	
	Secondly, we show that $:\theta \mapsto  \log L_t^{ \theta/\theta^*}$ 
	is unlikely to vary very abruptly.
	First note that $J^{N, \ts}([0, t]\times \bR)$ 
	(recall the definitions \ref{eq_def_JN} and \eqref{eq:const}) 
	is upper-bounded by
	\begin{equation*}
		J^{N, \ts}([0, t]\times \bR)\leq \bar J^{N}_t
		:= \sum_{j=1}^N \bar \pi^j\big([0, t]\times [0, B_0]\big)\,.
	\end{equation*}
	The law of $\bar \pi^j\big([0, t]\times [0, B_0]\big)$
	is Poisson with mean $t\, B_0$, for any $j$,
	and these variables are independent.
	By the strong law of large number, 
	$$\lim_{N\to\infty}\bar J^{N}_t/N=t\, B_0,\;\;\;\;\Pr-a.s.,$$ 
	where $\Pr$ is the probability on the original space  where $(\pi^j),\,j\in\mathbb{N}^*$ are defined.
	In particular,  $$\lim_{N\to\infty}\Pr(\mathcal J^N) = 1,$$ where 
	\begin{equation*}
		\mathcal J^N = \{\bar J^{N}_t/N \le 2 t\, B_0\}\,.
	\end{equation*}
	For any $\theta, \theta'\in \Theta$, 
	using the definition \eqref{eq_def_LN}, 
	on the event $\mathcal J^N$:
	\begin{equation}\label{eq_prop_unif_L}
		\begin{split}
			\frac 1N |\log  L_t^{ \theta/\theta^*}  -   \log  L_t^{ \theta'/\theta^*}|
			&\le \frac{J^{N, \ts}([0, t]\times \bR)}{N}
			\cdot \Big\|\log \frac{f_\theta}{f_{\theta'}}  \Big\|_\infty  
			+ t\, \|f_\theta - f_{\theta'}\|_\infty
			\le C_L\, |\theta - \theta'|
			\,, 
		\end{split}
	\end{equation}
	with $C_L := t\,(2 B_0 + 1)\,(B_1\vee B_L)>0$.

	Let $\delta, \eps>0$,  
	$B(\ts, \eps)$ be the ball in $\Theta$ of radius $\eps$ 
	and center $\ts,$ and 
	$$\eta := \min\big\{ |A^{\theta/\theta^*}_t |;\,
	\theta \in \Theta\setminus B(\ts, \eps)\big\}.$$
	Because of the previous observation  
	and the continuity of  $\theta \mapsto A^{\theta/\theta^*}_t$, 
	$\eta$ is positive.
	We can consider $N$ sufficiently large ($N\ge N_0$) to ensure 
	$\Pr([\mathcal J^N]^c) \le \delta/2$,
	and so restrict our analysis to the event $\mathcal J^N$.
	Let $r = \eta/(4C_L)>0$.
	Because $\Theta$ is compact,
	there exists a finite sequence $(\theta_\ell)_ {\ell \in \II{1, L}}$
	of elements of $\Theta \setminus B(\ts, \eps)$
	such that the balls $\big(B(\theta_\ell, r)\big)_{\ell}$ cover $\Theta \setminus B(\ts, \eps)$.
	For each $\ell \in \II{1, L}$, \eqref{eq_conv_MNT} and \eqref{eq_conv_ANT}, entail that
	\begin{equation*}
		\frac 1N \log L_t^{ \theta_\ell/\theta^*}\big(X^{N, \ts}_{[0,t]}\big)
		{\longrightarrow} A^{\theta_\ell/\theta^*}_t \; \; 
		\mbox{in}\; \; \Pr^{\ts}_N-\mbox{probability, }  \mbox{as } N\to \infty.
	\end{equation*}
	We can therefore ensure
	that $\Pr^{\ts}_N\left(\mathcal A^N \right) \geq 1 - \frac{\delta}{2}$
	holds for any $N$ sufficiently large ($N\ge N_1$),
	where
	\begin{equation*}
		\mathcal A^N := 	
		\Big\{	\forall \ell \in \II{1, L},\,
		\left| \frac 1N \log L_t^{ \theta_\ell/\theta^*}\big(
		X^{N, \ts}_{[0,t]}\big) - A^{\theta_\ell/\theta^*}_t
		\right| \leq \frac{\eta}{4}
		\Big\}	\,.
	\end{equation*}
	Remark also that $A^{\theta_\ell/\theta^*}_t  < -\eta$, 
	for any $\ell\in \II{1, L}$ by definition of $\eta$ 
	and the sign of $A^{\theta/\theta^*}_t$.
	
	Let $\theta \in \Theta \setminus B(\ts, \eps)$. There exists $\ell \in \II{1, L}$ such that $\|\theta - \theta_\ell\| \leq r$.
	On the event $\mathcal J^N\cap \mathcal A^N$, we have, by the Lipschitz property \eqref{eq_prop_unif_L}, that
	\begin{equation*}
		\begin{split}
			\frac 1N \left| \log L_t^{ \theta/\theta^*} - \log L_t^{ \theta_\ell/\theta^*} \right|
			&\leq C_L\, \|\theta - \theta_\ell\| \leq C_L\, r = \frac{\eta}{4}\,.
		\end{split}
	\end{equation*}
	Therefore, for any $\theta \in \Theta \setminus B(\ts, \eps)$
	\begin{equation*}
		\begin{split}
			\frac 1N \log L_t^{ \theta/\theta^*}
			&\leq A^{\theta_\ell/\theta^*}_t  
			+ \frac{\eta}{4} + \frac{\eta}{4}< - \frac{\eta}{2}\,.
		\end{split}
	\end{equation*}
	On the other hand, $\log L_t^{ \theta^*/\theta^*} = 0$ by definition.
	Since $\hat\theta^N_t$ is the maximizer of $L_t^{\theta/\theta^*}$, we thus have:
	\begin{equation*}
		\frac 1N \log L_t^{ \hat\theta^N_t/\theta^*} \geq  0> \sup\Big\{\frac 1N \log L_t^{ \theta/\theta^*};\, \theta \in  \Theta \setminus B(\ts, \eps) \Big\}\,,
	\end{equation*}
	on the event $\mathcal J^N\cap \mathcal A^N$ for $N$ sufficiently large ($N\ge N_0\vee N_1$).
	It then follows that: 
	\begin{equation*}
		\Pr^{\ts}_N\big( |\hat\theta^N_t - \ts| \geq \eps \big) 
		\le \Pr\big([\mathcal J^N]^c\big)
		+ \Pr\big([\mathcal A^N]^c\big) \le \delta\,.
	\end{equation*}
	As $\eps, \delta$ are freely chosen, this concludes the proof of the consistency.
\end{proof}


\subsection{Limit distribution and local asymptotic minimax optimality of MLE}
In this section we show the asymptotic normality and the local asymptotic minimax optimality of MLE. The proof follows the arguments of  \cite[$(7.12)$]{Ho14}. In particular we show that MLE satisfies the "coupling property" of \cite[$(7.9)$]{Ho14}.
\subsubsection{Proof of Theorem \ref{thm_minimax}}

\begin{proof}
	Denote by $\ell^N[\theta]=\log L^{\theta/\ts}_t $, then using Taylor formula and the fact that $\hat \theta^N$ is the argmax of $\ell^N$:
	\begin{equation}\label{eq_dev_dl}
		- \dot\ell^N[\ts] = \Big[\int_0^1  \ddot\ell^N[\ts+\xi\cdot(\thN-\ts)]\,\Rd \xi\Big]\cdot (\thN - \ts),
	\end{equation}
	where $\dot\ell^N$ and $\ddot\ell^N$
	denote respectively the gradient vector and the Hessian matrix 
	of $\ell^N$.
	We identify that $N^{-1/2} \dot\ell^N[\ts]$ is equal to the term $	\Delta_t^{N,\theta^*} $ from the LAN decomposition, given by \eqref{eq:decM}, and recall that this martingale satisfies $<	\Delta^{N,\theta^*} >_t = I^{N, \ts}_t$.
	The convergence of the later to $I^{\ts}_t$
	and the asymptotic normality of $	\Delta_t^{N,\theta^*} $
	is justified in Proposition~\ref{prop_aux}.
	
	On the other hand, we identify
	\begin{equation*}
		\frac{\ddot\ell^N[\ts]}{N}
		= -I^{N,\ts}_t
		+ \wht M^{N,\ts}_t\,,
	\end{equation*}
	in which
	\begin{equation*} 	
		\wht M^{N,\ts}_t
		= \frac 1 {N} \int_0^t \int_{\mathbb R}\Big(\frac{\ddot{f}_\ts}{f_{\ts}}  
		- \big(\frac{\dot{f}_\ts}{f_{\ts} }\big)^{\otimes 2} \Big)(x)\; \wtd{J}^N_J (\Rd s, \Rd x),
	\end{equation*}
	so that $(\wht M^{N,\ts}_s)_{s\ge 0}$ defines a martingale with values in $\mathcal{M}^d$.
	The squared Frobenius norm of $\wht M^{N,\ts}_t$, denoted $\|\wht M^{N,\ts}_t\|_F^2$, can then be bounded in expectation,
	with the following identity:
	\begin{equation}
		\bE_{\Pr_N^{\ts}}\ \big(\|\wht M^{N,\ts}_t\|_F^2\big)
		=  \frac 1N\int_0^t \mu^{N, \ts}_s \Big[
		\Big\|\Big(\frac{\ddot{f}_\ts}{f_{\ts}} 
		-  \big(\frac{\dot{f}_\ts}{f_{\ts} }\big)^{\otimes 2} \Big)\Big\|_F^2\cdot f_{\ts}
		\Big]\; \Rd s.
	\end{equation}
	which is obtained by expressing the expectation as the sum of the quadratic variations of the individual matrix entries
	\begin{equation*}
		\sum_{k, \ell = 1}^d \bE_{\Pr_N^{\ts}}\big(\langle\wht M^{N,\ts}[k, \ell]\rangle_t\big)
		= \sum_{k, \ell = 1}^d \frac 1N\int_0^t \mu^{N, \ts}_s \Big[
		\Big(\Big(\frac{\ddot{f}_\ts}{f_{\ts}} 
		-  \big(\frac{\dot{f}_\ts}{f_{\ts} }\big)^{\otimes 2} \Big)[k, \ell]\Big)^2\cdot f_{\ts}
		\Big]\; \Rd s\,.
	\end{equation*}
	Then, uniformly over $x\in \bR$, we have that
	\begin{equation*}
		\Big\|\Big(\frac{\ddot{f}_\ts}{f_{\ts}}(x)
		-  \big(\frac{\dot{f}_\ts}{f_{\ts} }\big)^{\otimes 2} \Big)(x)\Big\|_F^2
		\le 2 \Big( \Big\|\frac{\ddot{f}_\ts}{f_{\ts}}(x) \Big\|_F^2 
		+ \Big\|\Big(\frac{\dot{f}_\ts}{f_{\ts} }\Big)^{\otimes 2}(x) \Big\|_F^2\Big) 
		\le B_2^2 + B_1^4<\infty\,,
	\end{equation*}
	where $B_2 = \sup_{\theta, x} \Big\|\ddot{f}_\theta(x)/f_{\theta}(x) \Big\|_F^2$ is finite under Assumptions~\ref{ass:dif} and \ref{ass_ddf},
	which entails that $\bE_{\Pr_N^{\ts}}\big(\|\wht M^{N,\ts}_t\|_F^2\big)$
	tends to 0 as $N\to \infty$.
	In particular, $\wht M^{N,\ts}_t$ tends to 0 in probability.
	Though the estimate is uniform in $\theta$, $\thN$ depends on the whole behavior of the Poisson processes, so that it cannot directly be applied to the sequence $\thN$ for instance.
	
	On the other hand, we see as in \eqref{eq_prop_unif_L}
	the following upper-bound on the event $\mathcal{J}^N$,
	that was defined in \eqref{eq_def_JN}:
	\begin{equation*}
		\begin{split}
			\frac{\|\ddot\ell^N[\theta] - \ddot\ell^N[\ts]\|_F^2}{N}
			\le &4 B_0\, t\, \sup_x \Big\|\Big(\frac{\ddot{f}_\theta}{f_{\theta}} - \frac{\ddot{f}_\ts}{f_{\ts}}\Big) (x) \Big\|_F^2
			+ \Big\|\Big(\big(\frac{\dot{f}_\theta}{f_{\theta} }\big)^{\otimes 2}
			- \big(\frac{\dot{f}_\ts}{f_{\ts} }\big)^{\otimes 2}\Big)(x) \Big\|_F^2
			\\&+ B_0\, t \sup_x \Big\|\frac{\ddot{f}_\theta - \ddot{f}_\ts}{f_{\ts}} (x) \Big\|_F^2\,,
		\end{split}
	\end{equation*}
	which tends to 0, as $\theta$ tends to $\ts$,  by the uniform continuity in $\ts$ of $\ddot{f}_\theta$, $\dot{f}_\theta$ and $f_\theta$.
	Let 
	\begin{equation*}
		\wht I^{N, \ts}_t
		= \frac{-1}N\int_0^1\ddot\ell^N[\ts+\xi\cdot(\thN-\ts)]
		\,\Rd \xi\,.
	\end{equation*}
	Since $\thN$ is a consistent estimator, $\wht I^{N, \ts}_t$ 
	is therefore equal to $I^{N, \ts}_t$ up to a term that converges in $\Pr^{\ts}_N$-probability to 0.
	So, recalling Proposition~\ref{prop_aux}, 
	we deduce that the following convergence in law under $\Pr^{\ts}_N$ holds that
	$$(\wht I^{N, \ts}_t, 	\Delta_t^{N,\theta^*} )
	\stackrel{\mathcal{L}}{\longrightarrow} (I_t^\ts, (I_t^\ts)^{1/2} W_t),$$ 
	as $N\to \infty$. We can then apply the continuous mapping theorem to the function $(x,y)\mapsto - \idc{\det(x)\neq 0}x^{-1}\, y$
	(because  $\det(I_t^\ts) \neq 0$ so that the set of discontinuity points is negligible for the limiting distribution).
	Using \eqref{eq_dev_dl}, we  deduce  that $N^{-1/2} (\thN - \ts)$ converges in distribution 
	to $\mathcal N\big(0,  (I_t^\ts)^{-1}\big)$, under $\Pr^{\ts}_N$, as $N\to \infty$.

	In addition, because $\wht  I^{N, \ts}_t -I^{N, \ts}_t $ converges in $\Pr^{\ts}_N$-probability to $0$,
	we can state the LAN property \eqref{eq:LAN} with $\wht  I^{N, \ts}_t$ instead of $I^{N, \ts}_t$.
	With \eqref{eq_dev_dl},
	\cite[Theorem 7.12(ii)]{Ho14} 
	thus proves the local asymptotic minimax bound in Theorem~\ref{thm_minimax}.
\end{proof}
%

\subsection{Optimality properties of the maximum likelihood estimator in the absence of large jumps.}
%


%
%
%
\subsubsection{Proof of the optimality properties of the MLE
	in the absence of large jumps. }
\label{sec_IH}
\hfill \\
In this section we show how
the results of Theorem~\ref{thm_optimality} are deduced from \cite[Section III.1]{IH13}
by adapting the approach proposed in \cite{dMH23}.
They rely on the next  two estimates
in combination with the LAN property.

Let us denote  $\lambda_{\theta^*}^N[\hDev]
= \ell^N[\theta^*+ \frac{\hDev}{\sqrt{N}}]
=\log L^{(\theta^*+ \hDev/\sqrt{N})/\theta^*}_t$,
for any $\hDev\in \bR^d$ such that $\theta^*+ \hDev/\sqrt{N}\in \Theta$.

\begin{lemma}\label{lem_reg_lklh}
	Under Assumptions~\ref{ass:finite} and 
	\ref{ass:dif}, for  any $r\ge d$,
	there exists $C>0$ such that the following holds for any $\hDev, \mathrm \hDev$ and any $\ts\in \Theta$:
	\begin{equation*}
		\bE_{\Pr^N_{\theta^*}}\Big[\Big|\exp\Big(\frac{\lambda_{\theta^*}^N[\hDev]}{r} \Big) - \exp\Big(\frac{\lambda_{\theta^*}^N[\mathrm \hDev]}{r} \Big)\Big|^r\Big]
		\le C\cdot \|\hDev-\mathrm \hDev\|^r\,.
	\end{equation*}	
\end{lemma}

\begin{lemma}\label{lem_mom_lklh}
	Suppose $\varphi \equiv 0$
	and Assumptions~\ref{ass:theta} to \ref{ass_non_deg} and \ref{ass_any_moment}.
	Then, for 
	any $r>0$, there exists $C>0$
	such that the following holds 
	for any $\ts\in \Theta$ 
	and any $\hDev$ such that $\theta^*+ \frac{\hDev}{\sqrt{N}} \in \Theta$:
	\begin{equation}
		\bE_{\Pr^N_{\theta^*}}\Big[\exp\Big(\frac{\lambda_{\theta^*}^N[\hDev]}{2} \Big)\Big]
		\le C\, \|\hDev\|^{-r}\,.
		\label{eq_prop_bound_L}
	\end{equation}	
\end{lemma}

At the moment, we can only proceed with \eqref{eq_prop_bound_L} in the case where $\varphi \equiv 0$, which is thus assumed in this Section~\ref{sec_IH}.
This assumption appears typically needed for the following quantitative result of propagation of chaos.
\begin{lemma}\label{lem_moment_ctrl}
	Suppose $\varphi \equiv 0$
	and Assumptions~\ref{ass:finite} to \ref{ass:dif} and \ref{ass_any_moment}.
	For any $r>2$,  the following holds 
	for some $C>0$ that only depends on $r$, $t$  and 
	$\bE[|X_0^{1}|^{2r}]$
	\begin{equation*}
		\sup_{s\in [0, t]}
		\bE_{\Pr^N_\ts}[\mathcal W_1(\mu^{N, \ts}_s, \bar \mu^ {\ts}_s)^r]
		\le \frac{C}{N^{r/2}}.
	\end{equation*}
\end{lemma}
Let $(\bar X^{i, N, \ts}_{[0, t]})_{i\in \II{1, N}}$ be a family of replicates of the Mc Kean-Vlasov dynamics, with initial condition specified by $\bar X^{i, N, \ts}_0 = X^{i, N, \ts}_0=X^{i}_0$, and for any $s\in  [0, t]$:
\begin{equation*}
	\bar \mu_s^{N, \ts}(\Rd x)
	:= \frac1N \sum_{i = 1}^N \delta_{\bar X^{i, N, \ts}_s}(\Rd x)\,.
\end{equation*}
\begin{remark}
	The idea is to exploit a diagonal coupling with the same Poisson measure for the $j$-th particle
	and control together $|X^{i, N, \ts}_s - \bar X^{i, N, \ts}_s|^r$ and $\mathcal W_1(\mu^{N, \ts}_s, \bar \mu_s^{N, \ts})^r$
	by means of the Gr\"onwall Lemma.
	More precisely, 
	the Wasserstein norm appears as the natural quantity to bound the
	predictible part of the interaction term,
	together with a correction term which we handle with concentration inequalities taken from \cite{FG15}, 
	while the martingale terms are controlled
	by the Burckholder-Davis-Gundy inequality.	
	This last control fails to produce satisfying bounds
	if $\varphi \neq 0$,
	due to the dependency of the jump rate 
	on the state of the particle.
	We conjecture that such a strong result actually does not hold 
	in this case where  $\varphi \neq 0$.
\end{remark}

\begin{proof}[Proof of Lemma~\ref{lem_reg_lklh}]
	We start the proof 
	by noticing that
	\begin{equation}
		\label{eq_prop_drv_lbd}
		\begin{split}
			\bE_{\Pr^N_{\theta^*}}\Big[\Big|\exp\Big(\frac{\lambda_{\theta^*}^N[\hDev]}{r} \Big) - \exp\Big(\frac{\lambda_{\theta^*}^N[\mathrm \hDev]}{r} \Big)\Big|^r\Big]
			&= \bE_{\Pr^N_{\theta^*}}\Big[\|\hDev - \mathrm \hDev\|^r\cdot
			\Big\| \int_0^1	\partial_{\hDev}\exp\Big(\frac{\lambda_{\theta^*}^N[.]}{r} \Big)[\hDev_\xi]\, \Rd \xi \Big\|^r\Big]
			\\&\le  \|\hDev - \mathrm \hDev\|^r\,
			\int_0^1	\bE_{\Pr^N_{\theta^*}}\Big[\Big\|\partial_{\hDev}\exp\Big(\frac{\lambda_{\theta^*}^N[.]}{r} \Big)[\hDev_\xi]\Big\|^r\Big]\, \Rd \xi,
		\end{split}
	\end{equation}
	thanks to  the Jensen inequality and the Fubini theorem,
	where $\hDev_\xi = \hDev + \xi\, (\mathrm \hDev -  \hDev)$. Using the chain rule and the relation between $\lambda^N$ and $\ell^N$ we get 
	\begin{equation*}
		\Big\|\partial_{\hDev}\exp\Big(\frac{\lambda_{\theta^*}^N[.]}{r} \Big)[\hDev_\xi]\Big\|^r
		= \exp\bigg(\ell^N[\theta_\xi] \bigg)\cdot
		\bigg(\frac{1}{r\,\sqrt{N}}\bigg)^r\,
		\|\partial_{\theta} \ell^N[\theta_\xi]\|^r\,,
	\end{equation*}
	where $\theta_\xi = \ts + \frac{\hDev_\xi}{\sqrt{N}}$.
	For $\theta$, we recall also that $\exp\circ\ell^N[\theta] 
	= L^{\theta/\theta^*}_t$
	is the density of  $\Pr_N^{\theta}$ with respect to $\Pr_N^{\ts}$ on the time-interval $[0, t]$, 
	so that for any $\xi\in [0, 1]$:
	\begin{equation}\label{eq_prop_dec_lN}
		\bE_{\Pr_N^{\ts}}\Big[\Big\|\partial_{\hDev}\exp\Big(
		\frac{\lambda_{\theta^*}^N[.]}{r} \Big)[\hDev_\xi]\Big\|^r\Big]
		=  \bigg(\frac{1}{r\,\sqrt{N}}\bigg)^r\,
		\bE_{\Pr_N^{\theta_\xi}}\bigg[\|\partial_{\theta} \ell^N[\theta_\xi]\|^r\bigg]\,.
	\end{equation}
	Under  $\Pr_N^{\theta}$,  using \eqref{eq_def_JN} and \eqref{eq_def_JNcomp}, we identify
	the law of $\partial_{\theta} \ell^N[\theta]$
	from \eqref{eq_def_LN}
	as the one of:
	\begin{equation*}
		\int_{ [0, t ] \times \bR} \frac{ \dot f_{\theta}(x) }{  f_{\theta}(x) }\; J^{N, \theta} (\Rd s, \Rd x) 
		- \int_0^t \frac{ \dot f_{\theta}(x)  }{ f_{\theta^*}(x)}\,   \varUpsilon^{N, \theta/\ts} (\Rd s, \Rd x)\,,
	\end{equation*}
	where $\varUpsilon^{N, \theta/\ts}(\Rd s, \Rd x)
	= N\, f_{\ts}(x)\,
	\mu_s^{N, \theta}(\Rd x)\, \Rd s$.  Therefore
	$N^{-1/2}\partial_{\theta} \ell^N[\theta]$ has  the law of $\ \Delta^{N,\theta}_t$
	where $(\Delta^{N,\theta}_s)_{s\in [0, t]}$
	is a martingale defined as follows, in analogy with \eqref{eq_def_MN1},
	\begin{equation*}
		\begin{split}
			\Delta^{N,\theta}_t
			&= \frac1{\sqrt{N}}\iint_{[0,t]\times \bR}\frac{\dot{f}_{\theta} }{{f}_{\theta}} (x)\;
			\wtd J^{N,\theta}  ( \Rd s, \Rd x)
			\\&=\frac1{\sqrt{N}}\sum_{j=1}^N 
			\iint_{[0,t]\times \bR^+}
			\frac{\dot{f}_{\theta} }{{f}_{\theta}} \big(X^{j, N,\theta}_{s- }\big) 
			\idc{ z \le f_{\theta} \big( X^{j, N,\theta}_{s- }\big)} 
			\wtd{ \bar \pi}^j (\Rd s, \Rd z )\,.
		\end{split}
	\end{equation*}
	We can thus apply the Burkholder-Davis-Gundy inequality to the martingale $\Delta^{N,\theta}_t$
	with $r\ge 2$,   see	\cite{kallenberg}, Theorem 20.12, 
	then the Jensen inequality,
	and deduce that:
	\begin{equation*}
		\begin{split}
			&\bE_{\Pr_N^{\theta}}\Big[|N^{-1/2}\partial_{\theta} \ell^N[\theta]|^r\Big]
			=\bE_{\Pr_N^{\theta}}\Big[|\Delta^{N,\theta}_t|^r\Big]
			\\&\quad \leq
			\bE_{\Pr_N^{\theta}}\Big[[\Delta^{N,\theta}_t]^{r/2}\Big]
			= \bE_{\Pr_N^{\theta}}\Big[\Big|\frac 1N \sum_{j=1}^N 
			\iint_{[0,t]\times \bR^+}
			\Big (\frac{\dot{f}_{\theta} }{{f}_{\theta}} \big(X^{j, N,\theta}_{s- }\big) \Big)^2
			\idc{ z \le f_{\theta} \big( X^{j, N,\theta}_{s- }\big)} 
			{ \bar \pi}^j (\Rd s, \Rd z )   \Big|^{r/2}\Big] 
			\\&\quad \leq  B_1^r\; \bE_{\Pr_N^{\theta}}\Big[(\bar\pi^1([0,t]\times[0,B_0])^{r/2}\Big]=C(t,B_0,B_1)<\infty.
		\end{split}
	\end{equation*}
	Recalling in addition \eqref{eq_prop_drv_lbd} and \eqref{eq_prop_dec_lN},
	we conclude Lemma~\ref{lem_reg_lklh} 
	with $C= r^{-r}\cdot C_M^{r/2}$.
\end{proof}

\begin{proof}[Proof of Lemma~\ref{lem_moment_ctrl}]
	Recall that  $(\bar X^{i, N, \ts}_{[0, t]})_{i\in \II{1, N}}$ 
	denotes a family of replicates of the Mc Kean-Vlasov dynamics, 
	with initial condition specified 
	by $\bar X^{i, N, \ts}_0 = X^{i, N, \ts}_0=X^{i, \ts}_0$, 
	and for any $s\in  [0, t]$:
	\begin{equation*}
		\bar \mu_s^{N, \ts}(\Rd x)
		:= \frac1N \sum_{i = 1}^N \delta_{\bar X^{i, N, \ts}_s}(\Rd x)\,.
	\end{equation*}
	Since $$\mathcal W_1(\mu^{N, \ts}_s, \bar \mu_s^\ts)
	\le \mathcal W_1(\mu^{N, \ts}_s, \bar\mu^{N, \ts}_s) 
	+  \mathcal W_1(\bar\mu^{N, \ts}_s, \bar \mu^{\ts}_s),$$
	it holds that
	\begin{equation}\label{eq_ctrl_W1}
		2^{-r+1}\bE_{\Pr^N_\ts}[\mathcal W_1(\mu^{N, \ts}_s, \bar \mu_s^{\ts})^r]
		\le \bE_{\Pr^N_\ts}[\mathcal W_1(\mu^{N, \ts}_s, \bar\mu^{N, \ts}_s)^r]
		+ \bE_{\Pr^N_\ts}[\mathcal W_1(\bar\mu^{N, \ts}_s, \bar \mu^{\ts}_s)^r]\,.
	\end{equation}
	Recall that $(\bar X^{i, N, \ts}_{[0, t]})_{i\in \II{1, N}}$
	follows the Mc-Kean Vlasov dynamics driven by $\bar\mu^{\ts}_{[0, t]}$
	while   $\bar X^{i, N, \ts}_0$ are i.i.d. samples 
	with distribution $\bar \mu^{\ts}_0$.
	Thus, for any $s\in [0, t]$, 
	$\bar X^{i, N, \ts}_s$ are i.i.d. samples 
	with distribution $\bar \mu^{\ts}_s$.
	
	Let us denote by  $M_r(\ts)$ the upper-bound 
	on the moment of order $2r$ of $\bar \mu^{\ts}_s:$ 
	$$ M_r(\ts):=\bE\Big[\sup_{s\leq t}|\bar X^{1,N,\ts}_s|^{2r}\Big].$$
	Since
	\begin{align*}
		\sup_{s\leq t}|\bar X^{1,N,\ts}_s|^{2r}
		&\leq 
		\big(2^{2r-1}+(16t)^{2r-1}\|b\|_{Lip}^{2r}\big)\,|X_0^{1}|^{2r}+
		(8t)^{2r}\|b\|_{Lip}^{2r}\int_0^T\sup_{u\leq s}|\bar X^{1,N,\ts}_u|^{2r} \Rd s
		\\ &\quad +t^{2r}8^{2r-1}\big(|b(X_0^1)|^{2r}+|m|\sup_{\ts\in\Theta}\|f_{\ts}\|_{\infty})^{2r}\big) .
	\end{align*}
	Using Gr\"onwall lemma, for all $\ts\in\Theta$  we deduce that 
	\begin{equation}\label{eq_moments}
		M_r (\ts)
		\leq C\Big(t,r,m,\sup_{\theta\in\Theta}\|f_{\theta}\|_{\infty},
		\|b\|_{Lip},\bE\Big[| X^1_0|^{2r}\Big]\Big)
		<\infty
	\end{equation} 
	and since the previous bound doesn't depend on $\ts,$ it is uniform: 
	$\sup_{\ts\in\Theta} M_r(\ts)\leq C.$
	\

	We can therefore exploit \cite[Theorem 2]{FG15} and deduce that there exists $C_c$ depending only on $r$ and $M_r$ such that
	\begin{equation}\label{eq_FG_iid_moment}
		\sup_{s\in [0, t]} \bE_{\Pr^N_\ts}[\mathcal W_1(\bar\mu^{N, \ts}_s, \bar \mu^{\ts}_s)^r]
		\le  \frac{C_c}{N^{r/2}}\,. 
	\end{equation}
	
	Our aim now is to show a similar bound 
	for $\mathcal W_1(\mu^{N, \ts}_s, \bar\mu^{N, \ts}_s)$ 
	in \eqref{eq_ctrl_W1}.   
	Using  the diagonal coupling and the Jensen inequality, we deduce that
	$$\mathcal W_1(\mu^{N, \ts}_s, \bar\mu^{N, \ts}_s)
	\leq \frac1N \sum_{i \le N} 
	|\bar X^{i, N, \ts}_s - X^{i, N, \ts}_s|,$$
	and
	\begin{equation}\label{eq_prop_W1_Diag_cpl}
		\bE_{\Pr^N_\ts}[\mathcal W_1(\mu^{N, \ts}_s, \bar\mu^{N, \ts}_s)^r]
		\le \frac1N \sum_{i \le N} \bE_{\Pr^N_\ts}[| X^{i, N, \ts}_s - \bar X^{i, N, \ts}_s|^r]\,.
	\end{equation}
	
	There exists $C_d>0$ depending only on $r$ and $t$ such that for any $i\in \II{1, N}$
	\begin{equation*}
		\frac{1}{C_d}|X^{i, N, \ts}_s - \bar X^{i, N, \ts}_s|^r
		\le \int_0^s (\mathcal B^{i, N, \ts}_v)^r \Rd v 
		+ (\mathcal D^{i, N, \ts}_s)^r
		+ |\mathcal N^{i, N, \ts}_s|^r\,,
	\end{equation*}
	where
	\begin{equation*}
		\mathcal B^{i, N, \ts}_s = |b(X^{i, N, \ts}_s) - b(\bar X^{i, N, \ts}_s) |+ |\sum_{j= 1}^N \int_{\bR_+} \frac{u}{N} f_\ts(X^{j, N, \ts}_s) \nu(\Rd u) - m\, \,\bar \mu^{\ts}_s[f_{\ts}]|,
	\end{equation*}
	so that
	\begin{equation}\label{eq_prop_cB}
		2^{-r+1} (\mathcal B^{i, N, \ts}_s)^r
		\le \| b \|_{Lip}^r
		\cdot |X^{i, N, \ts}_s - \bar X^{i, N, \ts}_s|^r+ m^r\,\| f_\ts \|_{Lip}^r\, \;W_1(\mu^{N, \ts}_s, \bar \mu^{\ts}_s)^r\,,
	\end{equation}
	with $\|.\|_{Lip}$ denoting the Lipschitz norm, while 
	\begin{equation}\label{eq_def_cD}
		\begin{split}
			\mathcal D^{i, N, \ts}_s
			&= \big|\int_0^s \int_{\bR} \frac{u}{N} f_\ts(X^{i, N, \ts}_v) \nu(\Rd u)\, \Rd v\big|
			\le \frac{m\,s\, B_0}{N}\,
		\end{split}
	\end{equation}
	and  $\mathcal N^{i, N, \ts}_{[0, t]}$ is a martingale starting from 0 defined, for any $s\in [0, t]$, as follows
	\begin{equation*}
		\Rd \mathcal N^{i, N, \ts}_s
		= \sum_{j\neq i}  \int_{\bR_+} \int_{\bR} \frac{u}{N} \idc{z\le f_{\ts} ( X_{s-}^{j, N,\ts } ) }  \wtd \pi^j (\Rd s, \Rd z, \Rd u )
		\,.
	\end{equation*}
	Thanks to the Burkholder-Davis-Gundy inequality 
	see	\cite{kallenberg}, Theorem 20.12,
	then the Jensen inequality,
	\begin{equation}\label{eq_BDG_ineq}
		\begin{split}
			\bE_{\Pr^N_\ts}\Big[ \sup_{s\in [0, t]}|\mathcal N^{i, N, \ts}_s|^r\Big]
			&\le \bE_{\Pr^N_\ts}\Bigg[\Bigg(\frac 1N\sum_{j\neq i}\int_{[0,t]}  \int_{\bR_+} \int_{\bR} 
			\frac{u^2}{\sqrt N} \idc{z\le f_{\ts} ( X_{s-}^{j, N,\ts } ) } 
			\pi^j (\Rd s, \Rd z, \Rd u ) \Bigg)^{r/2}\Bigg]
			\\ &\le N^{-r/2}\,
			\bE_{\Pr^N_\ts}\Bigg[\frac 1N\sum_{j\neq i}\Bigg (\int_{[0,t]}  \int_{\bR_+} \int_{\bR} 
			{u^2} \idc{z\le f_{\ts} ( X_{s-}^{j, N,\ts } ) }  
			\pi^j (\Rd s, \Rd z, \Rd u ) \Bigg)^{r/2}\Bigg]
			\\ &\le N^{-r/2}\,
			\bE_{\Pr^N_\ts}\Bigg[\Bigg(\int_{[0,t]}  \int_{\bR_+} \int_{\bR} 
			{u^2} \idc{z\le f_{\ts} ( X_{s-}^{j, N,\ts } ) }  
			\pi^1 (\Rd s, \Rd z, \Rd u ) \Bigg)^{r/2}\Bigg]
			\\ &\le N^{-r/2}\,
			\bE_{\Pr^N_\ts}\Bigg[\Bigg(\int_{[0,t]}  \int_{[0,B_0]} \int_{\bR} 
			{u^2}   \pi^1 (\Rd s, \Rd z, \Rd u ) \Bigg)^{r/2}\Bigg]
			\le C{N^{-r/2}},
			\,
		\end{split}
	\end{equation}
	as the compound Poisson process in the last line 
	admits the moments of any order under Assumption~\ref{ass_any_moment} $(ii).$

	We can thus justify a  Gr\"onwall   approach
	on the quantity 
	\begin{equation*}
		\mathcal G^{N, \ts}_s
		=\frac1N\sum_{i \le N} \bE_{\Pr^N_\ts}[|\bar X^{i, N, \ts}_{s} - X^{i, N, \ts}_{s}|^r],
	\end{equation*}
	as $s$ varies from $0$ to $t$.
	The Wasserstein term in \eqref{eq_prop_cB}
	is to be upper-bounded as in our initial steps, 
	by recalling \eqref{eq_ctrl_W1},
	\eqref{eq_prop_W1_Diag_cpl} 
	and 
	\eqref{eq_FG_iid_moment}.
	The martingale term 
	is to be treated 
	thanks to \eqref{eq_BDG_ineq}
	together with the estimates given in \eqref{eq_def_cD}
	and the one associated with \eqref{eq_FG_iid_moment}.
	The leading error term is the one of the martingale, of order $N^{-r/2}$.
	Putting all these estimates together,
	for any sufficiently large $N$, we arrive at the following upper-bound:
	\begin{equation*}
		\mathcal G^{N, \ts}_s
		\le \frac{2 C_d\,C_\mathcal{N}}{N^{r/2}}
		+ C_d\,C_{\mathcal G}\, \int_0^s \mathcal G^{N, \ts}_v\, \Rd v\,,
	\end{equation*}
	where
	$C_{\mathcal G} = \| b \|_{Lip}^r
	+ 4^{r-1}m^r\,  B_L^r$.
	The Gr\"onwall Lemma implies that there exists $C>0$ such that $\sup_{s\in [0, t]} \mathcal G^{N, \ts}_s \le C/N^{r/2}$.
	The proof of Lemma~\ref{lem_moment_ctrl} is finally concluded with \eqref{eq_ctrl_W1} and \eqref{eq_FG_iid_moment}.
\end{proof}

We can next proceed with the proof of  \eqref{eq_prop_bound_L}.
\begin{proof}[Proof of Lemma~\ref{lem_mom_lklh}]
	
	Let $\ts\in\Theta$  be fixed and $h\in\bR^d$ 
	be such that $\theta=\ts+   \hDev/{\sqrt N}\in\Theta.$  
	Recall the notation
	$\lambda_{\theta^*}^N[\hDev]
	=\ell^N[\theta^*+ {\hDev}/{\sqrt{N}}]
	=\log L^{(\theta^*+ \hDev/\sqrt{N})/\theta^*}_t .$
	
	From \eqref{eq_def_L} we know that
	\begin{equation*}
		\ell^N[\theta]
		= P_{t}^{N, \theta/\ts} - R_{t}^{N, \theta/\ts},
	\end{equation*}
	where, for any $t\ge 0$,
	\begin{equation*}
		P_{t}^{N, \theta/\ts}
		=  \sum_{j=1}^N 
		\iint_{[0,t]\times \bR^+}
		\log\Big(\frac{f_{\theta} }{{f}_{\ts}}\Big) \big(X^{j, N,\ts}_{s- }\big) 
		\idc{ z \le f_{\ts} \big( X^{j, N,\ts}_{s- }\big)} 
		\bar \pi^j (\Rd s, \Rd z )\,
	\end{equation*}
	is the Poissonian term, while
	\begin{equation*}
		\begin{split}
			R_{t}^{N, \theta/\ts}
			&=  N \int_{[0,t]}\int_{\bR^+} ( f_{\theta  } - f_{\ts} ) \big (x\big)\,\mu^{N, \ts}_s(\Rd x)\, \Rd s
			\\&= \sum_{j= 1}^N \int_0^t ( f_{\theta  } - f_{\ts} ) \big (X^{j, N,\ts}_s\big)\, \Rd s
		\end{split}
	\end{equation*}
	refers to the jump rate.
	As we want to use the H\"older inequality to combine 
	the term proportional to $P_{t}^{N, \theta/\ts}$
	with a compensation that makes 
	the exponential a martingale,
	we introduce the following quantity:
	\begin{equation*}
		\begin{split}
			G_{t}^{N, \theta/\ts}
			&= \frac{N}3 \int_{[0,t]}\int_{\bR^+}
			\Big[\Big(\frac{f_{\theta} }{{f}_{\ts}}\Big)^{3/4}
			-1 \Big](x)\,  f_{\ts}(x)\, \mu^{N, \ts}_s(\Rd x)\, \Rd s
			\\&=  \sum_{j= 1}^N\frac{1}3 \int_{[0,t]}
			\Big[\Big(\frac{f_{\theta} }{{f}_{\ts}}\Big)^{3/4}
			-1 \Big](X^{j, N,\ts}_s)\,  f_{\ts}(X^{j, N,\ts}_s)\, \Rd s\,.
		\end{split}
	\end{equation*}
	It is chosen such that
	\begin{equation*}
		Z_{t}^{N, \theta/\ts}
		= \exp\Big[\frac{3}{4} P_{t}^{N, \theta/\ts} - 3 G_{t}^{N, \theta/\ts}\Big]\,
	\end{equation*}
	defines an exponential local martingale under $\Pr^N_{\ts}$, 
	and thus  a supermartingale, 
	in line with \cite[Section 5.2]{Ap09}.Therefore
	$$\bE_{\Pr^N_{\ts}}\Big[ Z_{t}^{N, \theta/\ts}\Big]
	\leq 1.$$
	
	By the H\"older inequality
	\begin{equation}\label{eq_prop_Ktheta}
		\begin{split}
			\bE_{\Pr^N_{\theta^*}}\Big[\exp\Big(\frac{\ell^N[\theta]}{2} \Big)\Big]
			&\le \bE_{\Pr^N_{\theta^*}}\Big[Z_{t}^{N, \theta/\ts}\Big]^{2/3}\cdot
			\bE_{\Pr^N_{\theta^*}}\Big[\exp\Big(
			\mathcal K_{t}^{N, \theta/\ts} \Big)\Big]^{1/3}
			\\&\le	\bE_{\Pr^N_{\theta^*}}\Big[\exp\Big(
			\mathcal K_{t}^{N, \theta/\ts} \Big)\Big]^{1/3}\,,
		\end{split}
	\end{equation}
	where 
	\begin{equation}\label{eq_def_Ktheta}
		\begin{split}
			\mathcal K_{t}^{N, \theta/\ts} 
			&	= 3\times 2\,G_{t}^{N, \theta/\ts} 
			- \frac{3}{2} R_{t}^{N, \theta/\ts}
			\\&	= N \int_0^t\int_{\bR^+} \mathrm{k}\Big(\frac{f_{\theta} }{{f}_{\ts}}
			\Big) \big(x\big)\,f_{\ts}(x)\,\mu^{N, \ts}_s(\Rd x)\, \Rd s\,,
		\end{split}
	\end{equation}
	and $\mathrm{k}(r) = 2 r^{3/4} - \frac{3}{2} r - \frac12$.
	Note that $\mathrm{k}(1)= \mathrm{k}'(1) = 0$, $\mathrm{k}''(1) = \frac{-3}{8}$. 
	$\mathrm{k}$ is also non-positive, hence
	\begin{equation}\label{eq_prop_Hneg}
		\exp\Big(\mathcal K_{t}^{N, \theta/\ts} \Big)\le 1\,.
	\end{equation}
	Note also that
	$$\mathrm{k}(r)=0\;\Longleftrightarrow \; r=1,$$
	and $ \mathrm{k}''(r)\leq 0,$ therefore  $\mathrm{k}$ is concave.
	
	Due to the second order Taylor expansion of $\mathrm{k}$ around 1,
	let $\eps_{\mathrm{k}}>0$ be such that $\mathrm{k}(1+r) \le \frac{-3r^2}{20}$
	for any $|r|\le \eps_{\mathrm{k}}$.
	Thanks to Assumption \ref{ass:deriv},
	then Assumption \ref{ass:dif}
	for the boundedness of $\dot f_{\ts} /f_{\ts}$,
	we deduce that there exists $\hat C, \hat \eps_{\mathrm{k}}>0$
	such that the following holds
	for any 
	$\theta \in B(\ts,  \hat \eps_{\mathrm{k}} )$
	and any $x\in \bR_+$:
	\begin{equation*}
		\mathrm{k}\Big(\frac{f_{\theta} }{{f}_{\ts}}\Big)(x)
		\le -\frac{3}{20}\Big((\theta - \ts)\eT \frac{\dot f_{\ts} }{f_{\ts}}(x)\Big)^2
		+ \hat C\cdot \|\theta - \ts\|^3\,.
	\end{equation*}
	Thus, for any such $\theta$:
	\begin{equation}\label{eq_prop_cK_DL}
		\mathcal K_{t}^{N, \theta/\ts} 
		\le -\frac{3N}{20} (\theta - \ts)\eT I_t^{N, \ts} (\theta - \ts)
		+ N\,t\,\hat C\, \|\theta - \ts\|^3\,,
	\end{equation}
	where $I_t^{N, \ts}$ is defined in \eqref{eq_def_INts}.
	Note that for a fixed value of $\hDev$, 
	using Proposition \ref{prop_aux}, 
	with $I_t^{\ts}$ given by \eqref{eq_def_I}, 
	we deduce the following limiting upper-bound:
	\begin{equation*}
		\limsup_{N\to \infty} \mathcal K_{t}^{N, (\ts+\hDev/\sqrt{N})/\ts} 
		\le  -\frac{3}{20} \, \hDev\eT I_t^{\ts} \hDev\,.
	\end{equation*}
	
	Let
	$$\mathcal H^N_{loc} = \Big\{\mathcal K_{t}^{N, (\ts+\hDev/\sqrt{N})/\ts} 
	\leq -\frac 1{ 10}\hDev\eT I_t^{\ts}\hDev \Big\} .$$
	
	On the complementary of the event $\mathcal H^N_{loc}$, because of $\eqref{eq_prop_cK_DL}$:
	\begin{equation*}
		\begin{split}
			&6 \hDev\eT I_t^{N, \ts}\hDev
			< 4\hDev\eT I_t^{\ts}\hDev
			+ 40 t\,\hat C\, \hat \eps_k\, \|\hDev\|^2,
		\end{split}
	\end{equation*}
	By possibly restricting the size of $\hat \eps_k$ (independently of $\hDev$),
	since  $(I_t^{\theta})_{\theta}$ is  uniformly  non-degenerated as noted in Remark~\ref{rem:unif-nondegen},
	we may assume that the last term in the r.h.s. is smaller than $\hDev\eT I_t^{\ts}\hDev$.
	So we  deduce that the complementary of $\mathcal H^N_{loc}$
	is included in the following rare event:
	\begin{equation*}
		\mathcal I^{N, \ts}
		=\Big\{ \hDev\eT I_t^{N, \ts}\hDev \le \frac{5}{6} \hDev\eT I_t^{\ts}\hDev \Big\}\,. 
	\end{equation*}
	Using Remark \ref{rem:unif-nondegen},
	$(I_t^{\theta})_{\theta}$ is  uniformly  non-degenerated, 
	therefore we can write, for some positives constants  $c,C,C_0$,
	\begin{equation}\label{eq_decomp}
		\begin{split}
			\bE_{\Pr^N_{\theta^*}}\Big[\exp\Big(
			&\mathcal K_{t}^{N, \theta/\ts} \Big)\Big]
			\leq  \bE_{\Pr^N_{\theta^*}}\Big[\exp\Big(
			\mathcal K_{t}^{N, \theta/\ts} \Big)\mathbf{1}_{\mathcal H^N_{ loc  }}\Big]+
			\bE_{\Pr^N_{\theta^*}}\Big[\mathbf{1}_{(\mathcal H^N_{ loc  })^c}\Big] \\
			&\leq  \exp\Big(-\frac{1}{10} \,\hDev\eT I_t^{\ts} \hDev\Big)
			+\Pr^N_{\theta^*}(\mathcal I^{N, \ts})
			\leq \exp (-c\|\hDev\|^2)+\Pr^N_{\theta^*}(\mathcal I^{N, \ts})
			\leq C\, \|\hDev\|^{-r}+C_0\, N^{-r/2}\,.
		\end{split}
	\end{equation}

	Indeed, under Assumption \ref{ass_non_deg}, using the fact that the moment of order $2 r$ of $\bar \mu^{\ts}_s$ is upper-bounded (see \eqref{eq_moments})
	uniformly in $\ts$ and $s\in [0, t]$, as a consequence of Lemma~\ref{lem_moment_ctrl} we can write
	\begin{equation}\label{eq_dev}
		\begin{split}
			\Pr^N_{\theta^*}(\mathcal I^{N, \ts})
			&\leq \Pr^N_{\theta^*} (
			\| I_t^{N, \ts}-I_t^{\ts}\|\geq C)
			\\&\leq \Pr^N_{\theta^*}\Big(\int_0^t \| \mu_s^{N,\ts}\Big[\frac{(\dot{f}_{\theta^*})^{\otimes 2}}{f_{\theta^*}}
			\Big]-\bar\mu_s^{\ts}\Big[\frac{(\dot{f}_{\theta^*})^{\otimes 2}}{f_{\theta^*}}
			\Big] \|\; \Rd s\,\geq C\Big)
			\\& \leq 
			\Pr^N_{\theta^*}\Big(B_1^2\,\int_0^t \mathcal W_1(\mu^{N, \ts}_s, \bar \mu^{\ts}_s) \; \Rd s\geq C\Big) 
			\\&  \leq 
			C(t,r)\sup_{s\in [0, t]} \bE_{\Pr^N_\ts}[\mathcal W_1(\mu^{N, \ts}_s, \bar \mu^{\ts}_s)^r]
			\le \frac{C(t,r)}{N^{r/2}}\,.
		\end{split} 
	\end{equation}
	
	As a consequence of \eqref{eq_prop_Ktheta}, \eqref{eq_decomp} 
	and \eqref {eq_dev}, 
	for any $h \in B(0, \hat \eps_{\mathrm{k}}\sqrt N  ),$ 
	the claim of  Lemma \ref{lem_mom_lklh} follows.

	We consider now the case of  
	$\theta \notin B(\ts,  \hat \eps_{\mathrm{k}} ).$ 
	Due to
	Assumption~\ref{ass_non_deg}, 
	Remark~\ref{rem_non_deg}, 
	the fact that $k(r)<0 $  for any $r\neq 1$, 
	and $f_{\ts}>0$,  
	the following quantity is strictly positive: 
	\begin{equation*}
		\mathfrak K_t^{\theta/\ts}
		= -\int_0^t\int_{\bR^+} \mathrm{k}\Big(\frac{f_{\theta} }{{f}_{\ts}}
		\Big) \big(x\big)\,f_{\ts}(x)\,\bar \mu^{\ts}_s(\Rd x)\, \Rd s>0.
	\end{equation*}
	Moreover, by continuity of $\mathfrak K_t^{\theta/\ts}$ with respect with $(\theta, \ts)$, 
	there is a positive uniform lower-bound
	\begin{equation*}
		\inf_{\theta, \ts: \|\theta - \ts\| > \hat \eps_{\mathrm k}}	\mathfrak K_t^{\theta/\ts} >0.
	\end{equation*}
	Let
	\begin{equation*}
		\mathcal H^{N, \theta/\ts}_{gl} 
		= \Big\{ \mathcal K_{t}^{N, \theta/\ts} 
		\le -\frac{N}2 \mathfrak K_t^{\theta/\ts}\Big\}\,,
	\end{equation*}
	where $\mathcal K_{t}^{N, \theta/\ts} $ is defined by \eqref{eq_def_Ktheta}.
	The complementary of $\mathcal H^{N, \theta/\ts}_{gl}$
	is  included in the following rare event:
	\begin{equation*}
		\mathfrak J^{N, \theta/\ts}
		= \Big\{-\int_0^t\int_{\bR^+} \mathrm{k}\Big(\frac{f_{\theta} }{{f}_{\ts}}
		\Big) \big(x\big)\,f_{\ts}(x)\,\mu^{N,\ts}_s(\Rd x)\, \Rd s \le \frac12 \mathfrak K_t^{\theta/\ts}\Big\}\,. 
	\end{equation*}
	Indeed, we can control its probability  with the help of Lemma~\ref{lem_moment_ctrl} 
	as follows:
	\begin{align}\label{al:hgcontrol}
		\Pr_{\ts}^N\left ( \mathfrak J^{N, \theta/\ts} \right)&
		\leq  \Pr_{\ts}^N\left ( \int_0^t\int_{\bR^+} \mathrm{k}\Big(\frac{f_{\theta} }{{f}_{\ts}}
		\Big) (x)\,f_{\ts}(x)\,\big[ \mu^{N,\ts}_s-\bar \mu^{\ts}_s \big](\Rd x)\, \Rd s 
		\geq \frac 12\mathfrak K_t^{\theta/\ts}\right)\\\nonumber 
		&\leq \Pr_{\ts}^N\left (\int_0^t\mathcal W_1(\mu^{N, \ts}_s, \bar \mu^ {\ts}_s)\Rd s >  C(t)
		\right)
		\leq C(t,r)N^{-r/2},
	\end{align}
	where in the last inequality we have used Jensen inequality and Lemma~\ref{lem_moment_ctrl}, and we denoted by  $C(t)$ and $C(t,r)$ two generic positive constants depending respectively on $t$, and on $t$ and $r.$  
	Now the result easily follows from \eqref{eq_prop_Hneg} and \eqref {al:hgcontrol}:
	\begin{equation*}
		\begin{split}
			\bE_{\Pr^N_{\theta^*}}\Big[\exp\Big(
			&\mathcal K_{t}^{N, \theta/\ts} \Big)\Big]
			\leq  \bE_{\Pr^N_{\theta^*}}\Big[\exp\Big(
			\mathcal K_{t}^{N, \theta/\ts} \Big)
			\idc{\mathcal \mathcal H^{N, \theta/\ts}_{gl}}\Big]
			+\bE_{\Pr^N_{\theta^*}}\Big[(\mathcal H^{N, \theta/\ts}_{gl})^c\Big] \\
			&\leq \exp(-\frac N 2\mathfrak K_t^{\theta/\ts})+ \Pr_{\ts}^N\left ( \mathfrak J^{N, \theta/\ts} \right)\leq C(t,r)N^{-r/2}.
		\end{split}
	\end{equation*}
	Recall that $\ts\in\theta$ and  $\theta\in\Theta$ are fixed such that $\theta=\ts+h/\sqrt N$ and 
	$\|\theta-\ts\|\geq \hat \eps_\mathrm{k}. $ Therefore from the last inequality we obtain that for all $r\geq 1,$ 
	$$ \bE_{\Pr^N_{\theta^*}}\Big[\exp\Big(
	\mathcal K_{t}^{N, \theta/\ts} \Big)\Big]\leq C(t,r)\|h\|^{-r/2},$$ 
	which completes the proof of Lemma~\ref{lem_mom_lklh}.
\end{proof}

\begin{proof}[Proof of Theorem~\ref{thm_optimality}]
	The uniform LAN property given in Theorem \ref{thm_optimality} is Condition N1 of  \cite[Chapter III]{IH13}. The non-degeneracy assumption, that $I_t^\theta>0$ for any $\theta$, corresponds to Condition N2 of \cite{IH13}. Lemmas~\ref{lem_reg_lklh} and \ref{lem_mom_lklh} are respectively Condition N3 and N4 of \cite{IH13}. 
	The conditions of \cite[Theorem III.1.1]{IH13} are thus satisfied,
	which entails statement (i) of Theorem \ref{thm_optimality}. Statement (ii) is a consequence of \cite[Corollary III.1.1]{IH13} while Statement (iii) is a  consequence of \cite[Theorem III.1.3]{IH13}.
\end{proof}

\subsection*{Declaration of generative AI and AI-assisted technologies in the manuscript preparation process} 

The authors used Mistral AI (Medium 3.5, Version 26.04) for language polishing (grammar and style), drafting initial versions of simulations and result presentations, and refining figure editing. Qwen's and ChatGPT's performance were tested,  for instance by asking for references and
by extending the draft from the one-dimensional to the multivariate setting (which revealed efficient with the former).
All scientific content, analysis, and conclusions were developed independently by the authors. The authors take full responsibility for the content of the published article.

\subsection*{Acknowledgements}

This work was produced as part of the activities of FAPESP Research, Innovation and Dissemination Center for Neuromathematics (grant \#2013/07699-0, S.Paulo Research Foundation (FAPESP)).

\bibliographystyle{cas-model2-names}

\bibliography{bib_LAN_MLE_MF.bib}

\appendix
\appendix
\section{Technical Details of Numerical Simulations}
\label{app:simulation_details}

\subsection{Numerical Methods}
\label{app:numerical_methods}

The numerical implementation closely follows the  model described in Section~\ref{sec_finite}, with discretization and algorithmic choices tailored to ensure stability and efficiency for finite-$N$ systems. Below, we detail the key numerical procedures used in the simulations and statistical analyses.

\subsubsection{Time Discretization and Spike Generation.}
The system of SDEs (\ref{eq:finite}) is approximated using a discrete-time scheme 
according to the parameters summarized in Table~\ref{tab_simulation_parameters_v16}.
The deterministic drift $b(x) = -\alpha x$ is handled via an explicit resolution at each discrete step, while spikes are generated using the Sellke construction \cite{Se83}, a standard method for simulating inhomogeneous Poisson processes. Specifically, for each neuron $i$ and time step $k$, the spike probability is determined by the instantaneous firing rate $f_\theta(X_{k \Delta t}^{(i,N,\theta)})\Delta t$.
These values are deducted iteratively for neuron $i$ from an initial value $E^i_0$ sampled (independently for each neuron) according to a standard exponential distribution,
until negative values are reached.
This event corresponds to the spike of the neuron, at which time a new sample of $E^i$ is generated according to the exponential distribution.
The construction is such that the remaining value of $E^i$ at each time step remains exponentially distributed, ensuring the desired jump rate while keeping the procedure of random generation at the same amount than the number of jumps
plus the number of neurons without any spike.

\begin{table}[htbp]
	\begin{tabular}{|lcl|p{0.5\textwidth}|}
		\hline
		& \textbf{Symbol}       & \textbf{Value} & \textbf{Description} \\
		\hline
		\multicolumn{3}{|c|}{	\textbf{Model Parameters}          }                      & \\
		& $N$                   & $100$               & Number of neurons. \\
		
		& $\alpha$               & $1.5$               & Drift coefficient in $b(x) = -\alpha x$. \\
		& $\varphi(x)$          & $-x$                & Reset function, \hfill\break applied to a neuron's state after it spikes. \\	
		& $\theta^*_0$             & $2.5$               & Saturation level of the firing rate function. \\
		& $\theta^*_1$             & $1.0$               & Inflection point of the firing rate function. \\
		& $\theta^*_2$             & $3.0$               & Slope at inflection point of the firing rate function. \\
		
		& $\nu$                 & $\mathcal{U}(\nu_1 - \nu_2, \nu_1 + \nu_2)$ & Distribution of jump sizes. \\
		& $\nu_1$                & $2.4$               & Base effect of a spike on other neurons. \\
		& $\nu_2$                & $ 0.8=  \nu_1 / 3$ & dispersion of the effect. \\
		\hline
		\multicolumn{3}{|c|}{		\textbf{Simulation Settings}}                & \\
		& $T$                   & $10.0$ s             & Final observation time. \\
		& $\Delta t$            & $0.002$ s            & Time step size. \\
		& $N_{\text{steps}}$  & $5000$               & Number of time steps ($T / \Delta t$). \\
		& $X_0^i$              & $\mathcal{U}(0, 4)$  & Initial state distribution for each neuron. \\
		\hline
		\multicolumn{3}{|c|}{		\textbf{Optimization Parameters}}               & \\
		& $\theta_{\text{init}}$ & $[1.2, 0.8, 4.0]$   & Initial guess for MLE optimization. \\
		& $\text{maxiter}$      & $200$               & Maximum iterations for the L-BFGS-B optimizer. \\
		& Bounds               & $\theta_0 > 10^{-3}, \theta_2 > 10^{-3}$ & Constraints for MLE optimization. \\
		\hline
	\end{tabular}
	\centering
	\caption{Parameter values used in the numerical simulations. The sigmoidal firing rate function is defined as $f_\theta(x) = \theta_0/[1 + \exp(-\theta_2 (x - \theta_1))]$, where $\theta = (\theta_0, \theta_1, \theta_2)$.}
	\label{tab_simulation_parameters_v16}
\end{table}

\subsubsection{Empirical Fisher Information}
The empirical Fisher information matrix $I_t^{N,\theta^*}$ is computed according to \eqref{eq_def_INts},
with the discretized in time version of the empirical  measure process $\mu^{N,\theta^*}$ (Equation~\ref{eq:muN}),
and the analytically computed gradient $\dot{f}_{\theta^*}$. 
The quantity $\Delta_t^{N,\theta^*}$ which appears in the LAN property (Equation~\ref{eq:LAN})
is similarly computed according to \eqref{eq_def_MN1},
by exploiting also the spiking times 
and corresponding previous potential value
from the algorithm.
Both implementations are provided in the \verb+`compute_delta_and_I`+ function of the code. 

\subsubsection{Maximum Likelihood Estimation}
The MLE $\hat{\theta}_t^N$ is obtained by maximizing the log-likelihood ratio (a discretized version of \ref{eq_def_L}) with respect to $\theta \in \Theta$. This is achieved using the \texttt{scipy} function \texttt{optimize.minimize} with the L-BFGS-B method, which supports box constraints to enforce $\theta_0 > 0$ and $\theta_2 > 0$. The negative log-likelihood is minimized over the parameter space, with an initial guess of $\theta_{\text{init}} = [1.2, 0.8, 4.0]$ and a maximum of 200 iterations. The bounds for the optimization are set to $(10^{-3}, \infty)$ for $\theta_0$ and $\theta_2$, while $\theta_1$ is unconstrained. This approach is implemented in the \verb+`estimate_mle_scipy`+ function, which leverages the analytical form of the likelihood ratio to ensure numerical efficiency.

\subsubsection{Statistical Tests}
To assess the empirical distributions of the estimators and LAN quantities, we employ the following multivariate statistical tests, all implemented using the \texttt{pingouin} and \texttt{scipy.stats} libraries in Python:

$\star$ \uline{Henze-Zirkler Test:} A multivariate normality test based on the work of \cite{HZ90}. This test evaluates whether the empirical distribution of $\Delta_t^{N,\theta^*}$, the LAN predictor, or the MLE matches a multivariate normal distribution. The null hypothesis is that the data is normally distributed.

$\star$ \uline{Hotelling's $T^2$ Test:} A test for the mean vector of a multivariate distribution \cite{Ho31}. For $\Delta_t^{N,\theta^*}$, the null hypothesis is that the mean is zero, while for the LAN predictor and MLE, it tests whether the mean matches the true parameter $\theta^*$. 

$\star$ \uline{Mauchly's Test:} A test for sphericity, which checks whether the covariance matrix of a dataset is proportional to the identity matrix \cite{Mau40}. In our context, we apply this test to the {transformed} data (e.g., $(\hat{I}_t^{N, \theta_*})^{1/2} (\hat{\theta}_{t}^N - \theta^*)$) to verify whether the empirical covariance matches the theoretical covariance structure predicted by the Fisher information matrix.

All tests are performed at a significance level of $\alpha = 0.05$, and the results are reported in the CSV files \texttt{test\_results\_*.csv} for $\Delta_t^{N,\theta^*}$ (abbreviated as \texttt{delta})
and $\hat{\theta}_{t}^N$ (\texttt{MLE}), but also $\wtd{\theta}_{t}^N$ (\texttt{LAN}).

\subsection{Other parameter values}
\label{app_variations}

Alternative scenarios have been considered.
First, we will discuss the effect of increasing the number of neurons in Section~\ref{app_larger_nbr},
thus making the system closer to the McKean-Vlasov dynamics.
Secondly in Section~\ref{app_random},
we will present scenarios for which the randomness is more pronounced on the contrary, 
firstly with a smaller neuron number, then with another set of parameters.

\subsubsection{Larger neuron number}
\label{app_larger_nbr}

When we increase the number of neurons to $N= 100$ and even more with $N=500$,
the fluctuations of the system get visibly averaged out, 
as we can see in Figures~\ref{fig_neuron_traj_n100} and \ref{fig_raster_N100}
that are to be compared with respectively Figures~\ref{fig_neuron_traj} and \ref{fig_raster}.
In particular for $N=500$, 
the dynamics appear very close to deterministic between spikes, while the global spiking intensity appears to stabilize well after 1s.

\begin{figure}[htbp]
	\begin{subfigure}[b]{0.48\textwidth}
		\centering
		\includegraphics[width=\textwidth]{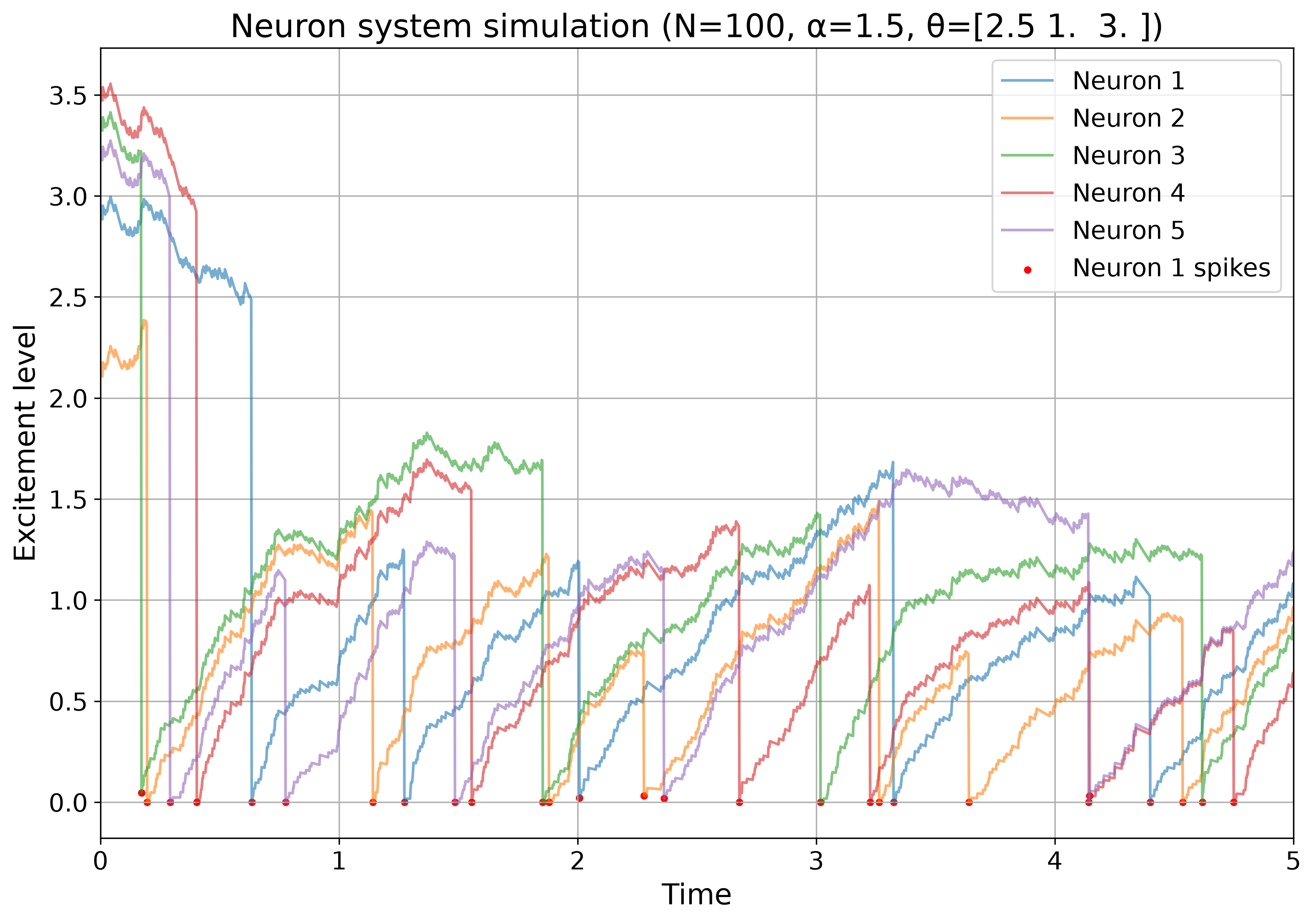}
		\caption{$N=100$.}
	\end{subfigure}
	\hfill
	\vrule
	\hfill
	\begin{subfigure}[b]{0.48\textwidth}
		\centering
		\includegraphics[width=\textwidth]{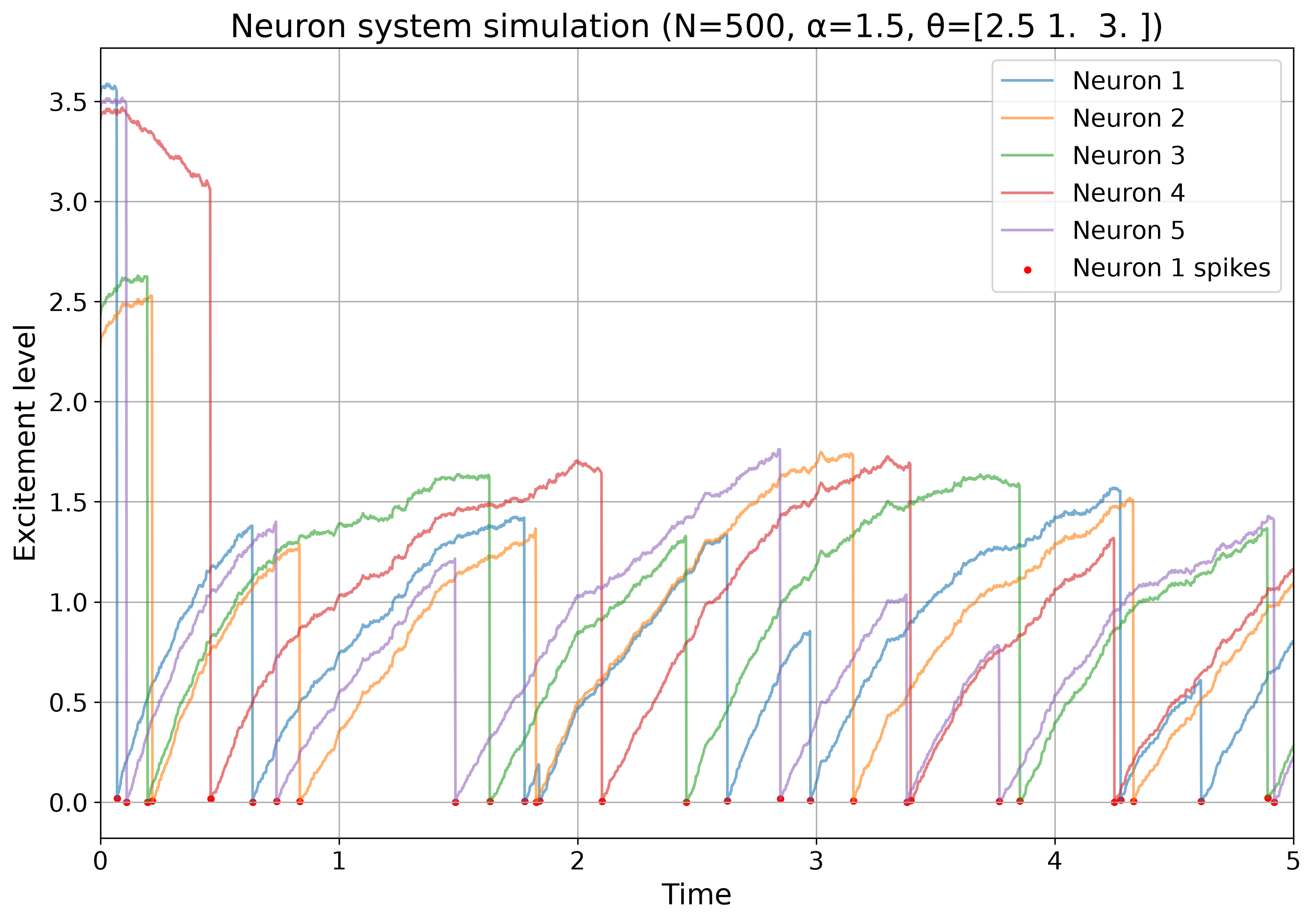}
		\caption{$N=500$.}
	\end{subfigure}

	\caption{Illustration of 5 neuron state trajectories over the time-window $[0, 5]$, for $N= 100$ and $N=500$ neurons.}
	\label{fig_neuron_traj_n100}
\end{figure}

\begin{figure}
	\centering
	\begin{subfigure}[b]{0.48\textwidth}
		\centering
		\includegraphics[width=\textwidth]{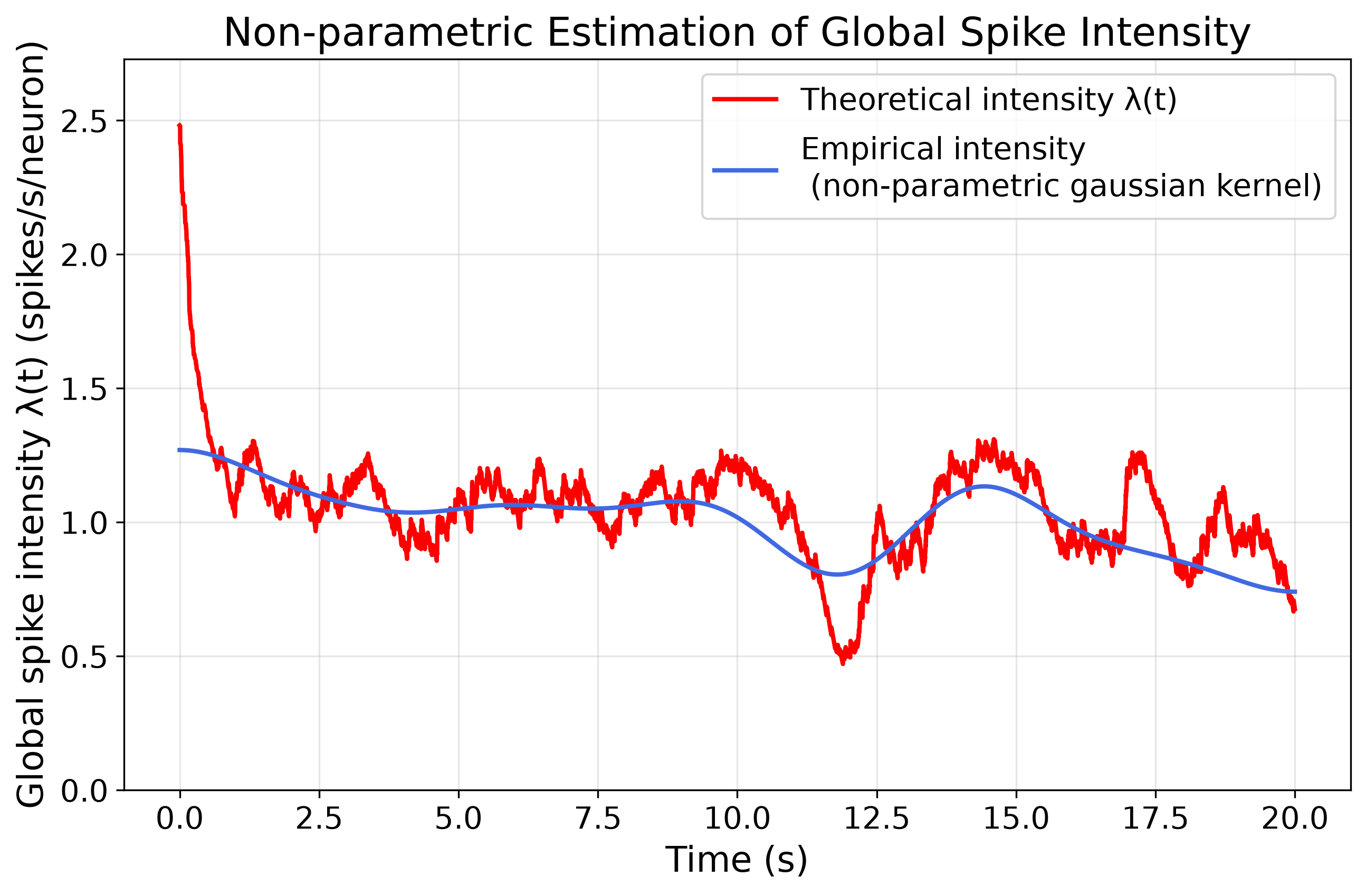}
		\caption{$N= 100$.}
	\end{subfigure}
	\hfill
	\vrule
	\hfill
	\begin{subfigure}[b]{0.48\textwidth}
		\centering
		\includegraphics[width=\textwidth]{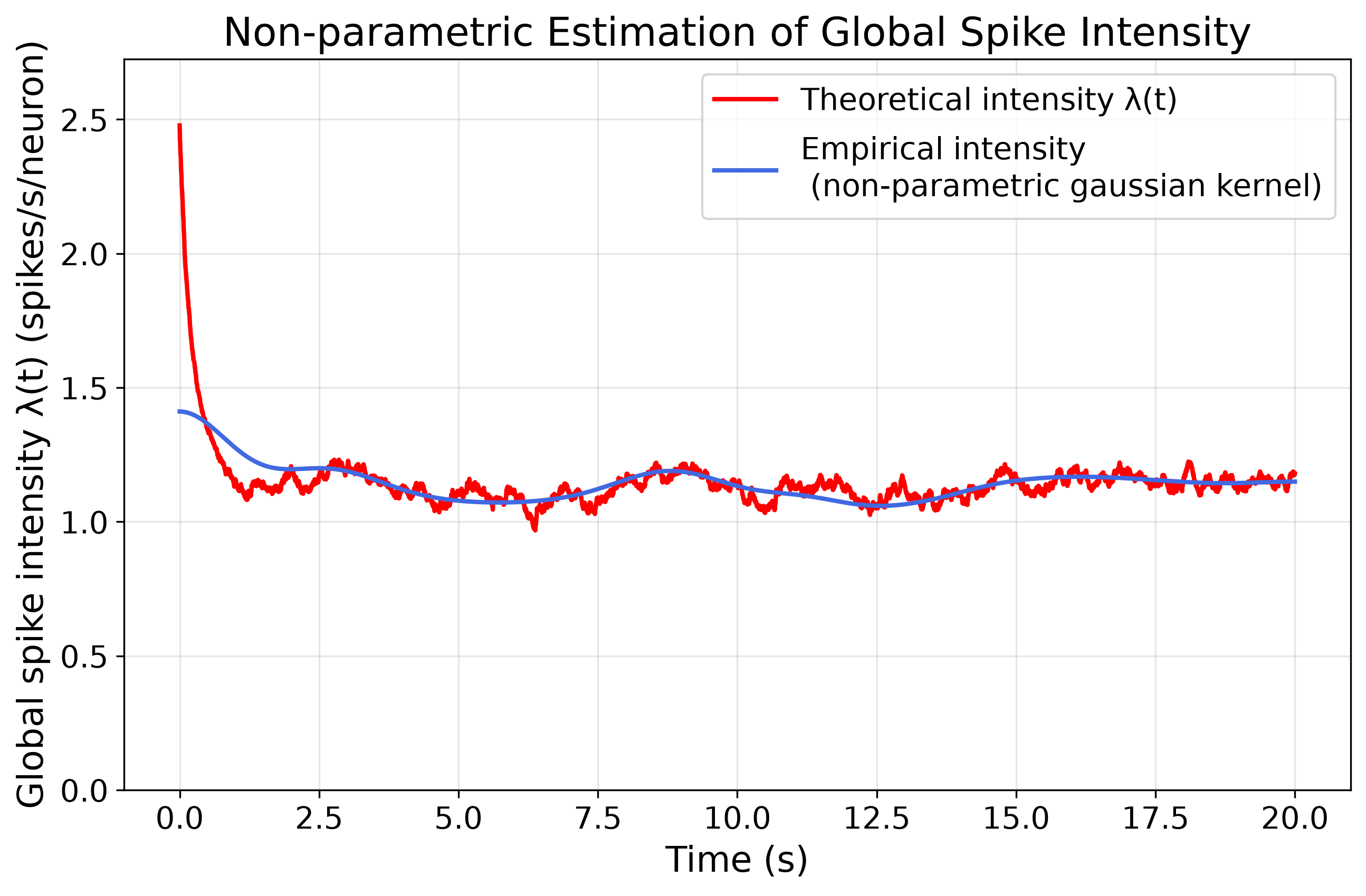}
		\caption{$N= 500$.
		}
	\end{subfigure}
	\caption{Estimation of the global spiking intensity for $N= 100$ and $N=500$ neurons, in blue with a nonparametric gaussian kernel intensity derived from the spike times,
		in red with the computation of the spiking rate over the neuronal states (for comparison).}
	\label{fig_raster_N100}
\end{figure}

Next, we compare the empirical and theoretical covariance matrices for $ \Delta^{N, \ts}_t$  and the MLE in Figures~\ref{fig_cov_comparison_N100}, \ref{fig_cov_comparison_N500}, to be compared with Figure~\ref{fig_cov_comparison}.
While the quality of approximation appears 
not to improve with $N=100$  as compared to $N=50$,
it is even worse concerning the MLE,
an improvement is observed with $N=500$.
Remark that the number of replicates may not be large enough 
for the empirical covariance matrix to be estimated accurately enough.
Note however that these estimations of covariance for the MLE are rescaled by $1/N$, so that the estimator is indeed more efficient as $N$ gets larger.

This inadequacy in the covariance for the replicates where $N=100$
naturally entails the rejection of Mauchly's test, as we can see in Table~\ref{tab:test_results}.
In this table, we display the p-values corresponding to our different tests for the varying population sizes under consideration.
Most of the other tests are accepted with large p-values\footnote{
	While we may expect $\Delta^{N, \ts}_t$
	to be closer than the MLE 
	to the expected normal distribution, 
	and the LAN predictor to be in-between,
	this is not clearly visible from the p-values.}.
Nonetheless, we observe that the test of mean value
is rejected for both the MLE and  $\Delta^{N, \ts}_t$
in the situation where $N=500$.
This outcome is surprising. It may be due to a rare event for the replicates, given that all quantities are affected similarly  (in a new set of 100 replicates of the same scenario, this test is far from being rejected).
The case where $N=25$
is  displayed as well in Table~\ref{tab:test_results},
and will be discussed in the following subsection,
because it corresponds to a larger level of randomness.

\begin{figure}[htbp]
	\centering
	\begin{subfigure}[b]{0.48\textwidth}
		\centering
		\includegraphics[width=\textwidth]{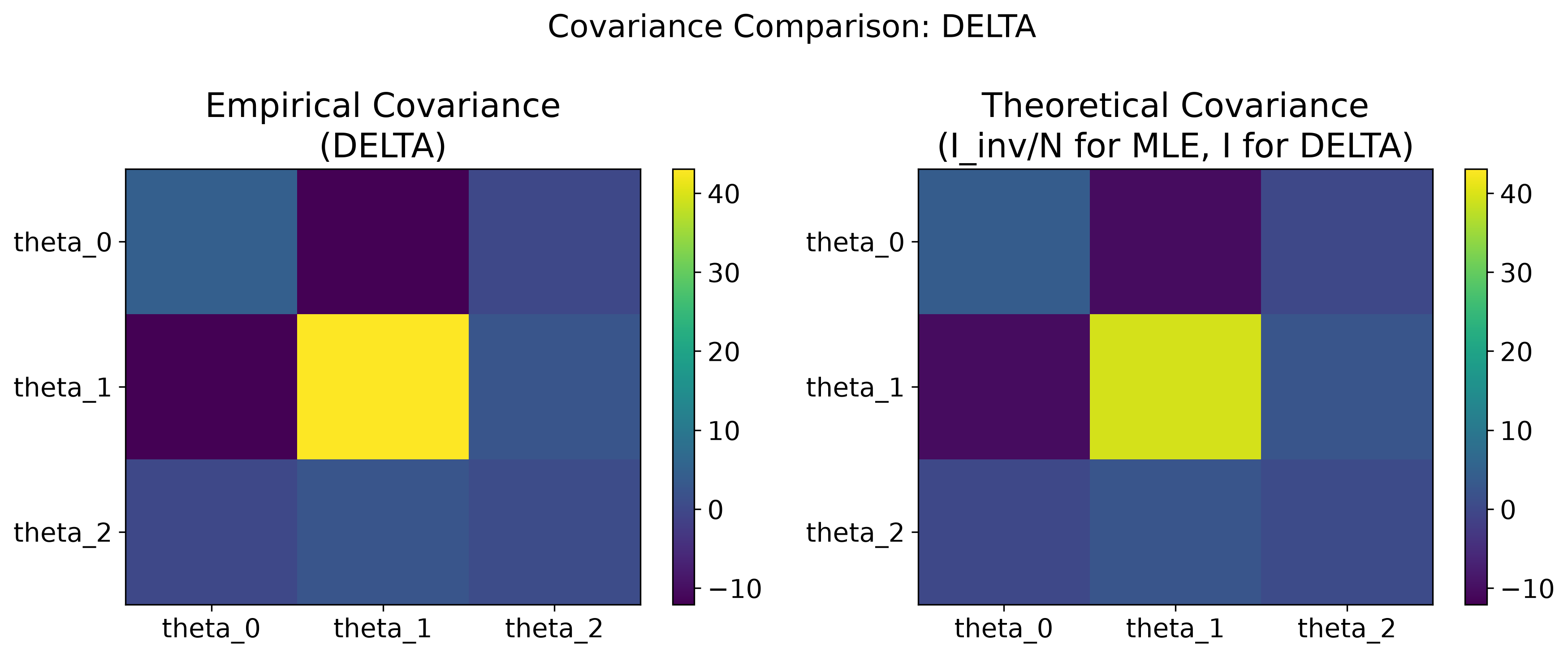}
		\caption{Empirical covariance matrix of $ \Delta^{N, \ts}_t$ (over 100 replicates) as compared to $\wht I^{N}_t$.
			Relative error in Frobenius norm: 0.10.}
	\end{subfigure}
	\hfill
	\vrule
	\hfill
	\begin{subfigure}[b]{0.48\textwidth}
		\centering
		\includegraphics[width=\textwidth]{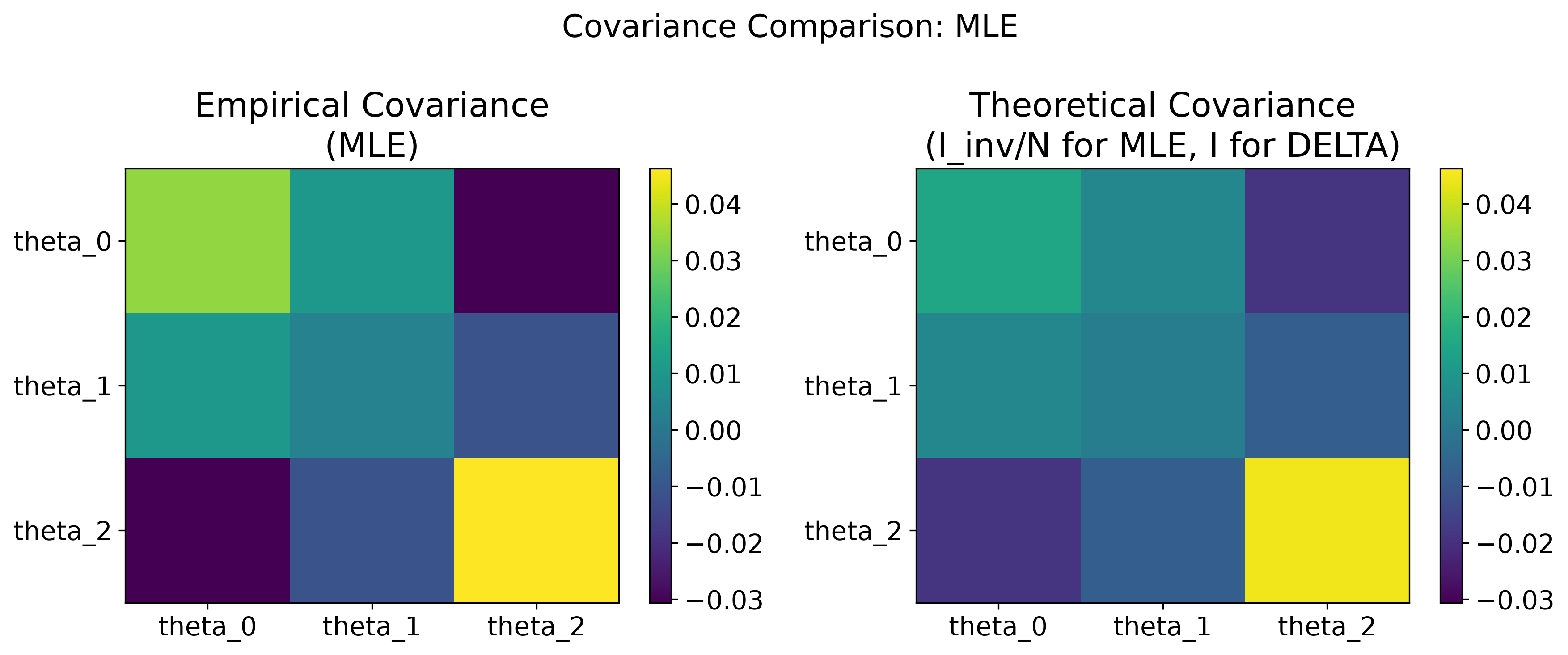}
		\caption{Empirical covariance matrix of the MLE as compared to $N^{-1} (\wht I^{N}_t)^{-1}$.
			Relative error in Frobenius norm: 0.48.
		}
	\end{subfigure}
	\caption{Comparison between the empirical covariance and the prediction with the estimated Fisher information matrix, when $N= 100$.}
	\label{fig_cov_comparison_N100}
\end{figure}

\begin{figure}[htbp]
	\centering
	\begin{subfigure}[b]{0.48\textwidth}
		\centering
		\includegraphics[width=\textwidth]{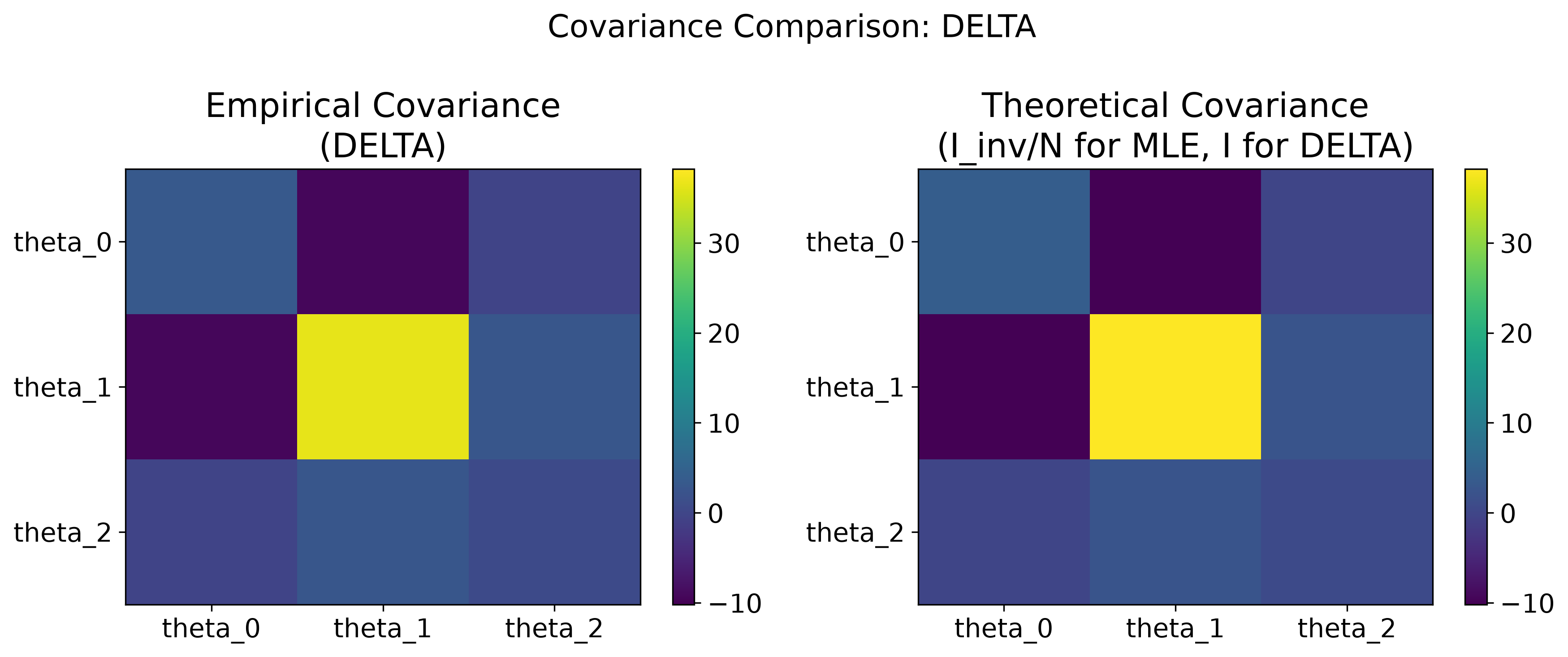}
		\caption{Empirical covariance matrix of $ \Delta^{N, \ts}_t$ (over 100 replicates) as compared to $\wht I^{N}_t$. Relative error in Frobenius norm: 0.06. }
	\end{subfigure}
	\hfill
	\vrule
	\hfill
	\begin{subfigure}[b]{0.48\textwidth}
		\centering
		\includegraphics[width=\textwidth]{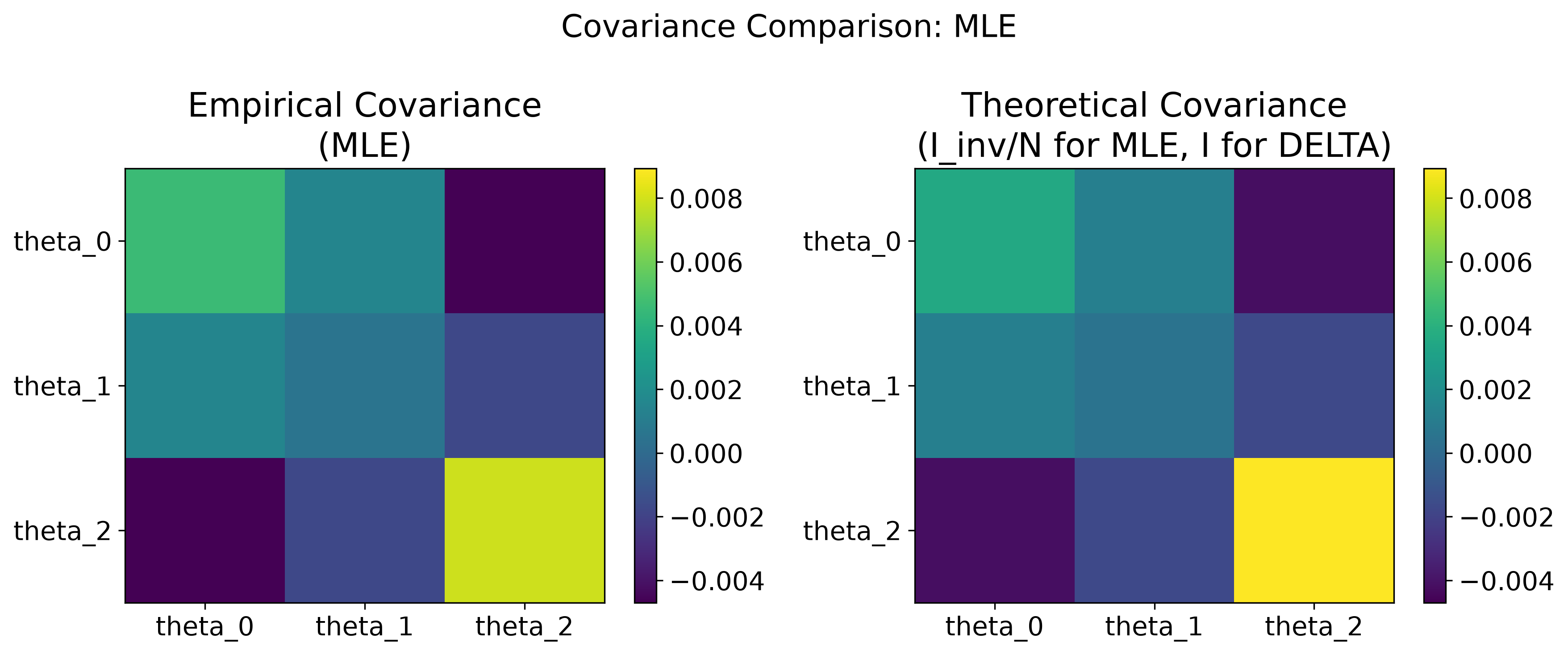}
		\caption{Empirical covariance matrix of the MLE as compared to $N^{-1} (\wht I^{N}_t)^{-1}$. Relative error in Frobenius norm: 0.15.
		}
	\end{subfigure}
	\caption{Comparison between the empirical covariance and the prediction with the estimated Fisher information matrix, when $N= 500$.}
	\label{fig_cov_comparison_N500}
\end{figure}

\begin{table}[htbp]
	\centering
	\footnotesize
	\setlength\tabcolsep{3pt}
	\begin{tabular}{|l|ccc||ccc||ccc||ccc|}
		\hline
		& \multicolumn{3}{|c||}{$N=25$} & \multicolumn{3}{|c||}{$N=50$} & \multicolumn{3}{|c||}{$N=100$} & \multicolumn{3}{|c|}{$N=500$} \\
		\hline
		&	HZ & HT & M & HZ & HT & M & HZ & HT & M & HZ & HT & M \\
		\hline
		MLE 
		&\!0.105\! &\!0.161\! & \cellcolor[gray]{0.90} 0.024
		&\!0.527\! &\!0.518\! &\!0.769\! 
		&\!0.504\! &\!0.072\! & \cellcolor[gray]{0.90} 0.002
		&\!0.623\! &\cellcolor[gray]{0.90} 0.016 &\!0.094\! \\
		LAN 
		&\!0.077\! & \cellcolor[gray]{0.90} 0.027 & \cellcolor[gray]{0.90} 0.006
		&\!0.746\! &\!0.348\! &\!0.538\! 
		&\!0.420\! &\!0.390\! & \cellcolor[gray]{0.90} 0.005
		&\!0.442\! & \cellcolor[gray]{0.90} 0.017 &\!0.104\! \\
		Delta 
		&\cellcolor[gray]{0.90} 0.025\! &\!0.232\! & \cellcolor[gray]{0.90} 0.050
		&\!0.452\! &\!0.714\! &\!0.205\! 
		&\!0.546\! &\!0.210\! &\!0.996\! 
		&\!0.535\! & \cellcolor[gray]{0.90} 0.016 &\!0.918\! \\
		\hline
	\end{tabular}
	\caption{Test results given the p-value for different sample sizes ($N$); methods: $HZ$, $HT$ and $M$ for respectively Henze-Zirkler, Hotelling's $T^2$ and Mauchly, that are tests for respectively normality, mean and sphericity; and quantities: $MLE$, $LAN$ and $Delta$ stands respectively for $\hat \theta_t^N$,  $\check \theta_t^N$ and $\wht \Delta^{N, \ts}_t/\sqrt{N}$. 
		Accepted if $p > 0.05$, Rejected (in gray) otherwise.}
	\label{tab:test_results}
	
\end{table}

Concerning finally the adequation between the MLE and the LAN predictor, we show in Figure~\ref{fig_LAN_MLE_v100}
the distribution of relative error that is to be compared with Figure~\ref{fig_LAN_MLE}-(B).
This plot confirms that the approximation gets clearly better as the number of neurons gets larger, even though the MLE gets closer to $\ts$.
For instance, 90\% of the replicates have a relative error below 15\% for $N= 100$ 
and below 10\% for $N=500$ (under both distances in both cases).

\begin{figure}[htbp]
	\centering
	\begin{subfigure}[b]{0.48\textwidth}
		\centering
		\includegraphics[width=\textwidth]{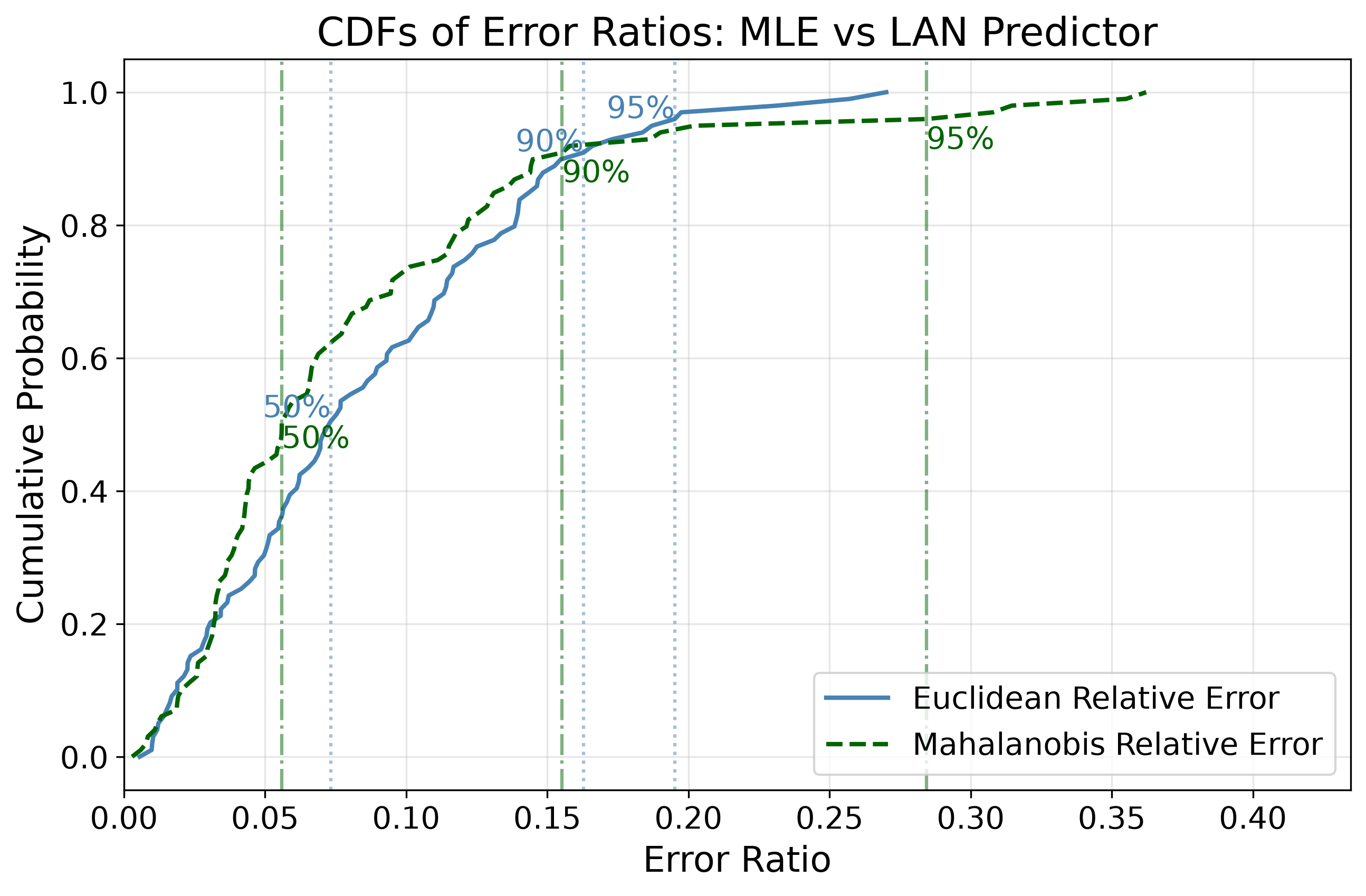}
		\caption{$N= 100$. }
	\end{subfigure}
	\hfill
	\vrule
	\hfill
	\begin{subfigure}[b]{0.48\textwidth}
		\centering
		\includegraphics[width=\textwidth]{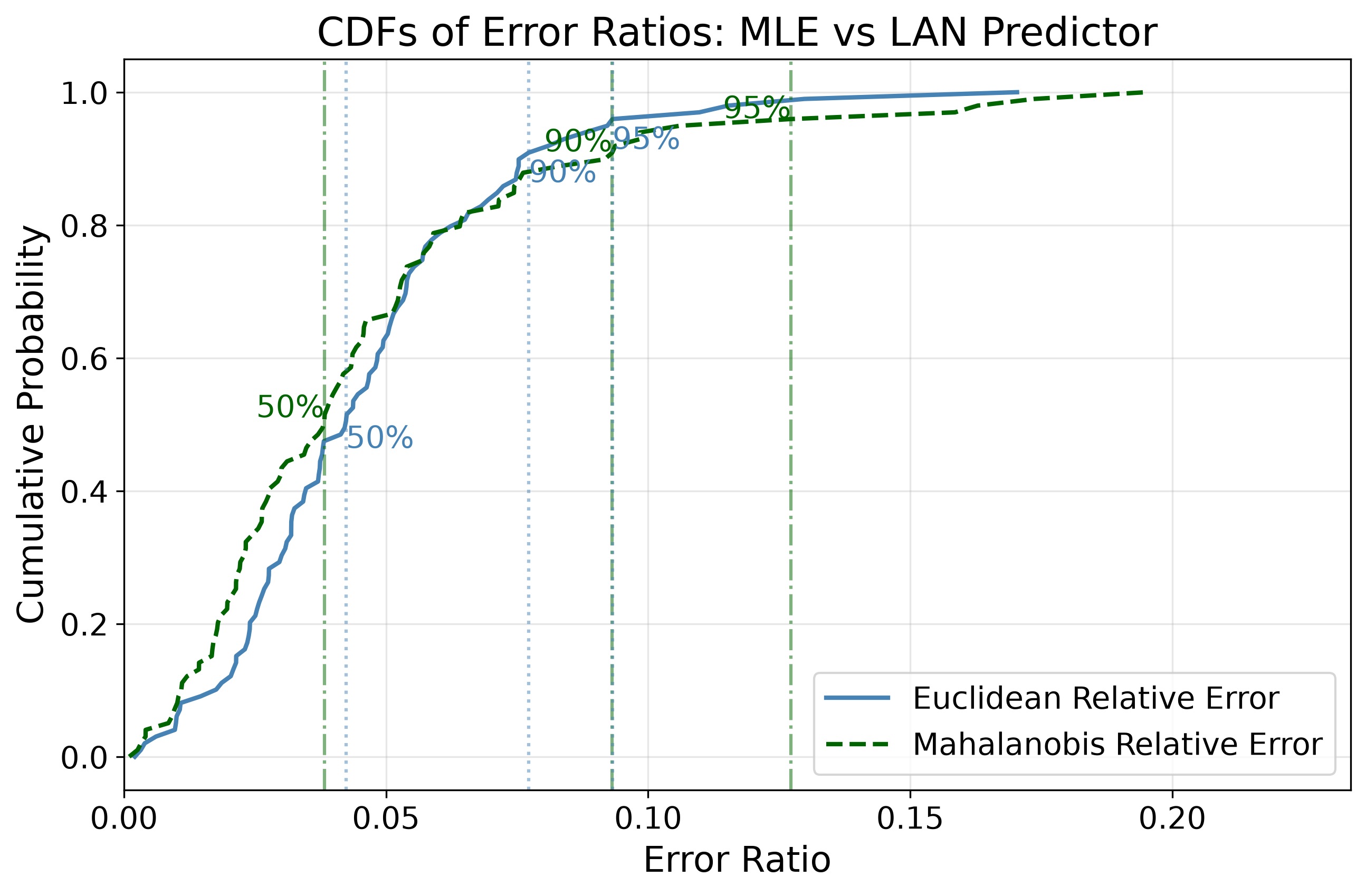}
		\caption{$N = 500$.
		}
	\end{subfigure}
	\caption{Comparison between the estimated MLE $\hat \theta_t^N$ and the LAN predictor  $\check \theta_t^N = \ts + 	(\wht I^{N, \ts}_t)^{-1} 	\wht \Delta^{N, \ts}_t/\sqrt{N}$
		as compared to $\ts$, for $N= 100$ and $N=500$ neurons.
		C.d.f. of the relative error $\|\hat \theta_t^N - \check \theta_t^N\|/\|\hat \theta_t^N - \theta_*\|$ for the Euclidian distance (blue line) and the Mahalanobis distance (green dashed line). Vertical lines display the quantiles of both distributions corresponding to 
		$50$\%, $90$\% and $95$\%.}
	\label{fig_LAN_MLE_v100}
\end{figure}

\subsubsection{Larger fluctuation levels}
\label{app_random}

The effect of a larger level of randomness can first 
be identified by looking at the scenario with $N=25$,
as it was hinted in the discussion of Table~\ref{tab:test_results}.
There, we can see that some tests are rejected, especially the ones regarding the covariance matrix, 
and that the MLE surprisingly scores better for the two other tests.

That the fluctuations of the system do visibly not get  averaged out when $N=25$ can be realized by looking at Figure~\ref{fig_raster_N25} whose panels are to be compared respectively with Figures~\ref{fig_neuron_traj} and \ref{fig_raster}.
In particular with the estimation of the global spiking intensity in panel~(B),
we can observe an instance of extinction of the system,
after an intermediate period of stabilization.

Next, we look at the covariance matrix of the MLE 
and the relative error with the LAN predictor in Figure~\ref{fig_cov_comparison_N25}, that is to be compared with Figure~\ref{fig_cov_comparison}.
Despite these large randomness levels,
the empirical covariance matrix of the MLE fits the expected matrix quite well.
The relative error is undoubtedly larger.
Yet more than 50\% of the replicates have a relative error below 20\%, and more than 95\% below 40\%,
which is still significant.
\medskip

\begin{figure}
	\centering
	\begin{subfigure}[b]{0.48\textwidth}
		\centering
		\includegraphics[width=\textwidth]{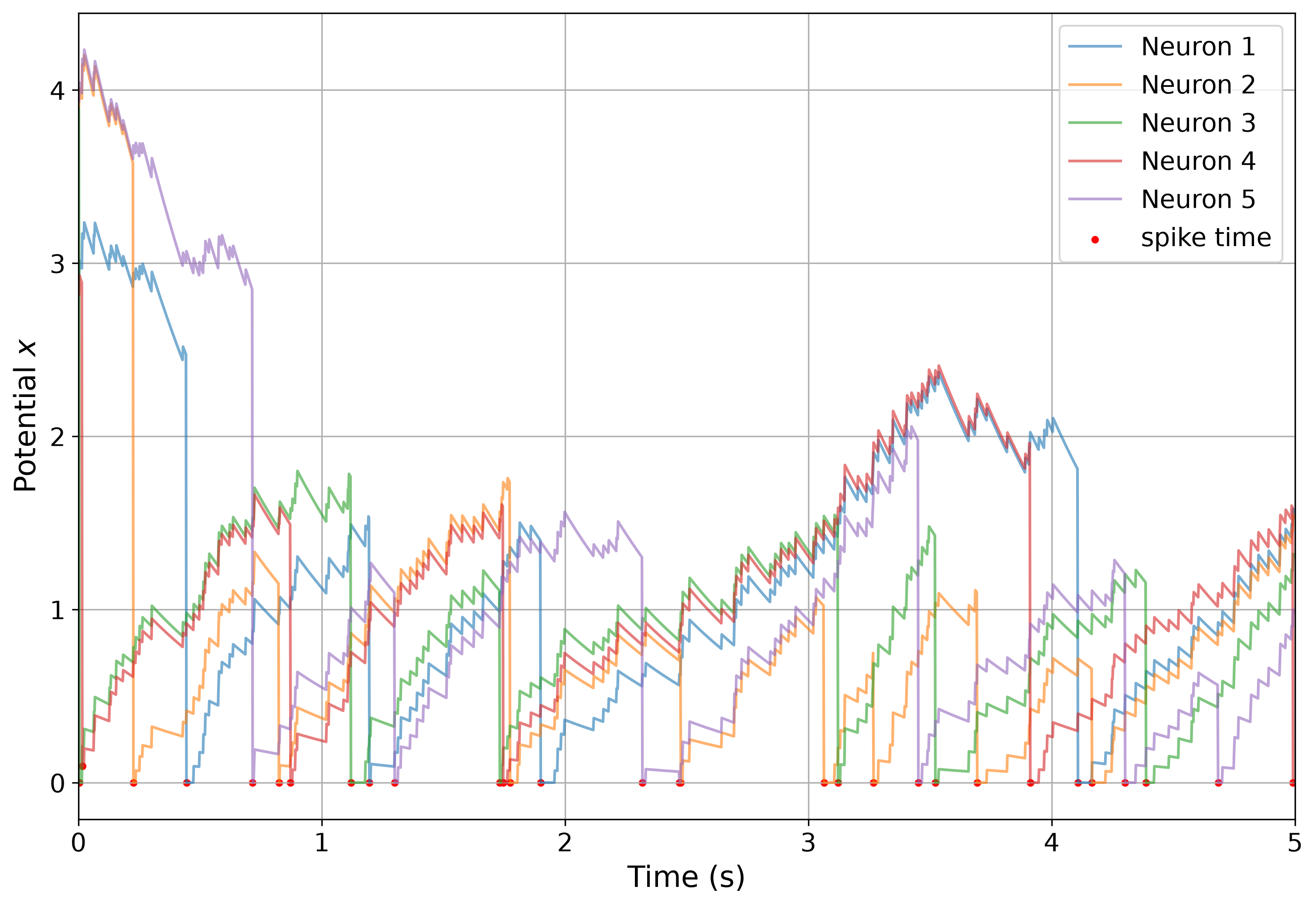}
		\caption{Neuron trajectories.}
	\end{subfigure}
	\hfill
	\vrule
	\hfill
	\begin{subfigure}[b]{0.48\textwidth}
		\centering
		\includegraphics[width=\textwidth]{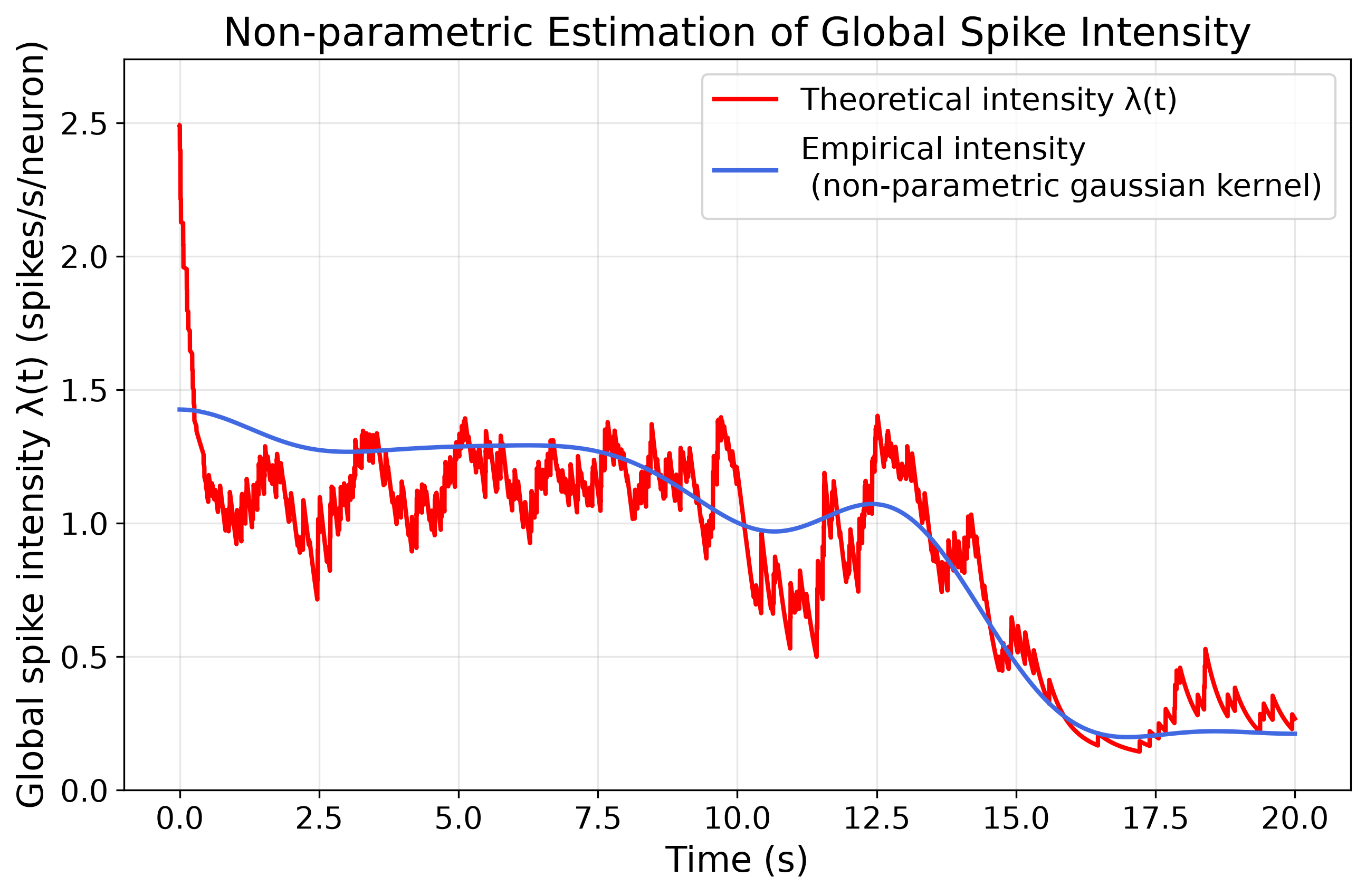}
		\caption{Estimated global spiking intensity.
			The extinction of the system can be observed.
		}
	\end{subfigure}
	\caption{Visualisation of the randomness level for $N= 25$.}
	\label{fig_raster_N25}
\end{figure}

\begin{figure}
	\centering
	\begin{subfigure}[b]{0.48\textwidth}
		\centering
		\includegraphics[width=\textwidth]{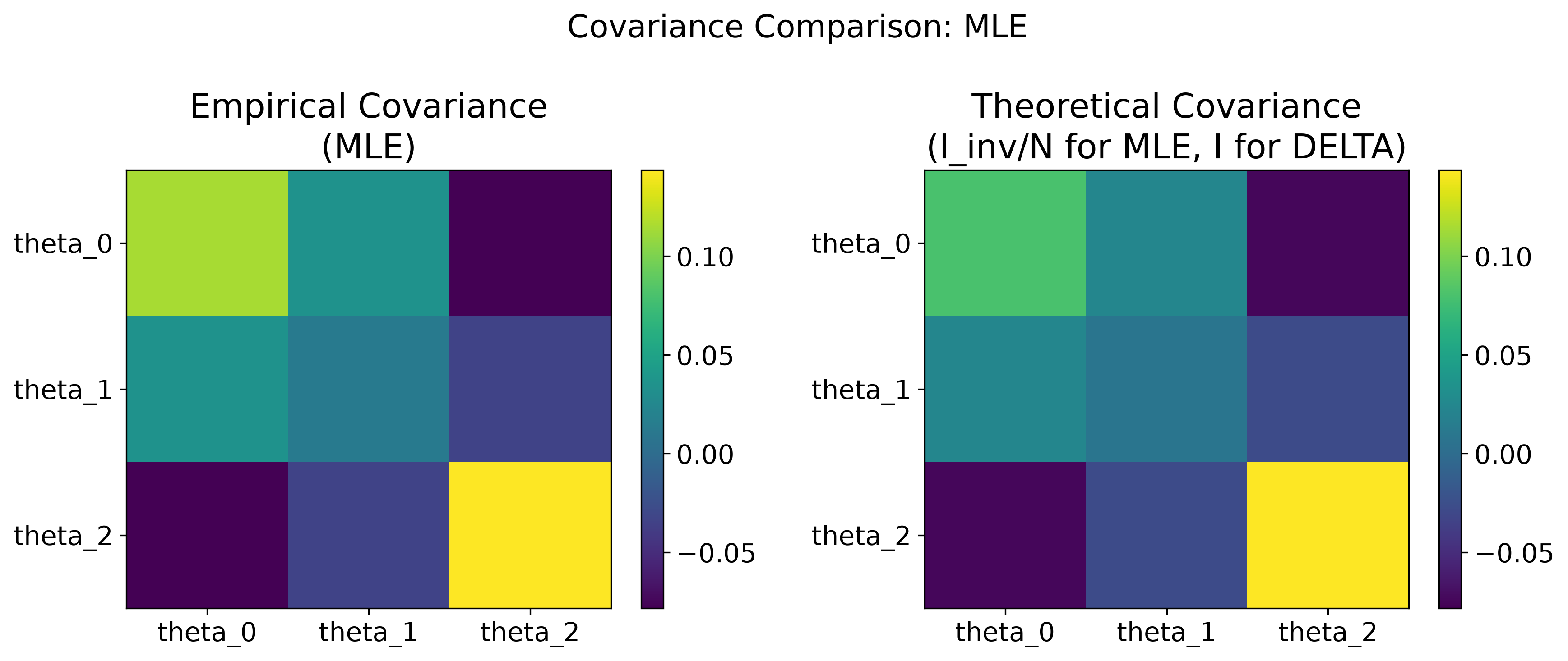}
		\caption{Empirical covariance matrix of the MLE as compared to $N^{-1} (\wht I^{N}_t)^{-1}$. Relative error in Frobenius norm: 0.20.
		}
	\end{subfigure}
	\hfill
	\vrule
	\hfill
	\begin{subfigure}[b]{0.48\textwidth}
		\centering
		\includegraphics[width=\textwidth]{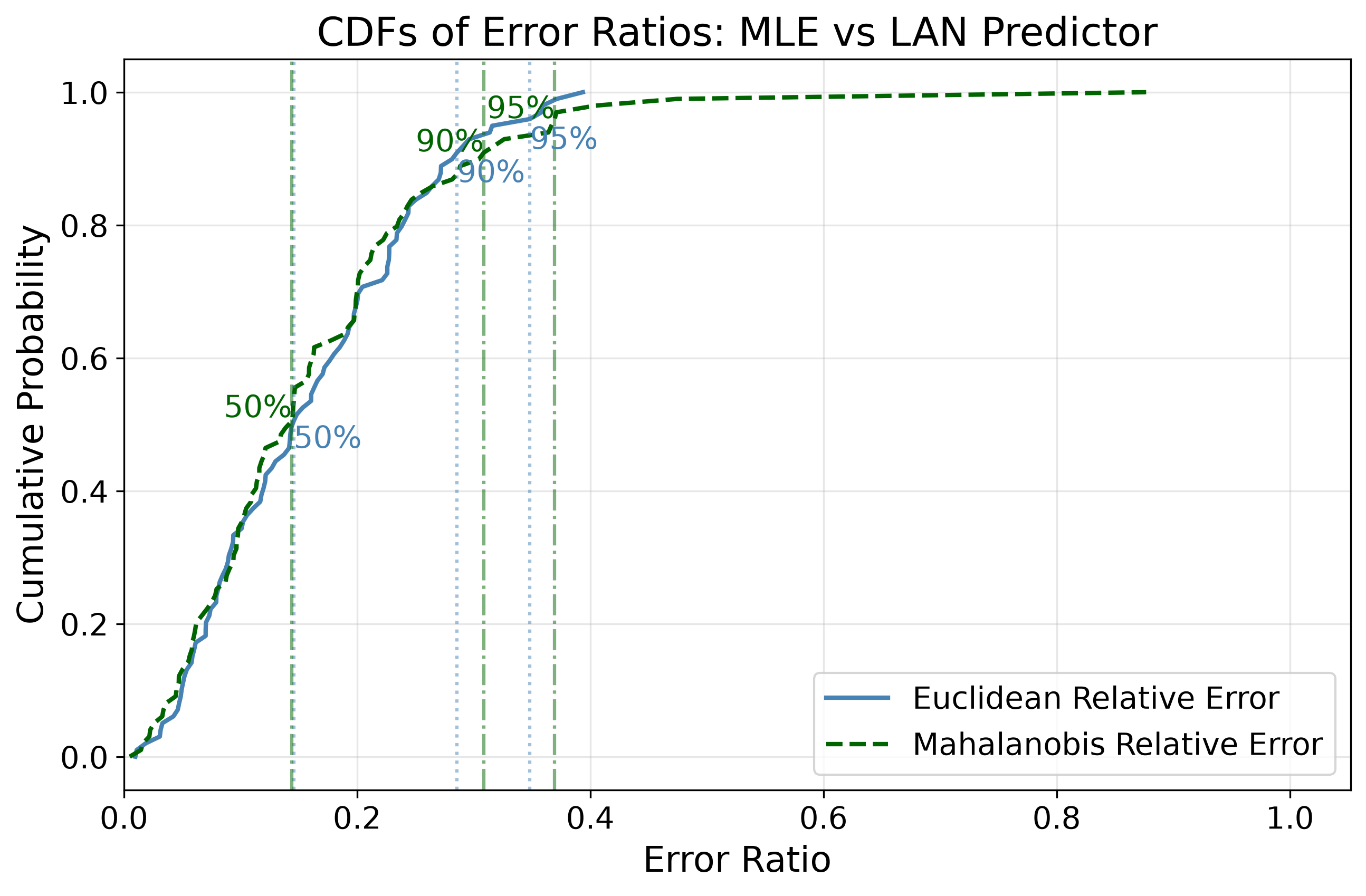}
		\caption{	C.d.f. of the relative error $\|\hat \theta_t^N - \check \theta_t^N\|/\|\hat \theta_t^N - \theta_*\|$.
		}
	\end{subfigure}
	\caption{Comparison between the LAN prediction and the MLE, for $N=25$.}
	\label{fig_cov_comparison_N25}
\end{figure}

\FloatBarrier
Finally, we describe the outcome of replicates generated with another set of parameters, given in Table~\ref{tbl_v39}, and for which the randomness level is amplified with still $N=50$.
As shown in Figure~\ref{fig_raster_N27}-(B),
the spiking rate is smaller,
and while the theoretical intensity corresponding to neuronal stats stabilizes quickly and remains quite confined, there are large variations in the observed intensity.
So neuron trajectories are affected by less frequent oscillations yet more variations in the effective slope, 
as we can see in Figure~\ref{fig_raster_N27}-(A). 
The covariance matrices are also quite badly conditioned,
in that there is much more precision in $\theta_0$ 
(as we can see in Figure~\ref{fig_cov_comparison_v39}-(A)).
and much less in $\theta_2$ 
(as we can see in Figure~\ref{fig_cov_comparison_v39}-(B)).
The relative error (for the covariance matrices in Frobenius norm) is actually reduced in these conditions, 
to $1.4\%$ for $\Delta_t^{N, \ts}$
and to $14\%$ for the MLE,
though it is difficult to identify visually.

\begin{minipage}{0.48\textwidth}
	\begin{tabular}{|c|c|c|c||c|c|c|}
		\hline
		$N$&	$\alpha$ & $\nu_1$ & $\nu_2$ & $\theta_0$ & $\theta_1$ & $\theta_2$  \\
		\hline
		$50$&	$0.3$ & $10$ & $ 10=\nu_1$  & $0.8$ & $0.6$ & $3.0$  \\
		\hline
	\end{tabular}
	\centering
	\captionof{table}{Alternative set of parameters.}
	\label{tbl_v39}
\end{minipage}
\hfill
\vline
\hfill
\begin{minipage}{0.48\textwidth}
	\centering
	\includegraphics[width = 0.6\textwidth]{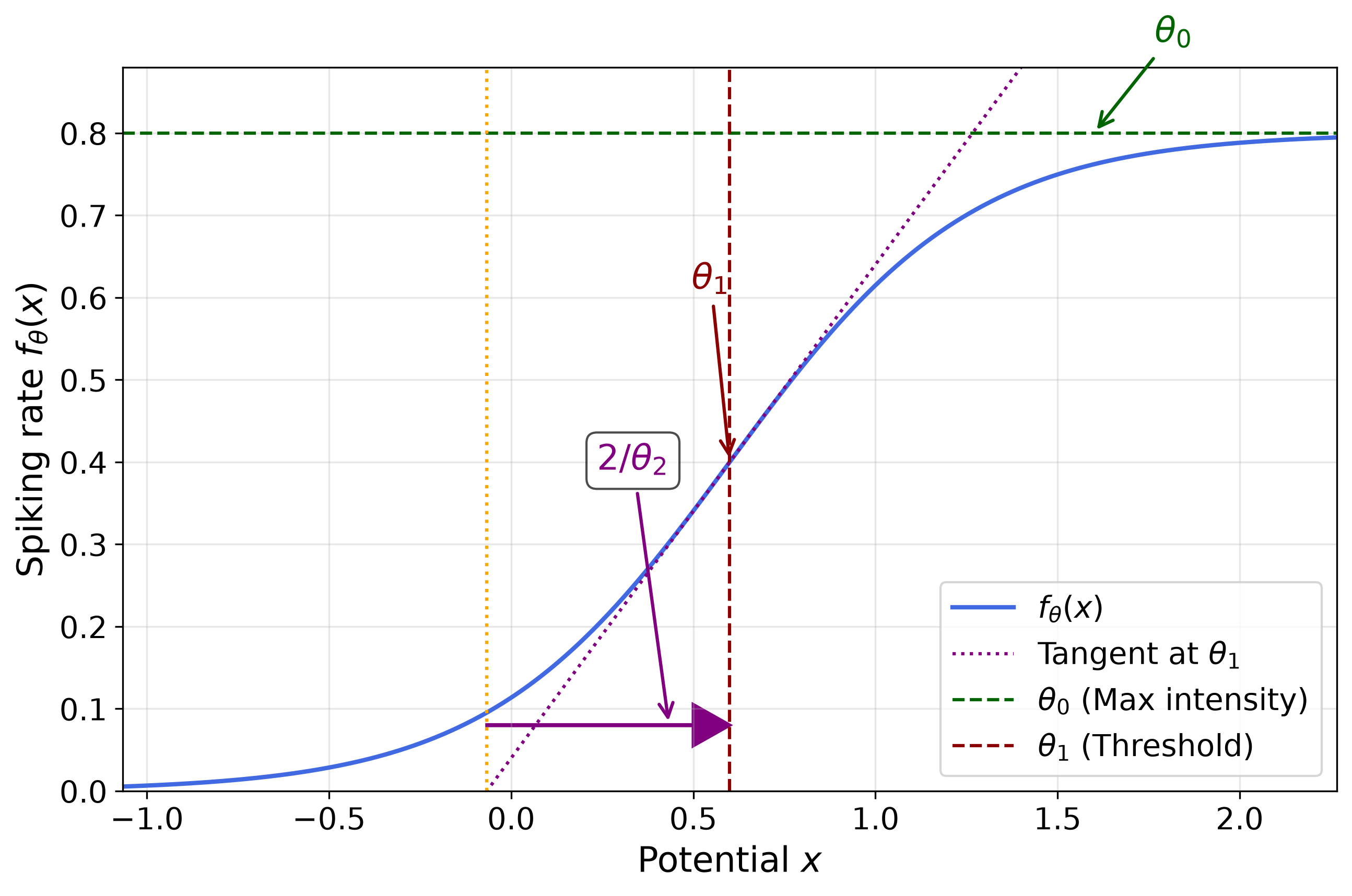}
	\captionof{figure}{Spiking rate,\newline for those new parameters.}
	\label{fig_spiking_rate_v39}
\end{minipage}

\begin{figure}[htbp]
	\centering
	\begin{subfigure}[b]{0.48\textwidth}
		\centering
		\includegraphics[width=\textwidth]{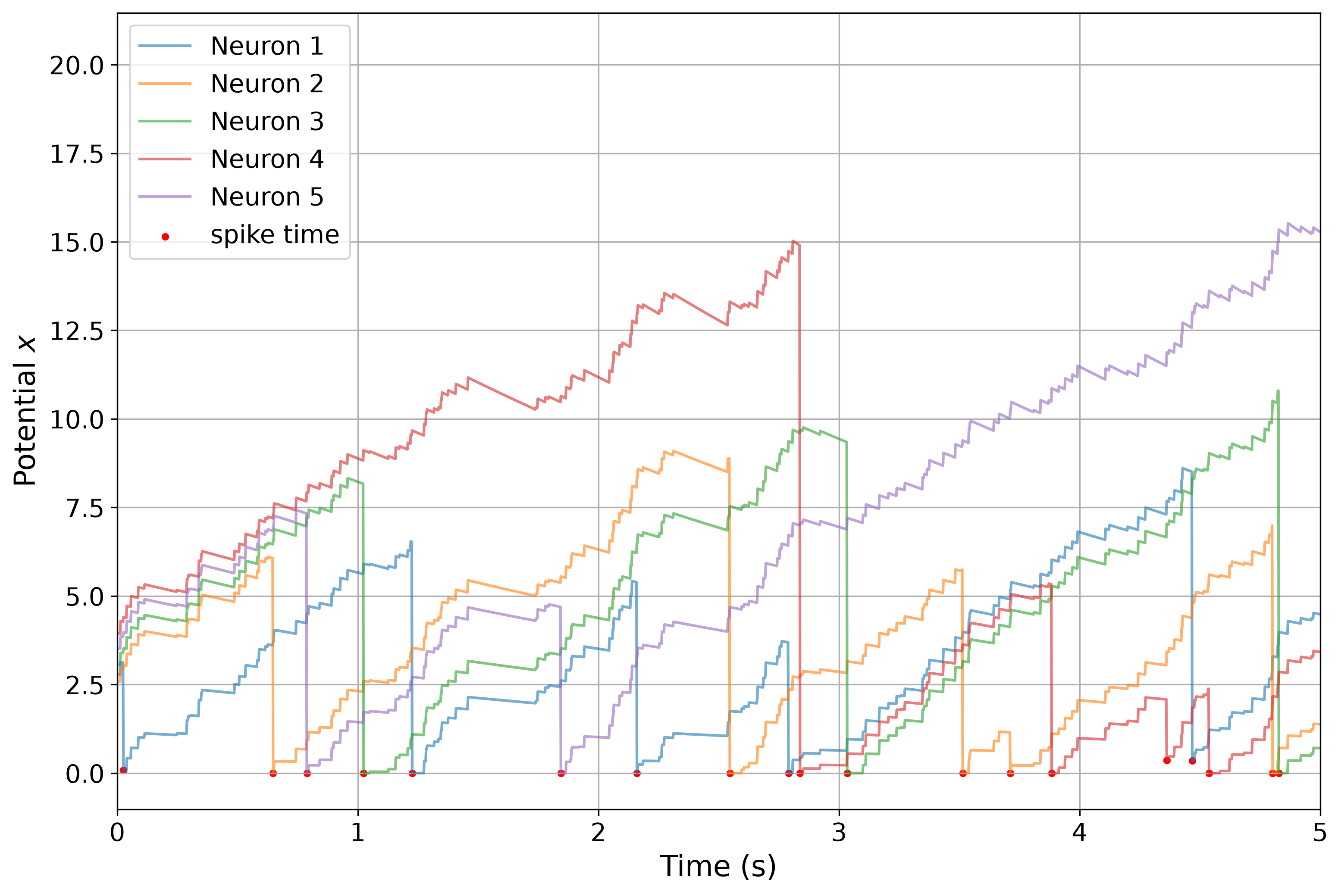}
		\caption{Neuron trajectories.}
	\end{subfigure}
	\hfill
	\vrule
	\hfill
	\begin{subfigure}[b]{0.48\textwidth}
		\centering
		\includegraphics[width=\textwidth]{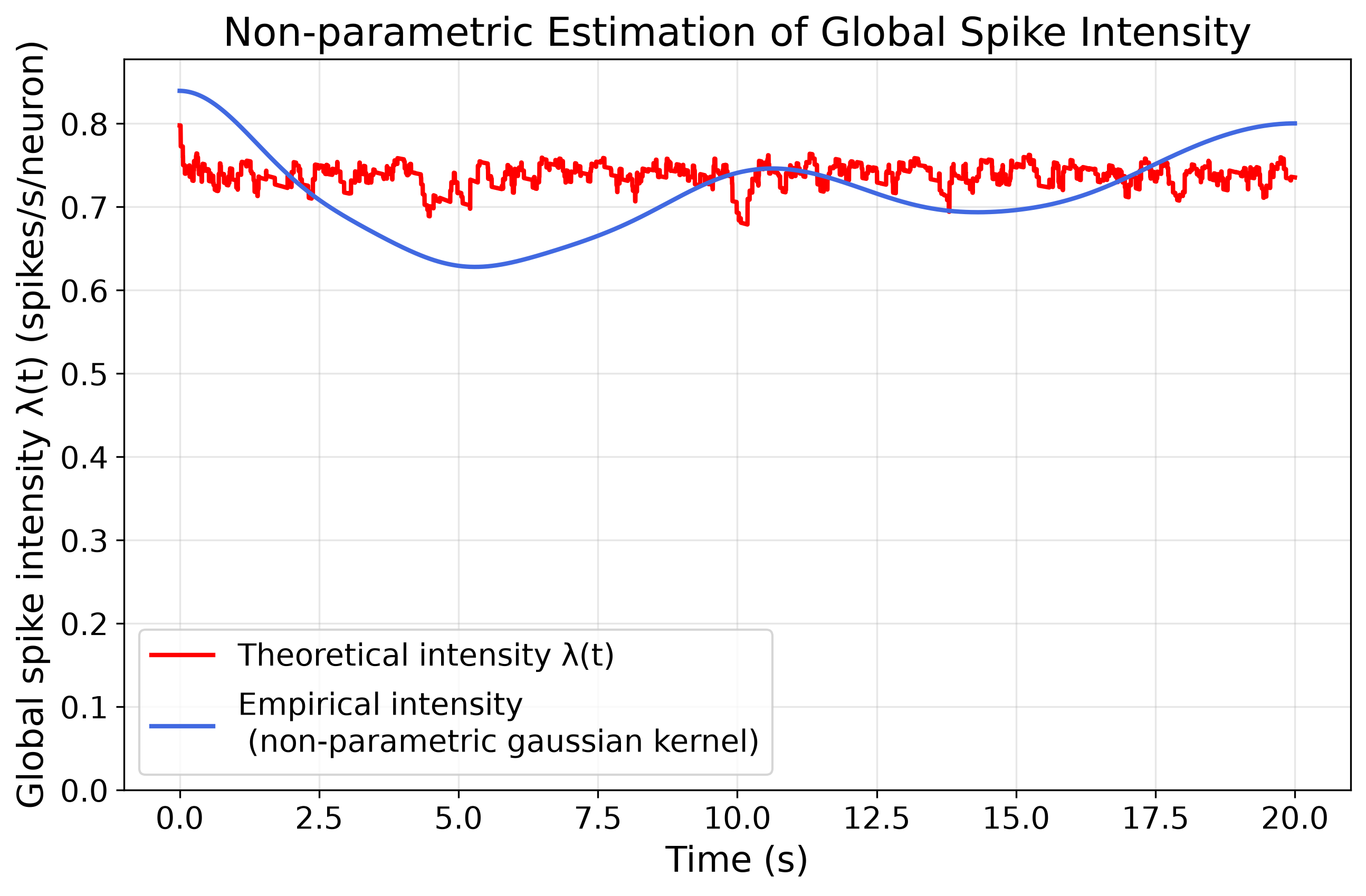}
		\caption{Estimated global spiking intensity.
			The extinction of the system can be observed.
		}
	\end{subfigure}
	\caption{Visualisation of the randomness level, for the set of parameters given in Table~\ref{tbl_v39}.}
	\label{fig_raster_N27}
\end{figure}

\begin{figure}[htbp]
	\centering
	\begin{subfigure}[b]{0.48\textwidth}
		\centering
		\includegraphics[width=\textwidth]{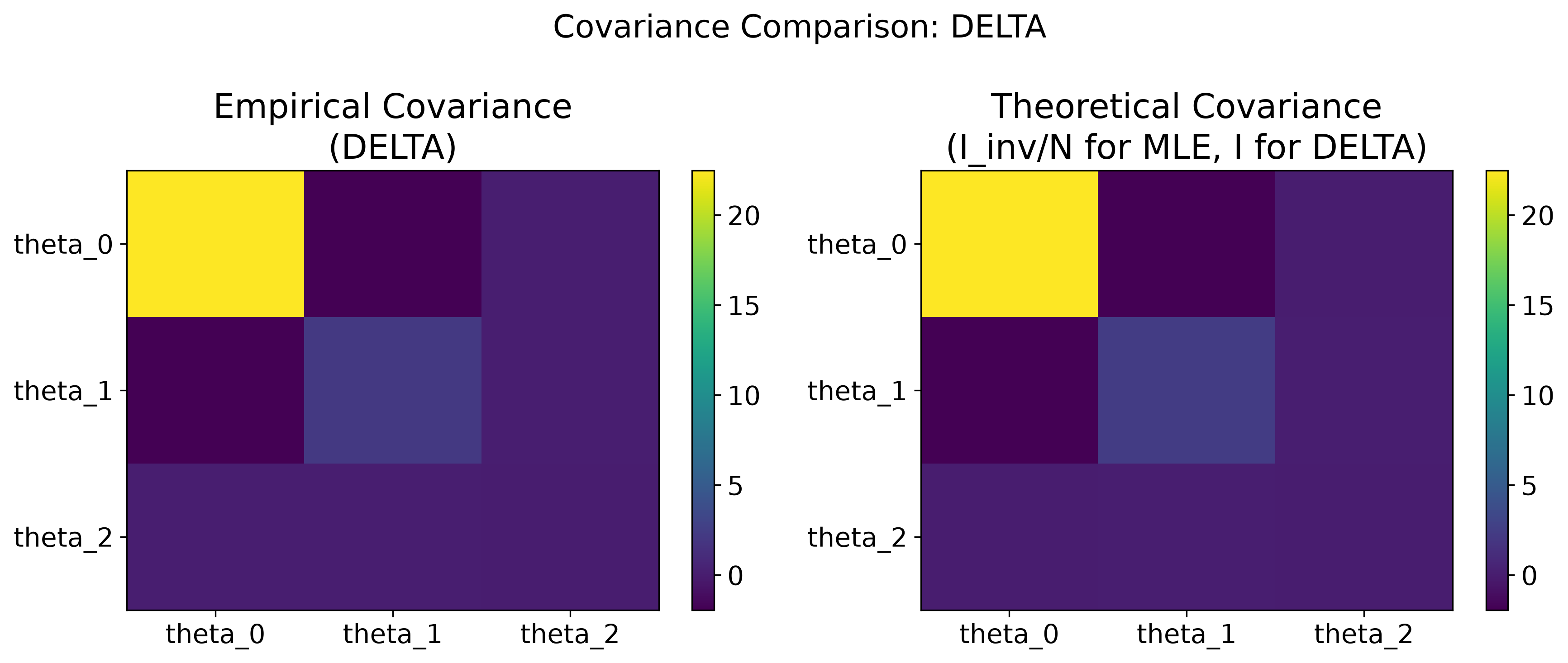}
		\caption{Empirical covariance matrix of $ \Delta^{N, \ts}_t$ (over 100 replicates) as compared to $\wht I^{N}_t$. Relative error in Frobenius norm: 0.014.}
	\end{subfigure}
	\hfill
	\vrule
	\hfill
	\begin{subfigure}[b]{0.48\textwidth}
		\centering
		\includegraphics[width=\textwidth]{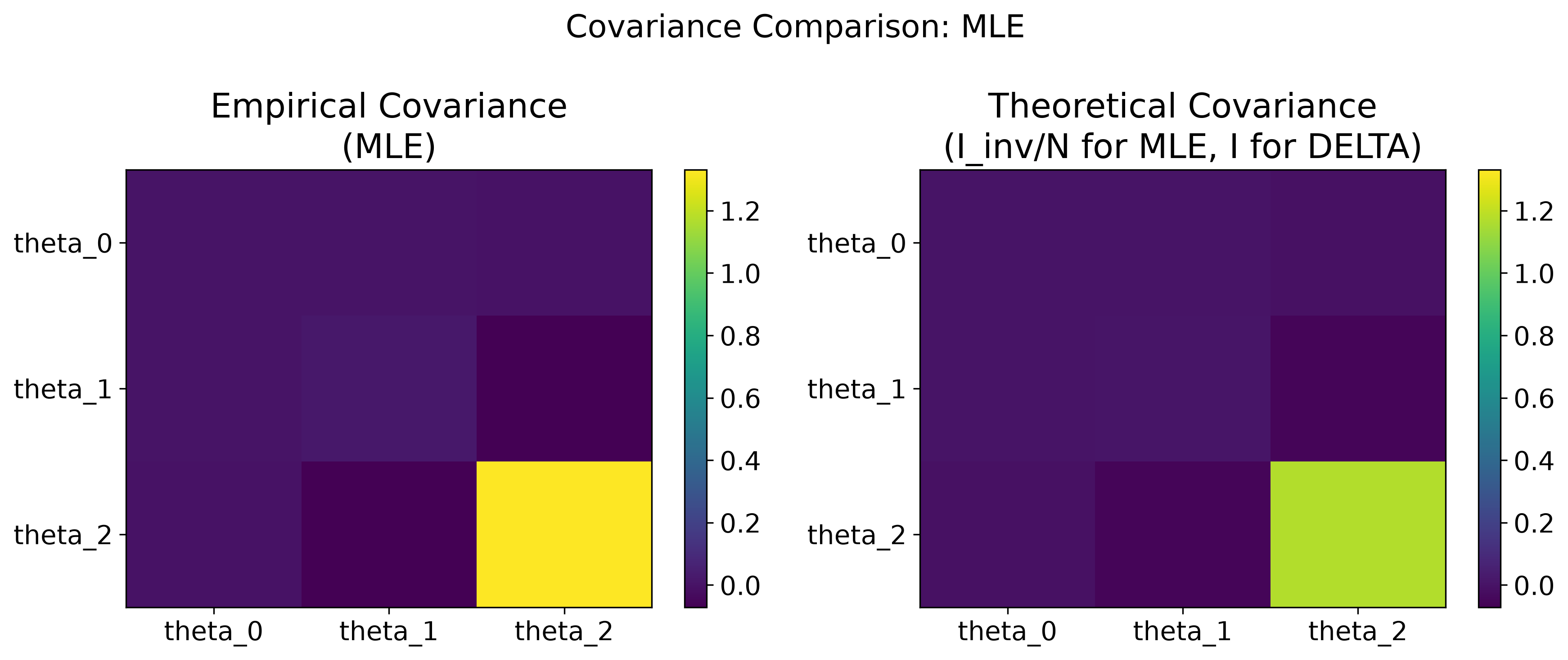}
		\caption{Empirical covariance matrix of the MLE as compared to $N^{-1} (\wht I^{N}_t)^{-1}$. Relative error in Frobenius norm: 0.14.
		}
	\end{subfigure}
	\caption{Comparison between the empirical covariance and the prediction with the estimated Fisher information matrix, for the set of parameters given in Table~\ref{tbl_v39}.}
	\label{fig_cov_comparison_v39}
\end{figure}

Table~\ref{tab:test_results_v39} displays the p-values corresponding to the different tests we conduct as well for this set of parameters. 
As we can see, the test of sphericity is rejected for the MLE (as well as the LAN predictor, and contrary to $\Delta_t^{N, \ts}$).
We suspect it is typically the one most susceptible to be rejected, especially for a case where the randomness in the Fisher information could not be neglected, notably when it is badly conditioned.
The compared histograms and cumulative distribution function of the relative error between the MLE and the LAN predictor are also provided in Figure~\ref{fig_LAN_MLE_v39}.
It is noticeable from the later 
that the deviations are much worse than for our reference set of parameters, see Figure~\ref{fig_LAN_MLE_v39}.

\begin{figure}[htbp]
	\centering
	\begin{subfigure}[b]{0.48\textwidth}
		\centering
		\includegraphics[width=\textwidth]{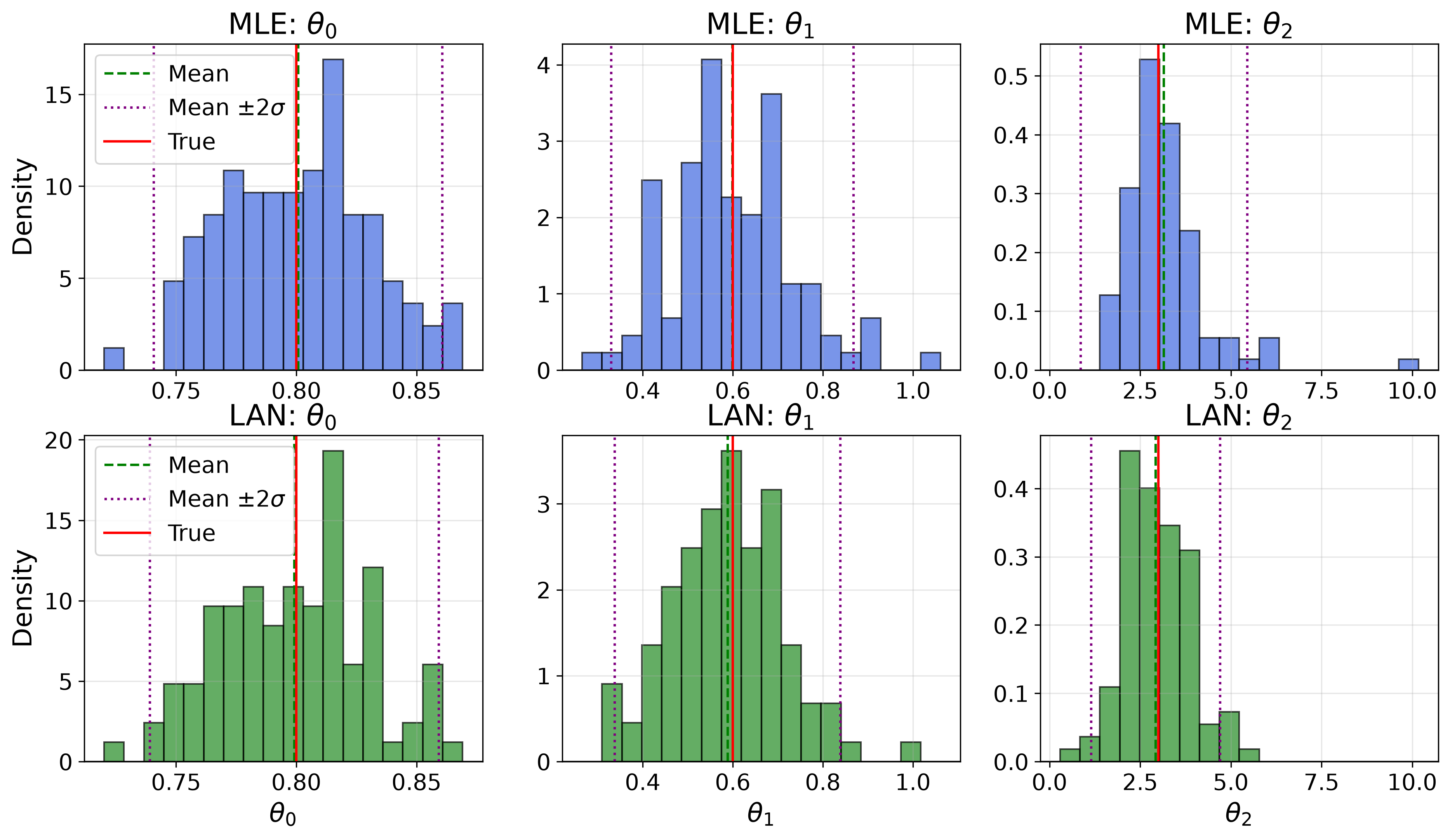}
		\caption{Histograms of the MLE (top row) and LAN predictor (bottom row) for the three parameters $\theta_0$, $\theta_1$, and $\theta_2$. }
	\end{subfigure}
	\hfill
	\vrule
	\hfill
	\begin{subfigure}[b]{0.48\textwidth}
		\centering
		\includegraphics[width=\textwidth]{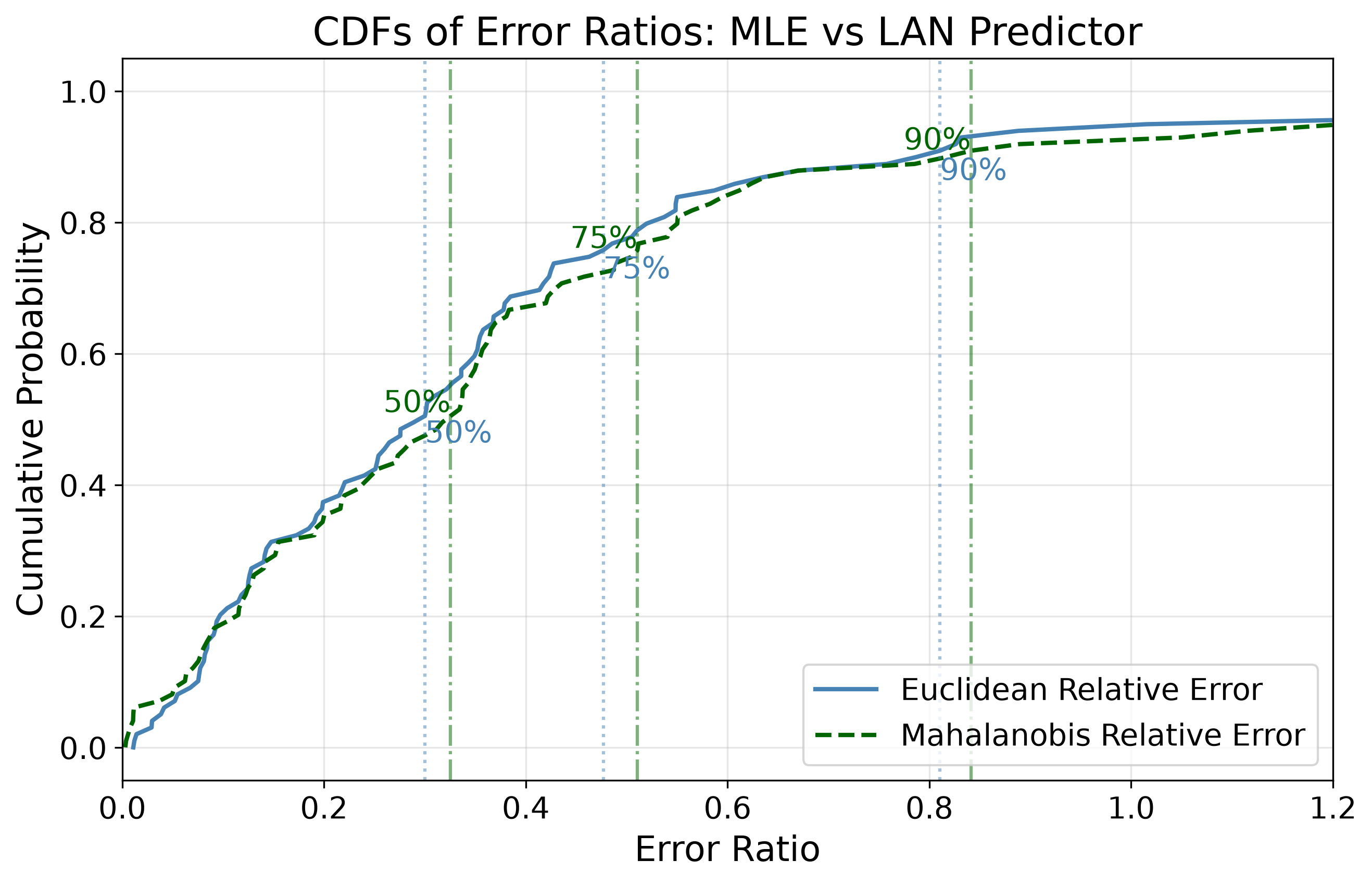}
		\caption{C.d.f. of the relative error $\|\hat \theta_t^N - \check \theta_t^N\|/\|\hat \theta_t^N - \theta_*\|$. 
		}
	\end{subfigure}
	\caption{Comparison between the estimated MLE $\hat \theta_t^N$ and the LAN predictor, for the set of parameters given in Table~\ref{tbl_v39}.}
	\label{fig_LAN_MLE_v39}
\end{figure}

\begin{table}[htbp]
	\centering
	\footnotesize
	\setlength\tabcolsep{3pt}
	\begin{tabular}{|l|ccc|}
		\hline
		&	HZ & HT & M \\
		\hline
		MLE 
		&\!\cellcolor[gray]{0.90} $3.2\,10^{-4}$ &\!0.54\! & \cellcolor[gray]{0.90} $8.5\,10^{-3}$ \\
		LAN 
		&0.60  &\!0.38\! & \cellcolor[gray]{0.90} $5.6\,10^{-3}$\\
		Delta 
		&0.58  &\!0.36\! & 0.71\! \\
		\hline
	\end{tabular}
	\caption{Test results given the p-value for the parameter set given in Table~\ref{tbl_v39}; methods: $HZ$, $HT$ and $M$ for respectively Henze-Zirkler, Hotelling's $T^2$ and Mauchly, that are tests for respectively normality, mean and sphericity; and quantities: $MLE$, $LAN$ and $Delta$ stands respectively for $\hat \theta_t^N$,  $\check \theta_t^N$ and $\wht \Delta^{N, \ts}_t/\sqrt{N}$. 
		Accepted if $p > 0.05$, Rejected (in gray) otherwise.}
	\label{tab:test_results_v39}
	
\end{table}

\end{document}